\documentclass[a4paper,12pt]{article}
\usepackage{tikz}
\usetikzlibrary{calc}
\usepackage{amssymb}
\usetikzlibrary{quotes,angles}
\usetikzlibrary{plotmarks}
\usetikzlibrary{arrows,shapes,positioning}
\usetikzlibrary{decorations.markings}
\usepackage{geometry}
\usepackage{appendix}
\tikzstyle arrowstyle=[scale=2]
\tikzstyle directed=[postaction={decorate,decoration={markings,
    mark=at position .65 with {\arrow[arrowstyle]{stealth}}}}]
\tikzstyle reverse directed=[postaction={decorate,decoration={markings,
    mark=at position .65 with {\arrowreversed[arrowstyle]{stealth};}}}]
    \tikzstyle left directed=[postaction={decorate,decoration={markings,
    mark=at position -.62 with {\arrow[arrowstyle]{stealth}}}}]

\tikzstyle left reverse directed=[postaction={decorate,decoration={markings,
    mark=at position -.62 with {\arrowreversed[arrowstyle]{stealth};}}}]
\usepackage{indentfirst}
\usepackage[plainpages=false]{hyperref}
\usepackage{amsfonts,latexsym,rawfonts,amsmath,amssymb,amsthm,mathrsfs}
\usepackage{amsmath,amssymb,amsfonts,latexsym,lscape,rawfonts}

\usepackage[all]{xy}
\usepackage{eufrak}
\usepackage{makeidx}         
\usepackage{graphicx,psfrag}
\usepackage{array,tabularx}

\usepackage{setspace}

\newtheorem{thm}{Theorem}[section]
\newtheorem{cor}[thm]{Corollary}
\newtheorem{lem}[thm]{Lemma}
\newtheorem{clm}[thm]{Claim}
\newtheorem{prop}[thm]{Proposition}

\theoremstyle{remark}
\newtheorem{rmk}[thm]{Remark}
\newtheorem{fact}[thm]{\textbf{Fact}}
\newtheorem{formula}[thm]{\textbf{Formula}}
\theoremstyle{definition}
\newtheorem{Def}[thm]{Definition}          
\newtheorem{convention}[thm]{Convention}            
\newtheorem{Notation}[thm]{Notation Convention}                  

\def \C {\mathbb C}

\title{The spectrum of an operator associated with  $G_{2}-$instantons with $1-$dimensional singularities and Hermitian Yang-Mills connections with isolated singularities} 

{\small{\author{Yuanqi Wang\thanks{Department of Mathematics, The University of Kansas, Lawrence, KS, USA.\ yqwang@ku.edu.}}}

\date{\vspace{-5ex}}

\begin{document}
\maketitle
\begin{abstract} This is the first step in an attempt at a deformation theory for $G_{2}-$instantons with $1-$dimensional conic singularities. Under a set of model data,
the linearization  yields a self-adjoint first order elliptic operator $P$ on a certain bundle over $\mathbb{S}^{5}$. As a dimension reduction, the operator $P$ also arises from  Hermitian Yang-Mills connections with isolated conic singularities on a Calabi-Yau $3$-fold. 

  Using the Quaternion structure in the Sasakian geometry of $\mathbb{S}^{5}$, we describe the set of all eigenvalues of $P$ (denoted by $Spec P$). We show that $SpecP$ consists of finitely many integers induced by certain sheaf cohomologies on $\mathbb{P}^{2}$, and  infinitely many real numbers induced by the spectrum of the rough Laplacian on the pullback endomorphism bundle over $\mathbb{S}^{5}$. The multiplicities and the form of an eigensection can be described fairly explicitly. 
  
  Using the representation theory of $SU(3)$ and the subgroup $S[U(1)\times U(2)]$,   we show an example in which $SpecP$ and the multiplicities can be completely determined.  
\end{abstract}

\tableofcontents
\section{Introduction}
\subsection{Overview}
  $G_{2}-$instantons (and projective $G_{2}-$instantons) are the analogue of both flat connections in dimension $3$, and anti self-dual connections in dimension $4$. Understanding the  singularities of (projective) $G_{2}-$instantons plays an important role in the programs  proposed by  Donaldson-Thomas \cite{DonaldsonThomas} and Donaldson-Segal \cite{DonaldsonSegal} on higher dimensional gauge theory. 

In conjunction with Jacob-Walpuski \cite{JW}, to construct   (projective) $G_{2}-$instantons with $1-$dimensional singularities on twisted connected-sum $G_{2}-$manifolds via gluing, an important step is a deformation theory built upon a Fredholm theory for the linearized operator. In \cite{WangyuanqiJGP},  it is  shown that  a Fredholm theory, and consequently a deformation theory, always exist for instantons with isolated singularities. However, the situation of instantons with $1-$dimensional singularities is expected to be drastically different. 

On a $G_{2}-$instanton with conic singularities along a circle, in the model setting,  the linearized operator yields a self-adjoint elliptic operator $P$ on a certain bundle over $\mathbb{S}^{5}$ (see Lemma \ref{lem formula of the model dirac deformation operator} below). The set of all eigenvalues  of $P$,  denoted by $SpecP$, plays a crucial role in the construction of a deformation theory. It determines the indicial roots of the linearization.

 The purpose of this paper is to describe $SpecP$ and address the multiplicities. 
\subsection{Background of the operator $P$}

Before stating the main results, we briefly recall some background on $G_{2}-$instantons with $1-$dimensional singularities, and explain where the operator $P$ comes from. 
 More details can be found in Section \ref{sect Sasakian}. 
 
 \subsubsection{General Background}
 Let $M^{7}$ be a $7-$dimensional manifold with a $G_{2}-$structure $\phi$, we denote the co-associative $4-$form by $\psi$. Given a smooth Hermitian vector bundle $E$ over $M^{7}$, let $adE$ denote the bundle of  skew-adjoint endomorphisms. 
 Let $A$  be a unitary connection, and $\sigma$ be a section of $adE$.  The pair $(A,\sigma)$ is called a $G_{2}-$monopole if the following equation is satisfied. 
\begin{equation}\label{equ instanton equation}
\star(F_{A}\wedge \psi)+d_{A}\sigma=0.
\end{equation} 
When $d_{A}\sigma=0$, we say that $A$ is a $G_{2}-$instanton. 

Let $\Omega^{k}_{adE}$ denote the bundle of  $adE-$valued $k-$forms. With gauge fixing, the linearization of   \eqref{equ instanton equation}
 in $A$ is an elliptic operator $L_{A,\phi}$ which maps $C^{\infty}[M^{7}, \Omega^{0}_{adE}\oplus \Omega^{1}_{adE}]$ to itself. Namely, 
\begin{equation}\label{equ introduction formula for deformation operator}L_{A,\phi}[\begin{array}{c}
\sigma   \\
  a   
\end{array}]=[\begin{array}{c}
d_{A}^{\star}a   \\
  d_{A}\sigma+\star(d_{A}a\wedge \psi)  
\end{array}], \end{equation}
 where $\sigma\in C^{\infty}[M^{7}, \Omega^{0}_{adE}]$, and  $a\in C^{\infty}[M^{7}, \Omega^{1}_{adE}]$. $L_{A,\phi}$ only depends on the projective connection induced by $A$. The same applies below to the model linearized operator $L_{A_{O},\phi_{\mathbb{C}^{3}\times \mathbb{S}^{1}}}$, and also to the operator $P$ of interest.
\subsubsection{A holomorphic Hermitian triple and the  associated data setting}
We now introduce the simple manifolds and maps required. 
\begin{Def} Using the dimensions of the domain and range manifold as subscripts,  we consider the following standard projection maps. 
\begin{equation}\label{equ tabular choice of solution 2nd order operator m neq 0}  \begin{tabular}{|p{3cm}|p{4.5cm}|}
  \hline
 Map   &  Domain and Range\\ \hline
 $\pi_{7,6}$ &  $(\mathbb{C}^{3}\setminus O)\times \mathbb{R}\rightarrow \mathbb{C}^{3}\setminus O$\\ \hline
 $\pi_{6,5}$ & $\mathbb{C}^{3}\setminus O\rightarrow \mathbb{S}^{5}$\\ \hline
 $\pi_{5,4}$ &  $\mathbb{S}^{5}\rightarrow \mathbb{P}^{2}$\\ \hline
$\pi_{6,4}\triangleq \pi_{5,4}\cdot \pi_{6,5}$ & $\mathbb{C}^{3}\setminus O \rightarrow  \mathbb{P}^{2}$ \\   \hline
$\pi_{7,4}\triangleq \pi_{6,4}\cdot \pi_{7,6}$ & $(\mathbb{C}^{3}\setminus O)\times \mathbb{S}^{1} \rightarrow  \mathbb{P}^{2}$ \\   \hline
$\pi_{7,5}\triangleq \pi_{6,5}\cdot \pi_{7,6}$ & $(\mathbb{C}^{3}\setminus O)\times \mathbb{S}^{1} \rightarrow  \mathbb{S}^{5}$ \\   \hline
\end{tabular}
 \renewcommand\arraystretch{1.5}
  \end{equation}
\end{Def}

 \begin{Def}\label{Def standard Hermitian metric on Obeta} The standard Hermitian metric on the universal bundle $O(-1)\rightarrow \mathbb{P}^{2}$ is $|Z_{0}|^{2}+|Z_{1}|^{2}+|Z_{2}|^{2}$. We denote it by $h_{O(-1)}$ (see \cite{GH}).  For any integer $l\neq 0$,  this metric induces uniquely  a Hermitian metric $h_{O(l)}$ on $O(l)$. We call the Chern connection of $h_{O(l)}$ \textit{the standard connection}.\end{Def}

We introduce the following terms. 
\begin{Def}\label{Def setting of 3 keywords}

 \textbf{1. Holomorphic Hermitian triple:} A triple $(E,h,A_{O})$ consists of a holomorphic vector bundle $E\rightarrow \mathbb{P}^{2}$, a Hermitian metric $h$ on $E$, and the Chern connection $A_{O}$ 
  is called a \textit{holomorphic Hermitian triple} on $\mathbb{P}^{2}$. 

 
  \textbf{2. Hermitian Yang-Mills triple:} A holomorphic Hermitian triple on $\mathbb{P}^{2}$ is called \textit{Hermitian Yang-Mills} if the connection $A_{O}$ is Hermitian Yang-Mills. 
  
  \textbf{3. Irreducible Hermitian Yang-Mills triple:}   A  Hermitian Yang-Mills triple on $\mathbb{P}^{2}$ is called \textit{irreducible Hermitian Yang-Mills} if $A_{O}$ is irreducible and $rankE\geq 2$. 
    
  A holomorphic Hermitian (Hermitian Yang-Mills) triple on $\mathbb{P}^{n}$ and the standard connection on $O(l)\rightarrow \mathbb{P}^{n}$ are defined in the same manner. Nevertheless, except for the K\"ahler identity in Lemma \ref{lem Kahler identity} below, we only consider such a triple on $\mathbb{P}^{2}$.

The bundle $(End_{0}E)(l)$  is  called the twisted traceless endomorphism bundle,  and $(End E)(l)$ is called the twisted endomorphism bundle. We call  the tensor product of the standard connection on $O(l)$ and $A_{O}$ the \textit{twisted connection}. We also call the tensor product of the standard metric on $O(l)$ and $h$ the twisted metric. A section of $(End E)(l)$ is called a \textit{twisted endomorphism}.  The same applies to the pullbacks. 

     \textbf{4. The associated  data setting:} Given a holomorphic Hermitian triple $(E,h,A_{O})$ on $\mathbb{P}^{2}$. Throughout, the operator $P$ and the other bundle rough Laplacians are defined by the following data.   \begin{itemize}\item On $\mathbb{S}^{5},\ \mathbb{C}^{3}\setminus O$, and $(\mathbb{C}^{3}\setminus O)\times \mathbb{S}^{1}$,  we consider the pullback of $A_{O}$ and  the pullback Hermitian metric on the (pullback) endomorphism bundles. 
     
   On a twisted endomorphism bundle over $\mathbb{P}^{2}$,  we consider the \textit{twisted connection} and the twisted metric.

\item Let \begin{equation}\label{equ def eta}\eta\triangleq d^{c}\log r\triangleq \sqrt{-1}(\bar{\partial}-\partial)\log r\end{equation} be the contact form on $\mathbb{S}^{5}$. Throughout, we consider the Fubini-Study metric $\frac{d\eta}{2}$ on $\mathbb{P}^{2}$, the standard round metric on $\mathbb{S}^{5}$, and the Euclidean metric on $\mathbb{C}^{3}\times \mathbb{S}^{1}$ (and $\mathbb{C}^{3}$).
\end{itemize}
\end{Def}

\subsubsection{The linearized operator under the model data}

Suppose $A_{O}$ is not flat on $\mathbb{P}^{2}$, the pullback connection  on $(\mathbb{C}^{3}\setminus O)\times \mathbb{S}^{1}$
 has conic singularity along the circle $O\times \mathbb{S}^{1}$. This is the prototype of what we are interested in. The model linear problem for $G_{2}-$instantons with conic singularities along a circle is as follows.  
 
On $\mathbb{C}^{3}$, let $\omega_{\mathbb{C}^{3}}=\frac{\sqrt{-1}}{2}(dZ_{0}d\bar{Z}_{0}+dZ_{1} d\bar{Z}_{1}+dZ_{2} d\bar{Z}_{2})$ be the standard K\"ahler form, and let $\Omega_{\mathbb{C}^{3}}=dZ_{0}dZ_{1}dZ_{2}$ be the standard holomorphic volume form.  The standard $G_{2}-$structure on $\mathbb{C}^{3}\times \mathbb{S}^{1}$ is defined by $$\phi_{\mathbb{C}^{3}\times \mathbb{S}^{1}}\triangleq ds\wedge \omega_{\mathbb{C}^{3}}+Re\Omega_{\mathbb{C}^{3}}.$$
The standard co-associative $4-$form is  $\psi_{\mathbb{C}^{3}\times \mathbb{S}^{1}}\triangleq \frac{\omega^{2}_{\mathbb{C}^{3}}}{2}-ds\wedge Im\Omega_{\mathbb{C}^{3}}.$

 Given a holomorphic Hermitian tripe $(E,h,A_{O})$ on $\mathbb{P}^{2}$, the model linearized operator is defined as follows. 
 \begin{equation}\label{equ  formula for model deformation operator}L_{A_{O},\phi_{\mathbb{C}^{3}\times \mathbb{S}^{1}}}[\begin{array}{c}
\sigma   \\
  a_{\mathbb{C}^{3}\times \mathbb{S}^{1}}
\end{array}]=[\begin{array}{c}
d_{A_{O},\mathbb{C}^{3}\times \mathbb{S}^{1}}^{\star_{\mathbb{C}^{3}\times \mathbb{S}^{1}}}a_{\mathbb{C}^{3}\times \mathbb{S}^{1}}   \\
  d_{A_{O},\mathbb{C}^{3}\times \mathbb{S}^{1}}\sigma+\star_{\mathbb{C}^{3}\times \mathbb{S}^{1}}(d_{A_{O},\mathbb{C}^{3}\times \mathbb{S}^{1}}a_{\mathbb{C}^{3}\times \mathbb{S}^{1}}\wedge \psi_{\mathbb{C}^{3}\times \mathbb{S}^{1}})  
\end{array}], \end{equation}
where $\sigma\in C^{\infty}[(\mathbb{C}^{3}\setminus O)\times \mathbb{S}^{1},\Omega^{0}_{ \pi^{\star}_{7,4}(adE)}]$ is a section of the  adjoint bundle, and\\  $a_{\mathbb{C}^{3}\times \mathbb{S}^{1}}\in C^{\infty}[(\mathbb{C}^{3}\setminus O)\times \mathbb{S}^{1}, \Omega^{1}_{ \pi^{\star}_{7,4}adE}]$ is an adjoint bundle-valued $1-$form.

A section $a_{\C^{3}\times \mathbb{S}^{1}}$ of $\Omega^{1}_{\pi_{7,4}^{\star}(adE)}\rightarrow (\C^{3}\setminus O)\times \mathbb{S}^{1}$ can be split into
  \begin{equation}\label{equ 1st splitting of the 1form}a_{\mathbb{C}^{3}\times \mathbb{S}^{1}}=\underline{a}_{s}ds+a_{\mathbb{C}^{3}},\end{equation}where $\underline{a}_{s}$ is a section of $\pi^{\star}_{7,4}(adE)$, and $a_{\mathbb{C}^{3}}$ is a section of  $\pi^{\star}_{7,6}(\Omega_{\mathbb{C}^{3}}^{1})\otimes \pi_{7,4}^{\star}(adE)$ i.e. an adjoint bundle-valued $1-$form without $ds-$component. 
  \subsubsection{The fine splitting}
 We employ the Sasakian geometry of $\mathbb{S}^{5}$, and aim at a more meticulous splitting with respect to the transverse K\"ahler structure.

 We denote the contact distribution $Ker\eta$ by $D$.  Let $\xi$ denote the standard Reed vector field on $\mathbb{S}^{5}$ (tangential to orbit of the $U(1)-$multiplications). The orthogonal complement of the contact form $\eta$ is denoted by $D^{\star}$, we call it the contact co-distribution. We call a form $\theta$  semi-basic if $\xi\lrcorner \theta=0$. A section of $D^{\star}$ is precisely a semi-basic $1-$form. 

We define the finer splitting of a section in the domain of the linearized operator. In view of formula \eqref{equ  formula for model deformation operator} and \eqref{equ 1st splitting of the 1form}, let $u\triangleq r\sigma,\ a_{s}\triangleq r\underline{a}_{s}$, we find 
  \begin{equation}\label{equ fine splitting aC3}  a_{\mathbb{C}^{3}}= a_{r}\frac{dr}{r}+(a_{\eta})\eta+a_{0},\end{equation} where
  \begin{itemize}\item $a_{r}$ and $a_{\eta}$ are sections of $\pi^{\star}_{7,4}(adE)$,    \item  $a_{0}$ is an adjoint bundle-valued semi-basic $1-$form i.e.  a $1-$form with no $ds$, $dr$, or $\eta-$component. 
  \end{itemize}
For further calculation, we let  $a_{\mathbb{S}^{5}}\triangleq a_{\eta}(\eta)+a_{0}$. Fixing $r$ and $s$, both $a_{\mathbb{S}^{5}}$ and $a_{0}$ are forms on $\mathbb{S}^{5}$.  We then obtain the splitting of the domain bundle of the linearized operator.   \begin{eqnarray}\nonumber& & \Omega^{0}_{\pi_{7,4}^{\star}(adE)}\oplus \Omega^{1}_{\pi_{7,4}^{\star}(adE)}=[\pi_{7,4}^{\star}(adE)]^{\oplus 4}\oplus [\pi_{7,5}^{\star}(D^{\star})\otimes  \pi_{7,4}^{\star}(adE)]:\\& & [\begin{array}{c}
\sigma   \\
 a 
\end{array}]=\left[\begin{array}{ccccc}\frac{1}{r} &0 &0 &0 &0 \\ 0 &\frac{ds}{r} &\frac{dr}{r}  &\eta &1 \end{array}\right]\left[\begin{array}{c}u \\ a_{s} \\  a_{r} \\ a_{\eta} \\ a_{0}\end{array}\right].\label{equ  5 element basis}
  \end{eqnarray}
  
  \begin{Def}\label{Def Dom bundle}
Henceforth, let $Dom_{\mathbb{S}^{5}}$ denote $[\pi_{5,4}^{\star}(adE)]^{\oplus 4}\oplus [D^{\star}\otimes  (\pi_{5,4}^{\star}adE)]$ as well as the space of smooth sections of the same bundle on $\mathbb{S}^{5}$. Similarly, on $(\mathbb{C}^{3} \setminus O)\times \mathbb{S}^{1}$, let $Dom_{7}$ denote the pullback $\pi^{\star}_{7,5}Dom_{\mathbb{S}^{5}}$ 
as well as the space of smooth sections. They are the ``domain"  of the operator $P$, and also of the linearized operator in \eqref{equ  formula for model deformation operator}. 
\end{Def}
\subsubsection{Introducing the operator $P$}

Now we can abbreviate Lemma \ref{lem formula of the model dirac deformation operator} below: given a holomorphic Hermitian triple on $\mathbb{P}^{2}$,   there exists isometries $I$, $K$ on the  bundle  $Dom_{7}$,  and a self-adjoint elliptic operator $P$ on $Dom_{\mathbb{S}^{5}}$, such that under the basis in \eqref{equ  5 element basis}, the following splitting holds. 
\begin{equation}\label{equ def of P}L_{A_{O},\phi_{\mathbb{C}^{3}\times \mathbb{S}^{1}}}=\frac{\partial}{\partial s}\circ I+K\circ (\frac{\partial}{\partial r}-\frac{P}{r}).
\end{equation}

As a dimension reduction, in complex dimension $3$, the operator $K\circ (\frac{\partial}{\partial r}-\frac{P}{r})$ is the model linearized operator with gauge fixing for  a Hermitian Yang-Mills monopole with isolated conic singularity (see Appendix \ref{Appendix digression} below). Thus, on a Calabi-Yau $3-$fold, our results  allow us to calculate the index of the linearized operator of Hermitian Yang-Mills monopoles with isolated conic singularities. 

\subsection{The main theorem}
Throughout,  $SpecP$ does not count the  multiplicity of an eigenvalue (cf. the ``$Spec^{mul}$" in Definition \ref{Def spec mul} below that counts multiplicity). This means an eigenvalue appears in  $SpecP$ exactly once, thus it is  a subset of $\mathbb{R}$.  The same applies to $Spec (\nabla^{\star}\nabla|_{\mathbb{S}^{5}})$ below. We address the multiplicities separately. 

\begin{Def}\label{Def eigenspace} (Eigenspaces) Given a  holomorphic Hermitian triple on $\mathbb{P}^{2}$ and in the associated data setting, for any real number $\mu$, let $$\mathbb{E}_{\mu}P\triangleq Ker(P-\mu Id).$$
Resultantly, $\mu\in Spec P$ if and only if  $\mathbb{E}_{\mu}P\neq \{0\}$. 

On the adjoint bundle $\pi^{\star}_{5,4}(adE)\rightarrow \mathbb{S}^{5}$,  let $\nabla^{\star}\nabla$ and $\nabla^{\star}\nabla|_{\mathbb{S}^{5}}$ both abbreviate the  rough Laplacian defined by the pullback connection $A_{O}$ and the standard round metric on $\mathbb{S}^{5}$. The set of all its eigenvalues is denoted by $Spec (\nabla^{\star}\nabla|_{\mathbb{S}^{5}})$. Let $$\mathbb{E}_{\lambda}(\nabla^{\star}\nabla|_{\mathbb{S}^{5}})\triangleq Ker [(\nabla^{\star}\nabla|_{\mathbb{S}^{5}})-\lambda Id].$$ 

We still call $\mathbb{E}_{\mu}P$ ($\mathbb{E}_{\lambda}(\nabla^{\star}\nabla|_{\mathbb{S}^{5}})$) the \textit{eigenspace} with respect to $\mu$ ($\lambda$), though $\mu$ is not necessarily an eigenvalue.  $\mu$ ($\lambda$) is an eigenvalue if and only if $\mathbb{E}_{\mu}P\neq \{0\}$ ($\mathbb{E}_{\lambda}(\nabla^{\star}\nabla|_{\mathbb{S}^{5}})\neq \{0\}$). This convention turns out to work well. 
\end{Def}
\begin{Notation}\label{Notation = I} The equal sign ``$=$" between two  vector spaces (bundles) always means at least a real isomorphism. The notation ``$dim$" means the real dimension. 
\end{Notation}
Before stating the main theorem, we need two more notions.
\begin{Def}\label{Def two parts of spec} Given a holomorphic Hermitian triple on $\mathbb{P}^{2}$, we define the following two subsets of $\mathbb{R}$. 
\begin{itemize}\item $S_{\nabla^{\star}\nabla}\triangleq \{\mu|(\mu^{2}+2\mu-3)\in Spec(\nabla^{\star}\nabla |_{\mathbb{S}^{5}})\}\cup \{\mu|\mu^{2}+4\mu\in Spec(\nabla^{\star}\nabla |_{\mathbb{S}^{5}})\}$ i.e. 
$$S_{\nabla^{\star}\nabla}\triangleq \cup_{\lambda\in Spec\nabla^{\star}\nabla|_{\mathbb{S}^{5}}}[ \{-1+\sqrt{4+\lambda}\}\cup \{-1-\sqrt{4+\lambda}\}\cup \{-2+\sqrt{4+\lambda}\}\cup \{-2-\sqrt{4+\lambda}\} ].$$
\item  $S_{coh}\triangleq \{l|l\ \textrm{is an integer and}\ H^{1}[\mathbb{P}^{2},(End E)(l)]\neq 0\}$. \end{itemize}
\end{Def}
Apparently, the set $S_{\nabla^{\star}\nabla}$ is induced by the spectrum of the rough Laplacian, 
 the set $S_{coh}$ consists of integers and is given by the sheaf cohomologies. 
Intuitively speaking, under natural assumptions,  our main theorem  ``decomposes" $Spec P$ into the union of these two sets. The multiplicities can be described.

\begin{thm}\label{Thm 1}
Let $(E,h,A_{O})$ be an irreducible Hermitian Yang-Mills triple on $\mathbb{P}^{2}$ with the associated data setting. Let $P$ be the first-order self-adjoint elliptic operator defined  in \eqref{equ def of P} (and Lemma \ref{lem formula of the model dirac deformation operator} below). 

$\mathbb{I}\ \ (Spectrum)$. The following spectral decomposition holds for $P$. $$SpecP=S_{\nabla^{\star}\nabla}\cup S_{coh}.$$
Consequently, the set $(SpecP)\cap (-3,0)$ contains and only contains the two numbers $-1,\ -2$.

$\mathbb{II}\ \ (Multiplicities).$ 
\begin{enumerate}\item In view of Definition \ref{Def eigenspace} of the eigenspaces and Remark \ref{rmk EP is complex} below on the complex structure of each $\mathbb{E}_{\mu}P$, the following complex isomorphisms hold. $$\mathbb{E}_{-1}P=H^{1}[\mathbb{P}^{2}, (End E)(-1)],\  \mathbb{E}_{-2}P=H^{1}[\mathbb{P}^{2}, (End E)(-2)].$$

\item If $\mu\in SpecP$ and $\mu$ is not an integer, the following real isomorphism holds.  $$\mathbb{E}_{\mu}P=[\mathbb{E}_{\mu^{2}+2\mu-3}(\nabla^{\star}\nabla|_{\mathbb{S}^{5}})]^{\oplus 2}\oplus [\mathbb{E}_{\mu^{2}+4\mu}(\nabla^{\star}\nabla|_{\mathbb{S}^{5}})]^{\oplus 2}.$$

\item If $\mu\in SpecP$, $\mu$ is an integer, but $\mu \neq -1$ or $-2$, then
 \begin{eqnarray*} 
dim \mathbb{E}_{\mu}P& =& 2dim \mathbb{E}_{\mu^{2}+2\mu-3}(\nabla^{\star}\nabla|_{\mathbb{S}^{5}})+2dim \mathbb{E}_{\mu^{2}+4\mu}(\nabla^{\star}\nabla|_{\mathbb{S}^{5}})+ 2h^{1}[\mathbb{P}^{2}, (End E)(\mu)]
\\& &-2h^{0}[\mathbb{P}^{2}, (End_{0}E)(\mu)]-2h^{0}[\mathbb{P}^{2}, (End_{0}E)(-\mu-3)].\end{eqnarray*}
\end{enumerate}
In particular, $dim KerP=2h^{1}[\mathbb{P}^{2}, End E]$. 
\end{thm}

 By the Enrique-Severi-Zariski Lemma, $S_{coh}$ is a finite set. Theorem \ref{Thm 1}.$\mathbb{I}$ says that up to this finite set, $Spec P$ is induced by $Spec(\nabla^{\star}\nabla |_{\mathbb{S}^{5}})$. 
 
 We can reduce $Spec(\nabla^{\star}\nabla |_{\mathbb{S}^{5}})$ to the spectrum of the rough Laplacians  on the twisted traceless endomorphism bundles on $\mathbb{P}^{2}$ (see Formula \ref{formula laplace on S5 vs laplace on CP2}).  The spectrum of such  bundle rough Laplacians is ``easier" to understand, in a sense like Theorem \ref{Thm spec of rough laplacian P2} below.

The splitting $SpecP=S_{\nabla^{\star}\nabla}\cup S_{coh}$ corresponds to that the domain $L^{2}(\mathbb{S}^{5},Dom_{\mathbb{S}^{5}})$ is the direct sum of a finite dimensional subspace given by the cohomologies and the infinite-dimensional  orthogonal complement. Please see Definition \ref{Def Vcoh} below for detail.

Theorem \ref{Thm 1}.$\mathbb{II}.1$ says that the multiplicity of $-1$ and $-2$ are both equal to $$2h^{1}[\mathbb{P}^{2}, (End E)(-1)]=2h^{1}[\mathbb{P}^{2}, (End E)(-2)]),$$ which are the real dimension of the cohomologies.

The identity $dim KerP=2h^{1}[\mathbb{P}^{2}, End E]$ in the end of Theorem \ref{Thm 1} says that, as  vector spaces over $\mathbb{R}$, the kernel of $P$ is isomorphic to the deformations space of the Hermitian Yang-Mills connection on $\mathbb{P}^{2}$.

Counting multiplicities,  the eigenvalues of  $P$ is symmetric with respect to $-\frac{3}{2}$ (see the Kodaira-Serre duality for eigenspaces in Lemma \ref{lem JH is serre duality}). Therefore, the eta invariant of $P$ can be calculated (for example, see \cite[Proposition 5.4]{WangyuanqiJFA}).

The binomials $\mu^{2}+2\mu-3$ and $\mu^{2}+4\mu$ are given by the  Bochner formulas (see Lemma \ref{lem Bochner} below). 
The Hermitian Yang-Mills condition is required for these  formulas. The irreducible Hermitian Yang-Mills condition is required for the Chern number inequality (see \eqref{equ 1 lem h1} below) that implies $-1$ and $-2$ must be eigenvalues.

\begin{rmk}The eigensections of $P$ admit an explicit form \eqref{equ 0 lem proj formula} in terms of the eigensections of $\nabla^{\star}\nabla |_{\mathbb{S}^{5}}$. Please see Remark \ref{rmk form of eigensection} and Lemma \ref{lem proj formula}  below.   \end{rmk}
\subsection{An example}
As an  application, if $E=T^{\prime}\mathbb{P}^{2}(k)$ is a twisted holomorphic tangent bundle of $\mathbb{P}^{2}$, we can completely determine $SpecP$. We call the Levi-Civita connection of the Fubini-Study metric on $T^{\prime}\mathbb{P}^{2}$ \textit{the Fubini-Study connection}, and denote it by $\nabla^{FS}$. 

\subsubsection{Spectrum of the rough Laplacian}
A step for the above goal is to determine the spectrum of the rough Laplacian. On the twisted endomorphism bundles $(EndT^{\prime}\mathbb{P}^{2})(l)$, the tensor product of the Fubini-Study connection (metric) and the standard connection (metric) on $O(l)$ is called  \textit{the twisted Fubini-Study connection (metric)}, respectively. The same notions also apply to $T^{\prime}\mathbb{P}^{2}(k)$.

\begin{thm}\label{Thm spec of rough laplacian P2} In the setting of Theorem \ref{Thm 1}, let the irreducible Hermitian Yang-Mills triple be a twisted holomorphic tangent bundle $T^{\prime}\mathbb{P}^{2}(k)$  with the twisted Fubini-Study metric and connection. 

 For any integer $l$, let the rough Laplacian $\nabla^{\star}\nabla|_{(End_{0}T^{\prime}\mathbb{P}^{2})(l)\rightarrow \mathbb{P}^{2}}$ be defined by the twisted Fubini-Study connection and the Fubini-Study metric $\frac{d\eta}{2}$. Then
 \begin{eqnarray}\nonumber  Spec\nabla^{\star}\nabla |_{(End_{0}T^{\prime}\mathbb{P}^{2})(l)\rightarrow \mathbb{P}^{2}}
 &=& \{\frac{4}{3}(a^{2}+b^{2}+ab+3a+3b)-\frac{4}{3}l^{2}-8\ |\  a,b\in \mathbb{Z};  \ a , b \geq 0;
\\& &\max(3-a-2b,b-a-3) \leq l\leq \min(2a+b-3,3+b-a)\}.\nonumber
\\& &\label{equ spec Laplacian P2}
\end{eqnarray}
Consequently, in the associated data setting, 

  \begin{eqnarray}\label{equ proof of Thm rough laplacian}\label{equ Cor 1}Spec\nabla^{\star}\nabla|_{\mathbb{S}^{5}}&=& \{\frac{4}{3}(a^{2}+b^{2}+ab+3a+3b)-\frac{l^{2}}{3}-8\ |\  a,b, l \in \mathbb{Z};  \ a , b \geq 0;\\& &\max(3-a-2b,b-a-3) \leq l\leq \min(2a+b-3,3+b-a)\}.\nonumber
\end{eqnarray}
The multiplicity of each eigenvalue in \eqref{equ spec Laplacian P2} and \eqref{equ Cor 1} is determined by Proposition \ref{prop multiplicity TP2}.
 \end{thm}
Henceforth, we abbreviate $\nabla^{\star}\nabla |_{(End_{0}T^{\prime}\mathbb{P}^{2})(l)\rightarrow \mathbb{P}^{2}}$ to  $\nabla^{\star}\nabla |_{(End_{0}T^{\prime}\mathbb{P}^{2})(l)}$, and more generally, we  abbreviate $\nabla^{\star}\nabla |_{(End_{0}E)(l)\rightarrow \mathbb{P}^{2}}$ to  $\nabla^{\star}\nabla |_{(End_{0}E)(l)}$.

In view of the spectral reduction in Formula \ref{formula laplace on S5 vs laplace on CP2} below,  it suffices to prove  \eqref{equ spec Laplacian P2}.  Because $T^{\prime}\mathbb{P}^{2}$ and $O(l)$ are both $SU(3)-$homogeneous, we prove it by Peter-Weyl formulation. \begin{itemize}\item The numbers $\frac{4}{3}(a^{2}+b^{2}+ab+3a+3b)$ and $-8$ therein arise from the Casimir operators of $su(3)$ and $su(2)$  on certain irreducible representations respectively. The number $-\frac{4l^{2}}{3}$ arises from the action of a certain element in the Cartan sub-algebra of $su(3)$. Please see \eqref{Cas e5 rhol} and Formula \ref{formula of Cas K},  \ref{equ Cas su3} below. 

The desired equation \eqref{equ spec Laplacian P2} means that these $3$ terms are all the contributions from the representation theoretic quantities. 
\item The condition on $a,\ b$ in \eqref{equ spec Laplacian P2} is the equivalence condition of that a certain irreducible $SU(3)-$representation appears as a summand in a certain infinite dimensional representation (see Fact \ref{fact IT} below). 
\end{itemize}

It is not obvious to the author how to directly calculate $Spec\nabla^{\star}\nabla|_{\mathbb{S}^{5}}\triangleq Spec\nabla^{\star}\nabla|_{\pi^{\star}_{5,4}(adE)\rightarrow \mathbb{S}^{5}}$ on $\mathbb{S}^{5}$.
\subsubsection{The example of $SpecP$}
Theorem \ref{Thm 1}, \ref{Thm spec of rough laplacian P2}, and Proposition \ref{prop multiplicity TP2} below imply the following. 
\begin{cor}\label{Cor 1}In the  setting of Theorem \ref{Thm 1} and \ref{Thm spec of rough laplacian P2}, still let the irreducible Hermitian Yang-Mills triple be a twisted holomorphic tangent bundle $T^{\prime}\mathbb{P}^{2}(k)$  with the twisted Fubini-Study metric and connection. Then $$S_{coh}=\{-1\}\cup \{-2\}.$$ 


Consequently, let $S_{\nabla^{\star}\nabla}$ be defined by Theorem \ref{Thm 1}.$\mathbb{I}$ and  \eqref{equ proof of Thm rough laplacian}, the following splitting holds.   $$SpecP=S_{\nabla^{\star}\nabla}\cup \{-1\}\cup \{-2\}.$$
 The first row of the following table contains all the eigenvalues of $P$ in the closed interval $[-4,1]$. The second row addresses  the multiplicity of each.  
  \begin{equation}\label{equ tabular eigenvalue and multiplicity}  \begin{tabular}{|p{2cm}|p{1.5cm}|p{1.8cm}||p{1.5cm}|p{1.5cm}|p{1.5cm}|p{1.5cm}|}
  \hline
 eigenvalue of $P$ &  $-4$ & $-2\sqrt{2}-1$ & -2 &-1 &$2\sqrt{2}-2$ &1   \\   \hline
  multiplicity &  12  &16 & 6 & 6&  16 &12  \\   \hline
  \end{tabular}
 \renewcommand\arraystretch{1.5}
  \end{equation}

  The multiplicities of the other eigenvalues are also determined by Proposition \ref{prop multiplicity TP2} and Theorem \ref{Thm 1}.$\mathbb{II}$.
\end{cor}


\subsection{Sketch of the proof of Theorem \ref{Thm 1} and Corollary \ref{Cor 1}. \label{sect sketch proof}}
 The ideas in proving Theorem \ref{Thm 1} can be partially described as follows. 
 
Our Sasaki-Quaternion structure on $\mathbb{S}^{5}$ is a special case of a Sasaki-Einstein $SU(2)-$structure mentioned in \cite[4.1]{FHN}. We employ more explicit information to prove the fine formula  for $P$ (Lemma \ref{lem formula of the model dirac deformation operator}). 

  Then we seek for Bochner formulas (see Lemma \ref{lem Bochner} below). The observation is that  the first and second row of $P^{2}+2P$ are ``autonomous" i.e. they are independent of the unknowns corresponding to the other rows. The same is true for the third and fourth row of $P^{2}+4P$. This  implies that if an eigensection does not correspond to an eigenvalue given by $Spec\nabla^{\star}\nabla|_{\mathbb{S}^{5}}$, the first $4-$component of the eigensection (regarding the decomposition in \eqref{equ  5 element basis}) must all be $0$, then it can be identified with a certain sheaf cohomology class. Please see  Theorem \ref{Thm eigenvalue Sasakian} below for more detail.


Corollary \ref{Cor 1} is the direct consequence of Theorem \ref{Thm 1},  \ref{Thm spec of rough laplacian P2}, Proposition \ref{prop multiplicity TP2},  and a little bit of algebraic geometry (Lemma \ref{lem h1 EndTP2}).\\


The paper is organized as follows.  In Section \ref{sect Sasakian}, we fully employ the Sasaki-Quaternion  structure of $(\mathbb{C}^{3}\setminus O)\rightarrow \mathbb{S}^{5}$ to prove the fine formula (Lemma \ref{lem formula of the model dirac deformation operator}) for  the operator  $P$. We prove Theorem \ref{Thm 1} in Section \ref{sect proof of Thm 1}. In Section \ref{sect proof of Cor 1}, we prove Theorem \ref{Thm spec of rough laplacian P2} and Proposition \ref{prop multiplicity TP2}  by representation theoretic methods, then combine them with Theorem \ref{Thm 1} to prove Corollary \ref{Cor 1}. The Appendix collects some results obtained by routine calculations.\\

\textbf{Acknowledgements:} The author would like to thank Simon Donaldson for his guidance, encouragement, and helpful discussions on this project. The author would like to thank Xinyi Yuan for helpful discussions.  

Part of the work is done in Simons Center for Geometry and Physics, Stony Brook University, NY, USA, under the support of Simons Collaboration on Special Holonomy in Geometry, Analysis, and Physics. The author would like to thank SCGP for their hosting and the excellent research environment.  

\section{A little bit of Sasakian geometry and the fine formula for the linearized operator\label{sect Sasakian}}
\subsection{The Sasakian geometry of $\mathbb{S}^{5}$ \label{sect SQ str}}
\subsubsection{General conventions}
We recall some general Riemannian geometry setting under which the required identities are established. 

 Given a Riemannian manifold $(M,g)$, for any tangent vector $v_{x}\in T_{x}M$, let $v_{x}^{\sharp_{M}}$ denote the metric dual form in $T^{\star}_{x}M$. Conversely, for any $1-$form $\theta_{x}\in T^{\star}_{x}M$, let $\theta_{x,\sharp_{M}}$ denote the metric dual vector in $T_{x}M$.  Given a $1-$form $h\in T^{\star}_{x}M$ and a $p-$form $\Omega\in \wedge^{p}T^{\star}_{x}M$, $p\geq 1$, we define the contraction by 
 \begin{equation}\label{equ def contraction between forms} h\lrcorner_{g}\Omega\triangleq h_{\sharp_{M}}\lrcorner \Omega.
 \end{equation}
 
 
 The superscript $\mathbb{C}$ on a (real) vector-bundle (vector space)  means the complexification.   Associated with the Riemannian metric,  in the below, the tensor operators   $\lrcorner$  (contraction), $^\sharp$ (pulling up), $_\sharp$ (pushing down), $^{\parallel_{X}}$ (projection onto a real vector field $X$), $P_{X^{\perp}}$ (projection onto the orthogonal complement of a real vector field $X$), and the star operators $\star_{0},\ \star_{\mathbb{P}^{2}}$  etc are all extended   $\mathbb{C}-$linearly onto the complexified tangent and co-tangent bundle. 
 

 \subsubsection{Sasakian coordinate system\label{sect Sasakian coordinate}}

The purpose of this section is to define the Sasakian coordinate.

We denote the contact distribution $Ker\eta$ by $D$.  Let $\xi$ denote the standard Reed vector field on $\mathbb{S}^{5}$ (tangential to orbit of the $U(1)-$multiplications). 
Then $D=\xi^{\perp}$, and $\eta$ is the metric dual of $\xi$.

The form $\frac{d\eta}{2}=\frac{\sqrt{-1}}{2}\partial \bar{\partial}\log(|Z_{0}|^{2}+|Z_{1}|^{2}+|Z_{2}|^{2})$ is (the pullback of) the (K\"ahler-form of) the Fubini-Study metric on $\mathbb{P}^{2}$. This fixes the scaling of the Fubini-Study metric in this article. The pullback of the Fubini-Study metric $g_{FS}$ to $\mathbb{S}^{5}$ is a metric on $D$, though it is not a metric on $\mathbb{S}^{5}$. It induces a metric on the contact co-distribution $D^{\star}$. Henceforth, \textit{we collectively call the form $\frac{d\eta}{2}$, and the metrics on $\mathbb{P}^{2}$, $D$, $D^{\star}$ mentioned in the underlying paragraph the Fubini-Study metric}. 

 It is well known that $D^{\star}=\pi^{\star}_{5,4}T^{\star}\mathbb{P}^{2}$. On the other hand, because $\pi_{5,4}$ is a Riemannian submersion, the tangent map $\pi_{5,4,\star}$ is an isometry $D\rightarrow T\mathbb{P}^{2}$ i.e.  $$g_{\mathbb{S}^{5}}\langle v, w\rangle=g_{FS}\langle \pi_{5,4,\star}v, \pi_{5,4,\star}w\rangle.$$


Using the Reeb vector field,  we split the tangent bundle of $\mathbb{C}^{3}\setminus O$ as
\begin{equation}\label{equ splitting of tan bundle C3 minus O}T^{\mathbb{C}}(\mathbb{C}^{3}\setminus O)=span(\frac{\partial}{\partial r},\xi)\oplus \pi^{\star}_{6,4}D^{\mathbb{C}}.\end{equation}
Similarly, the tangent bundle of the $7-$dimensional manifold $(\mathbb{C}^{3}\setminus O)\times \mathbb{S}^{1}$ splits as
\begin{equation}\label{equ splitting of tan bundle C3 minus O times S1}T^{\mathbb{C}}[(\mathbb{C}^{3}\setminus O)\times \mathbb{S}^{1}]=span(\frac{\partial}{\partial s},\frac{\partial}{\partial r},\xi)\oplus \pi^{\star}_{7,4}D^{\mathbb{C}}.\end{equation}
Each of the above two splittings is  orthogonal with respect to the underlying Euclidean metric. 
\subsubsection*{The coordinate neighborhoods}



For any $\beta=0,1,$ or $2$, we define the coordinate neighborhoods by the following.   
\begin{eqnarray}\label{equ def coordinate neighbor}& &U_{\beta,\mathbb{P}^{2}}\triangleq \{[Z]\in \mathbb{P}^{2}|Z_{\beta}\neq 0\}\subset \mathbb{P}^{2},\ U_{\beta,\mathbb{C}^{3}}\triangleq \{Z\in \mathbb{C}^{3}\setminus O|Z_{\beta}\neq 0\}\subset \mathbb{C}^{3},\\& &  \textrm{and}\ U_{\beta,\mathbb{S}^{5}}\triangleq \{Z\in \mathbb{C}^{3}\setminus O|Z_{\beta}\neq 0,\ |Z|=1\}\subset \mathbb{S}^{5}.\nonumber\end{eqnarray}

On each of the  open sets in \eqref{equ def coordinate neighbor}, we recall the well known complex coordinate functions by the following table. 
\begin{equation}\label{equ tabular choice of solution 2nd order operator m neq 0}  \begin{tabular}{|p{3.5cm}|p{4.5cm}|}
  \hline
 Coordinate neighborhoods    & Part of the coordinate functions \\ \hline
$U_{0,\mathbb{P}^{2}},\ U_{0,\mathbb{S}^{5}},\ U_{0,\mathbb{C}^{3}}$ & $u_{1}=\frac{Z_{1}}{Z_{0}},\ u_{2}=\frac{Z_{2}}{Z_{0}}$ \\   \hline
$U_{1,\mathbb{P}^{2}},\ U_{1,\mathbb{S}^{5}},\ U_{1,\mathbb{C}^{3}}$ & $v_{0}=\frac{Z_{0}}{Z_{1}},\ v_{2}=\frac{Z_{2}}{Z_{1}}$ \\   \hline
  $U_{2,\mathbb{P}^{2}},\ U_{2,\mathbb{S}^{5}},\ U_{2,\mathbb{C}^{3}}$ & $w_{0}=\frac{Z_{0}}{Z_{2}},\ w_{1}=\frac{Z_{1}}{Z_{2}}$ \\   \hline
\end{tabular}
 \renewcommand\arraystretch{1.5}
  \end{equation}

In  $U_{0,\mathbb{S}^{5}}$, the complexification $D^{\star,\mathbb{C}}$ is spanned by $du_{1},\ du_{2},\ d\bar{u}_{1},\  d\bar{u}_{2}$ everywhere. Defining the real coordinates $(x_{i},\ i=1,2,3,4)$ by   $$u_{1}=x_{1}+\sqrt{-1}x_{2},\ u_{2}=x_{3}+\sqrt{-1}x_{4},$$
both the real vector bundle $D^{\star}$ and the complex bundle $D^{\star,\mathbb{C}}$ are spanned by $dx_{1},\ dx_{2},\ dx_{3},\  dx_{4}$.
Similar facts also hold in the other two neighborhoods $U_{1,\mathbb{S}^{5}}$ and $U_{2,\mathbb{S}^{5}}$.

\subsubsection*{The coordinate maps}
Based on the above table, on $\mathbb{C}^{3}\setminus O$, we define the Sasakian coordinate.
\begin{Def}\label{Def Sasakian coordinate}Let $Z_{\beta}=|Z_{\beta}|e^{\sqrt{-1}\theta_{\beta}}$, $\theta_{\beta}\in \mathbb{S}^{1}\triangleq \mathbb{R}/2\pi \mathbb{Z}$. On $U_{\beta,\mathbb{C}^{3}}$,  the  \textit{Sasakian coordinate}  is defined to be the functions $(r,\theta_{\beta}, u_{j},u_{k})$ which is a homeomorphism from $U_{\beta,\C^{3}}\rightarrow \mathbb{R}^{+}\times \mathbb{S}^{1}\times \mathbb{C}^{2}$. Similarly, we also call the homeomorphism $(\theta_{\beta}, u_{j},u_{k})$ from $U_{\beta,\mathbb{S}^{5}}$ to the trivial circle bundle $\mathbb{S}^{1}\times \mathbb{C}^{2}$ the \textit{Sasakian coordinate}. 

Both the vector field $\frac{\partial}{\partial \theta_{\beta}}$ and the form $d\theta_{\beta}$ descend to the $\mathbb{S}^{1}-$component of $U_{\beta,\mathbb{S}^{5}}$ ($U_{\beta,\mathbb{C}^{3}}$). 
 \end{Def}
A prominent difference from the Sasakian coordinate to the holomorphic coordinate is that with respect to the standard complex structure on $\mathbb{C}^{3}$, 
 $\frac{\partial}{\partial u_{i}}$ ($i=1,\ 2$) is not  $(1,0)$ in  Sasakian coordinate, but it obviously is $(1,0)$ in the holomorphic coordinate. 
 In a Sasakian coordinate chart, the Reeb vector field can be described satisfactorily. 
\begin{fact}\label{fact Reeb is angular}Let $\beta=0,1$, or $2$, then   $\xi=\frac{\partial}{\partial \theta_{\beta}}$ in $U_{\beta,\mathbb{S}^{5}}$ under Sasakian coordinate.
\end{fact}
\begin{proof}[Proof of Fact \ref{fact Reeb is angular}:]  It suffices to observe that for any point $Z=(r,\theta_{\beta}, u_{j},u_{k})\in U_{\beta,\mathbb{C}^{3}}$, under the Sasakian coordinate, the scalar multiplication $e^{\sqrt{-1}t}Z$ is given by the translation in the angular variable.
$$e^{\sqrt{-1}t}\cdot (r,\theta_{\beta}, u_{j},u_{k})= (r,\theta_{\beta}+t, u_{j},u_{k})\ (\textrm{where}\ \theta_{\beta}+t\in \frac{\mathbb{R}}{2\pi\mathbb{Z}}). $$

\end{proof}

\subsubsection{The Sasaki-Quaternion  structure\label{sect SQ stru}}
The purpose of this section is to introduce the Sasakian-Quaternion structure i.e. Lemma \ref{lem shk str} below. Again, it is a special case  of a Sasaki-Einstein $SU(2)-$structure mentioned in \cite[4.1]{FHN}.  For our purposes, we look for more explicit information.
\subsubsection*{Preliminary}
We begin with a convenient convention. 
\begin{convention}\label{convention pullback and descent} (\textbf{pullback and descent}) 
Given a  holomorphic Hermitian triple $(E,h,A_{O})$ on $\mathbb{P}^{2}$,  let $\theta$ be a usual form or an endomorphism-valued form on $\mathbb{P}^{2}$,  $\mathbb{S}^{5}$,  $\mathbb{C}^{3}\setminus O$, or  $(\mathbb{C}^{3}\setminus O)\times \mathbb{S}^{1}$. Abusing  notation, we \textit{let $\theta$ also denote any pullback or descent}.

For example, the contact form $\eta$ originally defined on $\mathbb{C}^{3}\setminus O$ also means the one on $\mathbb{S}^{5}$. The form $\frac{d\eta}{2}$ a priori on $\mathbb{S}^{5}$ also means the Fubini-Study K\"aher-form on $\mathbb{P}^{2}$.

Similar abusing of notations also applies to differential operators. For example, in \eqref{equ do and dcp2} below,  the exterior derivative on $\mathbb{P}^{2}$, denoted by $d_{\mathbb{P}^{2}}$, also means the local pullback operator on $\mathbb{S}^{5}$. 
\end{convention}

Next, we introduce a formula for the contact form under a Sasakian coordinate chart.  We shall mainly work in $U_{0,\mathbb{S}^{5}}$. This is because $U_{0,\mathbb{S}^{5}}$ is dense and open in $\mathbb{C}^{3}\setminus O$, therefore it suffices to prove the desired identities therein.
 
For any $\beta=0,\ 1,$ or $2$,  let $\phi_{\beta}\triangleq \frac{r^{2}}{|Z_{\beta}|^{2}}=\frac{|Z_{0}|^{2}+|Z_{1}|^{2}+|Z_{2}|^{2}}{|Z_{\beta}|^{2}}$ be the K\"ahler potential of the Fubini-Study metric $\frac{d\eta}{2}$ in $U_{\beta,\mathbb{P}^{2}}$. Then 
\begin{equation}\label{equ norm of Z0} |Z_{\beta}|=\frac{r}{\sqrt{\phi_{\beta}}},\ \textrm{and}\ Z_{\beta}=\frac{r}{\sqrt{\phi_{\beta}}}e^{\sqrt{-1}\theta_{\beta}}.
\end{equation}
 
 \begin{formula}\label{formula eta}For any $\beta=0,1,$ or $2$, $\eta=d\theta_{\beta}+\frac{d^{c}\log\phi_{\beta}}{2}$ in $U_{\beta,\mathbb{S}^{5}}$.
 \end{formula}
 The above formula, whose proof is deferred to Appendix \ref{sect appendix Sasakian},  is used in the proof of Lemma \ref{lem G and H} below. 
 
 \subsubsection*{The semi-basic forms $G$ and $H$}
Using the standard holomorphic volume form on $\mathbb{C}^{3}$, we now define the forms $G$ and $H$ leading to the Sasaki-Quaternion structure. 

 On the cone $\mathbb{C}^{3}\setminus O$, the vector field $\frac{1}{2r^{3}}(r\frac{\partial}{\partial r}-\sqrt{-1}\xi)$ is $(1,0)$. Contracting with the standard $(3,0)-$form on $\mathbb{C}^{3}$, we obtain a form in $\wedge^{(2,0)}D^{\star,\mathbb{C}}$. Namely, the following Lemma holds.  
\begin{lem}\label{lem G and H} There exist (smooth) semi-basic $2-$forms $H$ and $G$ which are sections  of $$(\wedge^{2,0}\oplus\wedge^{0,2})D^{\mathbb{C},\star}\rightarrow \mathbb{S}^{5}$$ with the following properties. Let $\Omega_{\mathbb{C}^{3}}\triangleq dZ_{0}dZ_{1}dZ_{2}$ be the standard holomorphic volume form on $\mathbb{C}^{3}$, we have 
\begin{eqnarray}& &\Omega_{\mathbb{C}^{3}}=(r^{2}dr+\sqrt{-1}r^{3}\eta)\wedge H+(r^{3}\eta-\sqrt{-1}r^{2}dr)\wedge G, \label{equ formula for Omega full}
\\& &Re\Omega_{\mathbb{C}^{3}}=r^{2}dr\wedge H+r^{3}\eta\wedge G,\ Im\Omega_{\mathbb{C}^{3}}=r^{3}\eta\wedge H-r^{2}dr\wedge G,\ \textrm{and} \label{equ formula for Omega}
\\& & [\frac{1}{2r^{3}}(r\frac{\partial}{\partial r}-\sqrt{-1}\xi)]\lrcorner \Omega_{\mathbb{C}^{3}}\triangleq \Theta =H-\sqrt{-1}G.\label{equ contraction def of H and G}
\end{eqnarray}

Under the Sasakian coordinate in $U_{0,\mathbb{S}^{5}},\ U_{1,\mathbb{S}^{5}},\ U_{2,\mathbb{S}^{5}}$ respectively, the following holds true for $G$, $H$, and $\Theta$.

\begin{equation}\label{equ G and H}  \begin{tabular}{|p{1.5cm}|p{13cm}|} \hline 
  In $U_{0,\mathbb{S}^{5}}$ & $G=-\frac{1}{2\sqrt{-1}}(Z_{0}^{3}du_{1}du_{2}-\bar{Z}_{0}^{3}d\bar{u}_{1}d\bar{u}_{2}),\ H=\frac{1}{2}(Z_{0}^{3}du_{1}du_{2}+\bar{Z}_{0}^{3}d\bar{u}_{1}d\bar{u}_{2})$, \\  
    &  $\Theta= Z_{0}^{3}du_{1}du_{2}$. \\
In $U_{1,\mathbb{S}^{5}}$ & $G=\frac{1}{2\sqrt{-1}}(Z_{1}^{3}dv_{0}dv_{2}-\bar{Z}_{1}^{3}d\bar{v}_{0}d\bar{v}_{2}),\ H=-\frac{1}{2}(Z_{1}^{3}dv_{0}dv_{2}+\bar{Z}_{1}^{3}d\bar{v}_{0}d\bar{v}_{2})$, \\ 
& $\Theta=-Z_{1}^{3}dv_{0}dv_{2}$. \\
In $U_{2,\mathbb{S}^{5}}$ & $G=-\frac{1}{2\sqrt{-1}}(Z_{2}^{3}dw_{0}dw_{1}-\bar{Z}_{2}^{3}d\bar{w}_{0}d\bar{w}_{1}),\ H=\frac{1}{2}(Z_{2}^{3}dw_{0}dw_{1}+\bar{Z}_{2}^{3}d\bar{w}_{0}d\bar{w}_{1})$, \\ 
& $\Theta=Z_{2}^{3}dw_{0}dw_{1}$. \\
  \hline
\end{tabular}
 \renewcommand\arraystretch{1.5}
  \end{equation}

\end{lem}
The proof of the above Lemma is by routine calculation, and is also deferred to Appendix \ref{sect appendix Sasakian}.

In the Sasakian geometry of $\mathbb{C}^{3}\setminus O\rightarrow \mathbb{S}^{5}$, schematically speaking, the form $H-\sqrt{-1}G$ plays similar role as a no-where vanishing holomorphic volume form on a $K3-$surface.  This is analogous to hyper-K\"ahler geometry. 

We  search for more properties of the forms $G$ and $H$. Let $\star_{0}$ denote the Hodge star operator of the Fubini-Study metric on the exterior algebra of $D^{\star}$ (see the paragraph about Fubini-Study metric in Section \ref{sect Sasakian coordinate}). Because $H$ and $G$ are both in $(\wedge^{(2,0)}\oplus \wedge^{(0,2)})D^{\mathbb{C},\star}$, they are $\star_{0}-$self-dual i.e. 
\begin{equation}
\star_{0}G=G,\ \star_{0}H=H. 
\end{equation}
At the point $(1,0,0)\in \mathbb{S}^{5}$, 
\begin{equation}\label{equ G and H at 1.0.0}
G=-Im(du_{1}du_{2}),\ H=Re(du_{1}du_{2}).
\end{equation}

\subsubsection*{The Quaternion structure}
Before proceeding, we stipulate the following. 
\begin{convention}\label{convention contraction} We let $\lrcorner$ denote the contraction between two forms on $\mathbb{S}^{5}$ under the standard round metric. 

The same notation might also denote the usual contraction between a tangent vector and a form (without involving the Riemannian metric). Our rationale is that if a tensor operation or operator has no subscript for the domain, then it is on $\mathbb{S}^{5}$.
\end{convention}

 We now get into the crucial properties of $G$ and $H$. In view of the contraction \eqref{equ def contraction between forms}, for any integer $p\geq 1$,  let $\lrcorner$ denote the contraction between $T^{\star}\mathbb{S}^{5}$ and $\wedge^{p}T^{\star}\mathbb{S}^{5}$ under the standard metric $g_{\mathbb{S}^{5}}$ on $\mathbb{S}^{5}$. At an arbitrary point on $\mathbb{S}^{5}$, it is straight forward to verify the following (for example, under the Sasakian-Quaternion coordinate in Appendix \ref{Appendix SQ coordinate} so that $G$ and $H$ is of the canonical form \eqref{equ G and H at 1.0.0}).\begin{equation}\label{equ G and H are complex structures}
(a\lrcorner G)\lrcorner G=-a,\ (a\lrcorner H)\lrcorner H=-a.
\end{equation}
Consequently, both $\lrcorner G$ and $\lrcorner H$ are almost complex structures on $D^{\star}$. 
From now on, let $J_{G},\  J_{H}$ denote $\lrcorner G,\ \lrcorner H$. Let $J_{0}$ denote $\lrcorner \frac{d\eta}{2}$. They are all isometries.

We define the complex structures on $D$ naturally by the metric pulling up and down i.e. \begin{equation}J_{0}(X)\triangleq [J_{0}(X^{\sharp_{0}})]_{\sharp_{0}},\ J_{H}(X)\triangleq [J_{H}(X^{\sharp_{0}})]_{\sharp_{0}},\ J_{G}(X)\triangleq [J_{G}(X^{\sharp_{0}})]_{\sharp_{0}}.
\end{equation}
The Sasaki-Quaternion structure applies to the contact distribution $D$ as well. 

Based on the above, we routinely verify our main Lemma in the underlying section.
\begin{lem}\label{lem shk str}(The Sasaki-Quaternion  structure) On $D^{\star}\rightarrow \mathbb{S}^{5}$, $\pi^{\star}_{7,5}D^{\star}\rightarrow (\mathbb{C}^{3}\setminus O)\times \mathbb{S}^{1}$, and also on $D\rightarrow \mathbb{S}^{5}$, $D\rightarrow (\mathbb{C}^{3}\setminus O)\times \mathbb{S}^{1}$, the following holds.   \begin{equation}\label{equ shk str}J_{G}J_{H}=J_{0},\ J_{H}J_{0}=J_{G},\ J_{0}J_{G}=J_{H},\ J_{0}^{2}=J_{H}^{2}=J_{G}^{2}=-Id. \end{equation}
 Moreover, the Fubini-Study metric  on $D$ and $D^{\star}$ is preserved by each of $J_{0},\ J_{H},\ J_{G}$.

Consequently, the above is true for the complexfications $D^{\star,\mathbb{C}}$, $D^{\mathbb{C}}$, $\pi^{\star}_{7,5}D^{\star,\mathbb{C}},\ \pi^{\star}_{7,5}D^{\mathbb{C}}$. 

Given a  holomorphic Hermitian triple $(E,h,A_{O})$ on $\mathbb{P}^{2}$,  the identities in \eqref{equ shk str} hold for an endomorphism-valued semi-basic $1-$form. 
\end{lem}

  



\subsubsection{The Reeb Lie derivative and the transverse exterior derivative}
To describe certain eigensections of the operator $P$, we need the two first order differential operators in Formula \ref{formula lie derive of G and H} and Definition \ref{Def d0}.
\begin{formula}\label{formula lie derive of G and H}Let $L_{\xi}$ denote the Lie derivative in the direction of the Reeb vector field. Then the followings is true. 
\begin{equation}
 L_{\xi}H=3G,\ \ L_{\xi}G=-3H,\ \ L_{\xi}(\frac{d\eta}{2})=0. \end{equation}

\end{formula}
\begin{proof}[Proof of Formula \ref{formula lie derive of G and H}:] Differentiating \eqref{equ G and H} with respect to $\theta_{0}$, the equalities hold everywhere in $U_{0,\mathbb{S}^{5}}$, therefore everywhere  in $\mathbb{C}^{3}\setminus O$ by continuity. 
\end{proof}

Given a  holomorphic Hermitian triple on $\mathbb{P}^{2}$, let $a_{0}$ be a $\pi^{\star}_{5,4}EndE-$valued semi-basic $1-$form on $\mathbb{S}^{5}$. The following holds by the formula for $G$ and $H$ (Lemma \ref{equ G and H}) and the local formula for the Reeb vector field (Fact \ref{fact Reeb is angular}). 
\begin{equation}J_{G}L_{\xi}(a_{0})=L_{\xi}J_{G}(a_{0})+3J_{H}(a_{0}),\  J_{H}L_{\xi}(a_{0})=L_{\xi}J_{H}(a_{0})-3J_{G}(a_{0})\label{equ lie commutator with hyperkahler str}
\end{equation}

\begin{Notation}Most of the time, to avoid heavy notation, for an differential operator on the bundle,  we shall suppress the subscript for the connection.\end{Notation}

We now turn to the definition of a derivative operator with respect to the contact co-distribution $D^{\star}$. 
\begin{Def} \label{Def d0}Given a  holomorphic Hermitian triple $(E,h,A_{O})$ on $\mathbb{P}^{2}$, on the bundle $\wedge^{p}D^{\star}\otimes \pi_{5,4}^{\star}EndE \rightarrow \mathbb{S}^{5}$ of endomophism-valued semi-basic $p-$forms, the transverse exterior derivative operator $d_{0}$ is defined by 
\begin{equation}
d_{0}\triangleq d-\eta\wedge L_{\xi}.
\end{equation}
It turns out that if $\theta$ is semi-basic, so is $d_{0}\theta$. The Lie derivative $L_{\xi}$ is well defined because the endomorphism bundle is pulled back from $\mathbb{P}^{2}$.

\end{Def}

The operator $d_{0}$ admits a local splitting in the following sense.

For any $\beta=0,1,2$, let ``$b$" be a section (form) of $\wedge^{p}D^{\star,\mathbb{C}}\otimes \pi_{5,4}^{\star}EndE \rightarrow U_{\beta,\mathbb{S}^{5}}$. For any $\theta_{\beta}\in [0,2\pi)$, the restriction $b(\cdot, \theta_{\beta})$ onto the $\theta_{\beta}-$slice is a form on $U_{\beta,\mathbb{P}^{2}}$. We define $d_{\mathbb{P}^{2}}b$ to be  the (partial) exterior derivative in the direction of $U_{\beta,\mathbb{P}^{2}}$. A priori,  this partial exterior derivative is not defined globally on $\mathbb{S}^{5}$ because it is not the product manifold $\mathbb{P}^{2}\times \mathbb{S}^{1}$. However, the open set $U_{\beta,\mathbb{S}^{5}}$ is a trivial $\mathbb{S}^{1}-$bundle over $U_{\beta,\mathbb{P}^{2}}\subset \mathbb{P}^{2}$. Employing Formula \ref{formula eta} for $\eta$,  this leads to the following splitting of $d_{0}$ in $U_{\beta,\mathbb{S}^{5}}$.
\begin{eqnarray}\label{equ do and dcp2} d_{0}b &=&db -\eta\wedge L_{\xi}b=d_{\mathbb{P}^{2}}b+d\theta_{\beta}\wedge L_{\xi}b-\eta\wedge L_{\xi}b\nonumber
\\&=& d_{\mathbb{P}^{2}}b-\frac{1}{2}(d^{c}\log \phi_{\beta})\wedge L_{\xi}b.
\end{eqnarray}
In view of the above decomposition, we have the further splitting
\begin{eqnarray}\label{equ d0 dbar0 partial0}
d_{0}=\partial_{0}+\bar{\partial}_{0},\ \ \textrm{where}& & \nonumber
\\ \partial_{0}=\partial_{\mathbb{P}^{2}}+\frac{\sqrt{-1}}{2}(\partial_{\mathbb{P}^{2}} \log\phi_{\beta})\wedge L_{\xi},\ \textrm{and}\ & &
\bar{\partial}_{0}=\bar{\partial}_{\mathbb{P}^{2}}-\frac{\sqrt{-1}}{2}(\bar{\partial}_{\mathbb{P}^{2}} \log\phi_{\beta})\wedge L_{\xi}. 
\end{eqnarray}

Given an arbitrary semi-basic $(p,q)-$form $b$, $\partial_{0}b=(d_{0}b)^{p+1,q},\ \bar{\partial}_{0}b=(d_{0}b)^{p,q+1}$. Hence, both of the two operators are globally defined on $\mathbb{S}^{5}$. 

For any $\beta$,  let $x_{i}, i=1,2,3,4$ be the Euclidean coordinate functions on $U_{\beta, \mathbb{S}^{5}}=\mathbb{S}^{1}\times \mathbb{C}^{2}$. The identity $$\eta-d\theta_{\beta}=\eta(\frac{\partial}{\partial x_{j}})dx^{j}$$
is verified on the basis $(\xi=\frac{\partial}{\partial \theta_{\beta}}, \frac{\partial}{\partial x_{j}},\ j=1,2,3,4)$. Given an endomorphism $u$, we have 
\begin{equation}\label{equ d0u Euc formula}
d_{0}u= [\frac{\partial u}{\partial x_{j}}-\xi(u)\eta(\frac{\partial }{\partial x_{j}})]dx^{j}.
\end{equation}
Similarly, given a semi-basic endomorphism-valued $1-$form $a_{0}=\Sigma^{4}_{i=1}a_{i}dx^{i}$, we have 
\begin{equation}
d_{0}a_{0}= [\frac{\partial a_{i}}{\partial x_{j}}-\xi(a_{i})\eta(\frac{\partial }{\partial x_{j}})]dx^{j}\wedge dx^{i}.
\end{equation}

\begin{rmk}Our calculation for the bundle-valued forms remains true for usual forms, unless the irreducible condition is required.  This is because we can simply let it be the trivial line bundle. For example, Definition \ref{Def d0} and its subsequent calculations  hold for  usual forms. \end{rmk}
\subsection{The deformation operator for $G_{2}-$instantons}
\subsubsection{Separation of variables and the ``Quaternion" structure on $(\mathbb{C}^{3}\setminus O)\times \mathbb{S}^{1}$}
The purpose of this section is to state the formula (Lemma \ref{lem formula of the model dirac deformation operator}) for the linearized operator of the $G_{2}-$instanton equation under the model data defined in \eqref{equ  formula for model deformation operator}.
  \begin{formula}\label{formula pre splitting}(\cite[Proposition 3.13]{SW}) In view of the splitting in \eqref{equ 1st splitting of the 1form}, we have \begin{equation}
L_{A_{O},\phi_{\mathbb{C}^{3}\times \mathbb{S}^{1}}}[1,ds,1]\cdot\left[\begin{array}{c}\sigma \\ \underline{a}_{s}   \\  a_{\mathbb{C}^{3}}\end{array}\right]
=[1,ds,1]\cdot\{(\frac{\partial}{\partial s} \left[\begin{array}{ccc}0 & -1 &0 \\ 1 & 0 &0  \\ 0 & 0 &J_{\mathbb{C}^{3}} \end{array}\right]+\square) \left[\begin{array}{c}\sigma\\ \underline{a}_{s}   \\  a_{\mathbb{C}^{3}}\end{array}\right] \}
\end{equation}
where $\square \left[\begin{array}{c}\sigma\\ \underline{a}_{s}   \\  a_{\mathbb{C}^{3}}\end{array}\right]=\left[\begin{array}{c}d^{\star_{\mathbb{C}^{3}} }a_{\mathbb{C}^{3}}\\ (d_{\mathbb{C}^{3}}a_{\mathbb{C}^{3}})\lrcorner_{\mathbb{C}^{3}} \omega_{\mathbb{C}^{3}} \\ d_{\mathbb{C}^{3}}\sigma-J_{\mathbb{C}^{3}} (d_{\mathbb{C}^{3}}\underline{a}_{s})+(d_{\mathbb{C}^{3}}a_{\mathbb{C}^{3}})\lrcorner_{\mathbb{C}^{3}} Re\Omega_{\mathbb{C}^{3}}.\end{array}\right]$.
\end{formula}

To obtain a finer splitting, we need to generalize the Sasaki-Quaternion structure on $\mathbb{S}^{5}$ in Lemma  \ref{lem shk str} to the domain bundle  on the $7-$dimensional manifold $(\C^{3}\setminus O)\times \mathbb{S}^{1}$.

\begin{lem}\label{lem hk on 7-dim model}(The ``Quaternion" structure on $(\C^{3}\setminus O)\times \mathbb{S}^{1}$) Under the setting from \eqref{equ fine splitting aC3} to \eqref{equ  5 element basis} above, let  the column vector $\left[\begin{array}{c}u \\ a_{s} \\  a_{r} \\ a_{\eta} \\ a_{0}\end{array}\right]$ represents the $5$ components of $Dom_{\mathbb{S}^{5}}$ respectively. Let 
\begin{equation}I=\left|\begin{array}{ccccc} 0 &-1 &0 & 0 & 0 \\ 1 & 0 & 0 
& 0 & 0\\ 0 & 0 & 0 
& -1 & 0  \\  0 & 0 & 1 
& 0 & 0 \\
  0 & 0 &  0
& 0 &  J_{0}\end{array}\right|,\ K=\left|\begin{array}{ccccc} 0 &0 &-1 & 0 & 0 \\ 0 & 0 & 0 
& 1 & 0\\ 1 & 0 & 0 
& 0 & 0  \\  0 & -1 & 0
& 0 & 0 \\
  0 & 0 &  0
& 0 &  J_{H}\end{array}\right|,
\end{equation}
\begin{equation}T=\left|\begin{array}{ccccc} 0 &0 &0 & -1 & 0 \\ 0 & 0 & -1 
& 0 & 0\\ 0 & 1 & 0 
& 0 & 0  \\  1 & 0 & 0 
& 0 & 0 \\
  0 & 0 &  0
& 0 &  J_{G}\end{array}\right|,\   \textrm{and}\ \underline{T}=\left|\begin{array}{ccccc} 0 &0 &0 & -1 & 0 \\ 0 & 0 & -1 
& 0 & 0\\ 0 & 1 & 0 
& 0 & 0  \\  1 & 0 & 0 
& 0 & 0 \\
  0 & 0 &  0
& 0 &  -J_{G}\end{array}\right|
\end{equation}
be the isometries of $Dom_{\mathbb{S}^{5}}$ (and $Dom_{7}$)
 acting on the column vector (by left multiplication). Then  $I,K,\underline{T}$ form an quaternion structure. i.e. all the pairwise multiplications anti-commute, and the following is true.  \begin{equation}\label{equ the quarternionic str}K\underline{T}=I,\ IK=\underline{T},\ \underline{T}I=K,\ \textrm{and}\ I^{2}=K^{2}=\underline{T}^{2}=-Id_{Dom_{\mathbb{S}^{5}}}. \end{equation}

\end{lem}

Under Convention \ref{convention contraction} on the tensor contractions, we state our main Lemma. 
\begin{lem}\label{lem formula of the model dirac deformation operator} Given a Hermitian Yang-Mills triple $(E,h,A_{O})$ on $\mathbb{P}^{2}$, 
 under the model setting in \eqref{equ  formula for model deformation operator}, the following formula for the model linearized operator holds true. 

\begin{equation}L_{A_{O},\phi_{\mathbb{C}^{3}\times \mathbb{S}^{1}}}\left[\begin{array}{ccccc}\frac{1}{r} &0 &0 &0 &0 \\ 0 &\frac{ds}{r} &\frac{dr}{r}  &\eta &1 \end{array}\right]\cdot\left[\begin{array}{c}u \\ a_{s} \\  a_{r} \\ a_{\eta} \\ a_{0}\end{array}\right]
=\left[\begin{array}{ccccc}\frac{1}{r} &0 &0 &0 &0 \\ 0 &\frac{ds}{r} &\frac{dr}{r}  &\eta &1 \end{array}\right] [\frac{\partial}{\partial s}\circ I+K\circ (\frac{\partial}{\partial r}-\frac{P}{r})]\left[\begin{array}{c}u \\ a_{s} \\  a_{r} \\ a_{\eta} \\ a_{0}\end{array}\right],\end{equation}where
\begin{equation}\label{equ formula for P}
P\left[\begin{array}{c}u \\ a_{s} \\  a_{r} \\ a_{\eta} \\ a_{0}\end{array}\right]=\left[\begin{array}{ccccc}1  & -L_{\xi} & 0 & 0& -(d_{0}\cdot)\lrcorner H\\ L_{\xi} & 1 & 0 & 0& (d_{0}\cdot)\lrcorner {G} \\  0 & 0 & -4 & -L_{\xi}& d^{\star_{0}}_{0}  \\  0 & 0 & L_{\xi} & -4& -(d_{0}\cdot)\lrcorner \frac{d\eta}{2}  \\  J_{H}d_{0} & -J_{G}d_{0} & d_{0} & J_{0}d_{0}& -L_{\xi}J_{0}\end{array}\right]\left[\begin{array}{c}u \\ a_{s} \\  a_{r} \\ a_{\eta} \\ a_{0}\end{array}\right]
\end{equation}
is a first-order self-adjoint elliptic operator on the bundle $Dom_{\mathbb{S}^{5}}$ (see Definition \ref{Def Dom bundle}).
\end{lem}
Under the tools established in the introduction, Section \ref{sect SQ str}, and the generalized Quaternion structure in the above Lemma \ref{lem hk on 7-dim model}, the proof of the above Lemma is a routine and fairly tedious tensor calculation. We defer it to Appendix \ref{Appendix proof of formula of P}. 

\begin{convention} Let the $[\pi^{\star}_{5,4}(adE)]^{\oplus 4}-$component be zero, a semi-basic  $\pi_{7,5}^{\star}adE-$valued $1-$form $a_{0}$ can be viewed as in  $Dom_{\mathbb{S}^{5}}$. The corresponding vector under the basis in \eqref{equ  5 element basis} is $\left[\begin{array}{c}0 \\ 0 \\  0 \\ 0 \\ a_{0}\end{array}\right]$ i.e. the first $4$ components are zero. 
\end{convention}
We  note again that Lemma \eqref{lem formula of the model dirac deformation operator} does not require the connection $A_{O}$ to be Hermitian Yang-Mills, but the Bochner formulas in the following section do i.e. it is not known whether they remain true  if $A_{O}$ is not  Hermitian Yang-Mills.

\subsubsection{Bochner formula for $P$}
In the Quaternion structure, both  $I,\ K$ commute with $\frac{\partial}{\partial r}$ and $\frac{\partial}{\partial s}$. Using the Quaternion identities \eqref{equ the quarternionic str}, we routinely verify the following formula for commutators between the operator $P$ and  $I,\ K,\ \underline{T}$.
\begin{equation}\label{equ commutators for PIJK}
PI=IP,\ KP+PK=-3K,\  \underline{T}P+P\underline{T}=-3\underline{T}.
\end{equation}
\begin{rmk}\label{rmk EP is complex} That $P$ commutes with $I$ makes $I$ a complex structure of the eigenspaces of $P$.
\end{rmk}
Consequently, we straight-forwardly verify the following formula for the square of  $L_{A_{O},\phi_{\mathbb{C}^{3}\times \mathbb{S}^{1}}}$. 
\begin{equation}\label{equ L2 by Separation}
L^{2}_{A_{O},\phi_{\mathbb{C}^{3}\times \mathbb{S}^{1}}}=-\frac{\partial^{2}}{\partial s^{2}}-\frac{\partial^{2}}{\partial r^{2}}-\frac{3}{r}\frac{\partial}{\partial r}+\frac{P^{2}+2P}{r^{2}}.
\end{equation}

We have another formula for $L^{2}_{A_{O},\phi_{\mathbb{C}^{3}\times \mathbb{S}^{1}}}$ than the above. 

\begin{rmk} In conjunction with the discussion below equation \eqref{equ introduction formula for deformation operator}, because the linearized operator only depends on the projective connection induced, we denote the curvature form of the projective connection  by $F^{0}_{A_{O}}$ (or $F^{0}_{A_{0}}$).
\end{rmk}

In arbitrary dimension $n$,  on $\mathbb{R}^{n}\setminus O$, let $A_{0}$ be a pullback  connection from $\mathbb{S}^{n-1}$. Let $\boxdot^{dim n}_{A_{0}}\triangleq \nabla^{\star}\nabla+2F^{0}_{A_{0}}\otimes_{\mathbb{R}^{n}}$ be the operator acting on $\Omega^{1}(EndE)\rightarrow \mathbb{R}^{n}\setminus O$ ($EndE$ is the pullback endomorphism bundle). Using the formula 
\cite[Lemma 3.2]{WangyuanqiJGP}  for the rough Laplacian $\nabla^{\star,\mathbb{R}^{n}}\nabla^{\mathbb{R}^{n}}$ on $1-$forms, we find
\begin{equation}\label{equ square dot operator}\boxdot^{dim n}_{A_{0}}=-\frac{\partial^{2}}{\partial r^{2}}-\frac{n-3}{r}\frac{\partial}{\partial r}+\frac{\widehat{B}^{dim n}_{0}}{r^{2}},
\end{equation}
where


\begin{equation}\label{equ formula for B0 general dim}
\widehat{B}_{0,dim n}\left[\begin{array}{c}  a_{r} \\ a \\ \end{array}\right]\triangleq \left[\begin{array}{c}  \nabla^{\mathbb{S}^{n-1},\star}\nabla^{\mathbb{S}^{n-1}}a_{r}-2d^{\star}a+2(n-2)a_{r}\\  \nabla^{\mathbb{S}^{n-1},\star}\nabla^{\mathbb{S}^{n-1}}a-2d_{s}a_{r}+(n-2)a+2F^{0}_{A_{0}}\otimes_{\mathbb{S}^{n-1}}a \end{array}\right].
\end{equation}
$\boxdot^{dim n}_{A_{0}}$ is the ``linearized" operator  for the Yang-Mills equation with gauge fixing (cf. \cite[6.1]{WangyuanqiJFA}).

Then we go on to prove Theorem \ref{Thm eigenvalue Sasakian}. 

\begin{formula}\label{Formula Bochner for P} Given a Hermitian Yang-Mills triple on $\mathbb{P}^{2}$, in view of Lemma \eqref{lem formula of the model dirac deformation operator}, and still let $\nabla^{\star}\nabla$ denote the rough Laplacian on $\pi^{\star}_{5,4}adE$, the following identity holds $$P^{2}+2P=B_{0,dim6},$$ where
\begin{equation}\label{equ formula for B0}
B_{0,dim6}\left[\begin{array}{c}u \\ a_{s} \\  a_{r} \\ a \\ \end{array}\right]\triangleq \left[\begin{array}{c}\nabla^{\star}\nabla u+3u\\ \nabla^{\star}\nabla a_{s}+3a_{s} \\  \nabla^{\star}\nabla a_{r}-2d^{\star}a+8a_{r}\\  \nabla^{\star}\nabla a-2d a_{r}+4a+2F^{0}_{A_{0}}\otimes_{\mathbb{S}^{5}}a \end{array}\right].
\end{equation}\end{formula}
In relation to the remark under Theorem \ref{Thm 1},  we need the Hermitian Yang-Mills condition in Lemma \ref{Formula Bochner for P} but not in \ref{lem formula of the model dirac deformation operator} because the formula  \cite[(33)]{WangyuanqiJGP} needs the pullback   connection on $(\mathbb{C}^{3}\setminus O)\times \mathbb{S}^{1}$ to be a projective $G_{2}-$instanton. 
\begin{proof}[Proof of Formula \ref{Formula Bochner for P}:] The observation is that, using the usual Euclidean coordinates on $\mathbb{C}^{3}\times\mathbb{S}^{1}$ (induced from $\mathbb{R}^{7}$), we have a another way to compute $L^{2}_{A_{O},\phi_{\mathbb{C}^{3}\times \mathbb{S}^{1}}}$ that yields 
\begin{equation}\label{equ L2 by Euc coordinate}L^{2}_{A_{O},\phi_{\mathbb{C}^{3}\times \mathbb{S}^{1}}}=-\frac{\partial^{2}}{\partial s^{2}}-\frac{\partial^{2}}{\partial r^{2}}-\frac{3}{r}\frac{\partial}{\partial r}+\frac{B_{0,dim6}}{r^{2}}.
\end{equation}
The proof is  complete  comparing \eqref{equ L2 by Euc coordinate} with \eqref{equ L2 by Separation}.

It remains to show \eqref{equ L2 by Euc coordinate}. It directly follows from the identities in \cite{WangyuanqiJGP}. In view of the splitting in \eqref{equ 1st splitting of the 1form}, the Bochner formula  \cite[(146)]{WangyuanqiJGP}, which holds for projective $G_{2}-$instantons (because locally the connection form acts  on endorphisms-valued forms via Lie-bracket),  yields 
\begin{equation}\label{equ formula for L2}
L^{2}_{A_{O},\phi_{\mathbb{C}^{3}\times \mathbb{S}^{1}}}[\frac{1}{r},1]\left[\begin{array}{c}u \\ a_{\mathbb{C}^{3}} \end{array}\right]=\left[\begin{array}{c}\nabla^{\star,\mathbb{C}^{3}\times \mathbb{S}^{1}} \nabla^{\mathbb{C}^{3}\times \mathbb{S}^{1}} (\frac{u}{r})\\ \nabla^{\star,\mathbb{C}^{3}\times \mathbb{S}^{1}}\nabla^{\mathbb{C}^{3}\times \mathbb{S}^{1}} a_{\mathbb{C}^{3}}+2F^{0}_{A_{0}}\otimes_{\mathbb{C}^{3}}\underline{a}\end{array}\right].
\end{equation}
 We note that the proof of \cite[(146)]{WangyuanqiJGP} is by Euclidean coordinates for the model $G_{2}-$structure, thus it also holds in our case as $\mathbb{C}^{3}\times \mathbb{S}^{1}$ possesses such coordinates for the standard $G_{2}-$structure $\phi_{\mathbb{C}^{3}\times \mathbb{S}^{1}}$.
  
The point is that $\nabla^{\star,\mathbb{C}^{3}\times \mathbb{S}^{1}} \nabla^{\mathbb{C}^{3}\times \mathbb{S}^{1}}=-\frac{\partial^{2}}{\partial s^{2}}+\nabla^{\star,\mathbb{C}^{3}} \nabla^{\mathbb{C}^{3}}$, and $ds$ is
$\nabla^{\mathbb{C}^{3}\times \mathbb{S}^{1}}-$parallel. We then compute
  $$\nabla^{\star,\mathbb{C}^{3}\times \mathbb{S}^{1}} \nabla^{\mathbb{C}^{3}\times \mathbb{S}^{1}}(\frac{ a_{s}ds}{r})=[\nabla^{\star,\mathbb{C}^{3}\times \mathbb{S}^{1}} \nabla^{\mathbb{C}^{3}\times \mathbb{S}^{1}}(\frac{a_{s}}{r})]ds=[-\frac{\partial^{2}}{\partial s^{2}}(\frac{a_{s}}{r})+\nabla^{\star,\mathbb{C}^{3}} \nabla^{\mathbb{C}^{3}}(\frac{a_{s}}{r})]ds,$$ and the similar identity holds for $\nabla^{\star,\mathbb{C}^{3}\times \mathbb{S}^{1}} \nabla^{\mathbb{C}^{3}\times \mathbb{S}^{1}} (\frac{u}{r})$ i.e. $$\nabla^{\star,\mathbb{C}^{3}\times \mathbb{S}^{1}} \nabla^{\mathbb{C}^{3}\times \mathbb{S}^{1}}(\frac{u}{r})=-\frac{\partial^{2}}{\partial s^{2}}(\frac{u}{r})+\nabla^{\star,\mathbb{C}^{3}} \nabla^{\mathbb{C}^{3}}(\frac{u}{r}).$$ Hence, 
  in view of the splitting in  \eqref{equ fine splitting aC3}, we find 
\begin{equation}\label{equ formula for L2 x}
L^{2}_{A_{O},\phi_{\mathbb{C}^{3}\times \mathbb{S}^{1}}}[\frac{1}{r},\frac{ds}{r},1]\left[\begin{array}{c}u \\ a_{s}\\ a \end{array}\right]\triangleq -\frac{\partial^{2}}{\partial s^{2}}[\frac{1}{r},\frac{ds}{r},1]\left[\begin{array}{c}u \\ a_{s}\\ a \end{array}\right]+ [\frac{1}{r},\frac{ds}{r},1]\left[\begin{array}{c}r\nabla^{\star,\mathbb{C}^{3}} \nabla^{\mathbb{C}^{3}} (\frac{u}{r})\\ r\nabla^{\star,\mathbb{C}^{3}} \nabla^{\mathbb{C}^{3}} (\frac{a_{s}}{r})\\ \nabla^{\star,\mathbb{C}^{3}}\nabla^{\star,\mathbb{C}^{3}} a+2F^{0}_{A_{0}}\otimes_{\mathbb{C}^{3}}a\end{array}\right].
\end{equation}
Now we view $\mathbb{C}^{3}\setminus \{O\}$ as the real $6-$dimensional flat cone over $\mathbb{S}^{5}$. Using the formulas
\cite[(29) and Lemma 3.2]{WangyuanqiJGP} (let $n=6$ therein) for the rough Laplacians $\nabla^{\star,\mathbb{C}^{3}}\nabla^{\mathbb{C}^{3}} a$, $r\nabla^{\star,\mathbb{C}^{3}} \nabla^{\mathbb{C}^{3}} (\frac{a_{s}}{r})$, and $r\nabla^{\star,\mathbb{C}^{3}} \nabla^{\mathbb{C}^{3}} (\frac{u}{r})$, the proof for \eqref{equ L2 by Euc coordinate} is complete. \end{proof}

 The following formula says the $3$ forms yielding the Sasaki-Quaternion structure are all $d_{0}-$harmonic.
\begin{formula}\label{formula d0 closedness of the 3 forms}The following vanishing holds. \begin{equation}\label{equ formula d0 closedness of the 3 forms}d_{0}(\frac{d\eta}{2})=d_{0}G=d_{0}H=d_{0}\Theta=d_{0}\bar{\Theta}=0.\end{equation} Consequently, because they are all $\star_{0}$ self-dual, $$d_{0}^{\star_{0}}(\frac{d\eta}{2})=d_{0}^{\star_{0}}G=d_{0}^{\star_{0}}H=d_{0}^{\star_{0}}\Theta=d_{0}^{\star_{0}}\bar{\Theta}=0.$$
\end{formula}
\begin{proof}[Proof of Formula \ref{formula d0 closedness of the 3 forms}:] 
Routine calculation shows that the two individual terms in the formula \eqref{equ do and dcp2} for $d_{0}$  are 
$$d_{\mathbb{P}^{2}}(Z^{3}_{0}du_{1}du_{2})=-\frac{3}{2}\bar{\partial}\log \phi_{0}\wedge (Z^{3}_{0}du_{1}du_{2}),$$
and 
$$-\frac{1}{2}[(d^{c}_{\mathbb{P}^{2}})\log \phi_{0}]\wedge L_{\xi}(Z^{3}_{0}du_{1}du_{2})=\frac{3}{2}[\bar{\partial}\log \phi_{0}]\wedge (Z^{3}_{0}du_{1}du_{2}).$$
Then \eqref{equ do and dcp2} says that $d_{0}(Z^{3}_{0}du_{1}du_{2})=0$. Taking complex conjugate, 
we find $$d_{0}(\bar{Z}^{3}_{0}d\bar{u}_{1}d\bar{u}_{2})=0.$$ Using the expression of $G,\ H,$ and $\Theta$ in
\eqref{equ G and H}, \eqref{equ formula d0 closedness of the 3 forms} holds in $U_{0,\mathbb{S}^{5}}$. Because they are both smooth forms, by continuity, \eqref{equ formula d0 closedness of the 3 forms} holds true everywhere on $\mathbb{S}^{5}$.
\end{proof}

The $\star_{0}$ self-duality of the forms $\frac{d\eta}{2}$, $G$, $H$ also yields the following identities. 
\begin{formula}\label{formula dstarJ and da contraction} $d_{0}^{\star_{0}} J_{0}(a_{0})=d_{0}(a_{0})\lrcorner \omega_{0}$,  $d_{0}^{\star_{0}} J_{G}(a_{0})=d_{0}(a_{0})\lrcorner G$,  $d_{0}^{\star_{0}} J_{H}(a_{0})=d_{0}(a_{0})\lrcorner H$. 
\end{formula} 
\begin{proof}[Proof of Formula \ref{formula dstarJ and da contraction}:] We only prove the second identity, the other two are similar. We calculate
\begin{eqnarray*}& &d^{\star_{0}}_{0}(J_{G}a_{0})\triangleq -\star_{0}d_{0} \star_{0}(a_{0}\lrcorner G)=-\star_{0}d_{0} \star_{0}\star_{0}(a_{0}\wedge G)=\star_{0}d_{0}(a_{0}\wedge G)\nonumber
\\&=&\star_{0}[(d_{0}a_{0})\wedge G]\ \ \ \ \ \ \ \ \ \ \ \ \ \ \ \ \ \ \ \ \ (\textrm{by the}\ d_{0}-\textrm{closeness in Formula}\ \eqref{formula d0 closedness of the 3 forms}).
\\&=&(d_{0}a_{0})\lrcorner G.
\end{eqnarray*}
\end{proof}

\begin{lem}\label{lem Bochner} (Bochner formulas for $P$) Given a Hermitian Yang-Mills triple on $\mathbb{P}^{2}$, still in view of the 5 component separation in \eqref{equ fine splitting aC3} and Lemma \ref{lem formula of the model dirac deformation operator},  the following holds. \begin{equation}\label{equ formula for P2+2P separation}
(P^{2}+2P)\left[\begin{array}{c}u \\ a_{s} \\  a_{r} \\ a_{\eta} \\ a_{0} \end{array}\right]=\left[\begin{array}{c}\nabla^{\star}\nabla u+3u\\ \nabla^{\star}\nabla a_{s}+3a_{s} \\  \nabla^{\star}\nabla a_{r}-2d^{\star_{0}}_{0}a_{0}+2L_{\xi}a_{\eta}+8a_{r}\\  \nabla^{\star}\nabla a_{\eta}-2L_{\xi}a_{r}+8a_{\eta}+2d_{0}^{\star_{0}}J_{0}(a_{0})\\  \nabla^{\star}\nabla a_{0}-2d_{0}a_{r}-2J_{0}(d_{0}a_{\eta})+4a_{0}+2F^{0}_{A_{O}}\otimes_{\mathbb{S}^{5}}a_{0}\end{array}\right].
\end{equation}
Consequently,
\begin{equation}\label{equ formula for P2+4P separation}
(P^{2}+4P)\left[\begin{array}{c}u \\ a_{s} \\  a_{r} \\ a_{\eta} \\ a_{0} \end{array}\right]=\left[\begin{array}{c}\nabla^{\star}\nabla u+5u-2L_{\xi}a_{s}-2d^{\star_{0}}_{0}J_{H}(a_{0})\\  \nabla^{\star}\nabla a_{s}+5a_{s}+2L_{\xi}u+2d^{\star_{0}}_{0}J_{G}(a_{0}) \\  \nabla^{\star}\nabla a_{r}\\  \nabla^{\star}\nabla a_{\eta}\\ \nabla^{\star}\nabla a_{0}+4a_{0}+2F^{0}_{A_{O}}\otimes_{\mathbb{S}^{5}}a_{0}+2J_{H}(d_{0}u)-2J_{G}(d_{0}a_{s})-2L_{\xi}(J_{0}a_{0})\end{array}\right]. 
\end{equation}
\end{lem}

Here, we need the Hermitian Yang-Mills condition because we need the pullback of the  connection to $(\mathbb{C}^{3}\setminus O)\times \mathbb{S}^{1}$ to be a projective $G_{2}-$instanton. Please see \cite[reduction from (32) to (33) using the projective instanton condition]{WangyuanqiJGP}. 
\begin{proof}[Proof of Lemma \ref{lem Bochner}:] We keep the identity $P^{2}+2P=B_{0,dim6}$ in mind throughout. Formula \eqref{equ formula for P2+4P separation} for $P^{2}+4P$ is obtained simply by adding twice of $P$  to \eqref{equ formula for P2+2P separation} (see formula \eqref{equ formula for P}). It suffices to prove \eqref{equ formula for P2+2P separation}.

 First of all, formula \eqref{formula decomposition for dstar S5} below for the operator $d_{\mathbb{S}^{5}}^{\star_{\mathbb{S}^{5}}}$ yields that row 3 of \eqref{equ formula for B0} is  equal to row 3 of the desired formula \eqref{equ formula for P2+2P separation}.

On row 4 of the desired formula \eqref{equ formula for P2+2P separation},
\begin{itemize}\item  formula \ref{formula 1 proof lem rough laplacian in eta and a0 component} below for the  rough Laplacian says that the $\eta$ component (co-efficient for $\eta$) of $\nabla^{\star}\nabla a_{0}$  is $[2d^{\star_{0}}_{0}J_{0}(a_{0})]\eta$;

\item  the $\eta-$component of $-2d_{s}a_{r}$ is  $-2\eta\wedge L_{\xi}a_{r}$, and that of $4a$ is apparently $4a_{\eta}\eta$;
\item Formula \ref{clm 3 proof lem rough laplacian in eta and a0 component} below says that the $\eta-$component of $\nabla^{\star}\nabla(\eta a_{\eta})$ is $4a_{\eta}+\nabla^{\star}\nabla a_{\eta}$.
\end{itemize}
The above facts amount to that the $\eta-$component of row 4 in \eqref {equ formula for B0}  is $$\eta[\nabla^{\star}\nabla a_{\eta}-2L_{\xi}a_{r}+8a_{\eta}+2d_{0}^{\star_{0}}J_{0}(a_{0})],$$
which exactly gives row 4 of the desired formula \eqref{equ formula for P2+2P separation} as co-efficient of $\eta$. 

For row 5, by  similar idea, we only have to observe that Formula \ref{clm 3 proof lem rough laplacian in eta and a0 component} below also says that the $\pi_{5,4}^{\star}(D^{\star}\otimes adE)$ (semi-basic) component  of $\nabla^{\star}\nabla(\eta a_{\eta})$ is $-2J_{0}(d_{0}a_{\eta})$. This completes the proof of \eqref{equ formula for P2+2P separation}. \end{proof}



\section{Eigenvalues and eigenspaces of the operator $P$ on $\mathbb{S}^{5}$, and the proof of Theorem \ref{Thm 1} \label{sect proof of Thm 1}}

\subsection{Fourier expansion with respect to the Reeb vector field}
Given a holomorphic Hermitian triple $(E,h,A_{O})$ on $\mathbb{P}^{2}$, the purpose of this section is to show that any (sufficiently regular) endomorphism on $\mathbb{S}^{5}$ admits a global ``Fourier series" in terms of the trivializations $s_{-k}$ of  $\pi^{\star}_{5,4}O(-k)\rightarrow \mathbb{S}^{5}$ (Lemma \ref{lem global F series}). The ``Fourier"-co-efficient of
 $s_{-k}$ is a section of $(EndE)(k)$ over $\mathbb{P}^{2}$.

The series is important in characterizing the vector spaces $V_{l}$ (see Section \ref{sect characterizing invariant subspaces} below),  and in the spectral reduction from $\mathbb{S}^{5}$ to $\mathbb{P}^{2}$ (see Formula \ref{formula laplace on S5 vs laplace on CP2} below).

This global Fourier-Series is equal to the usual Fourier-series in terms of $e^{\sqrt{-1}k\theta_{\beta}}$ defined (locally) in $U_{\beta,\mathbb{S}^{5}}$, for any $\beta$ among $0,\ -1,\ -2$.

Let $\nu\in C^{1}[S^{5},\pi^{\star}_{5,4}EndE]$. For any $\beta=0,\ 1$ or $2$, in $U_{\beta,\mathbb{S}^{5}}$,  the usual Fourier expansion 
\begin{equation}\label{equ local Fourier series}\nu=\Sigma_{k\in \mathbb{Z}}\nu_{\beta}(k)e^{\sqrt{-1}k\theta_{\beta}}\end{equation} 
converges uniformly in  $U_{\beta,\mathbb{S}^{5}}$ under the Hermitian metric (see Lemma \ref{lem term by term differentiation of F series} below). For any $k$, $v_{\beta}(k)$ is a section of $\pi^{\star}_{5,4}EndE\rightarrow U_{\beta,\mathbb{S}^{5}}$.


Because on the overlap $U_{\beta,\mathbb{S}^{5}}\cap U_{\alpha,\mathbb{S}^{5}}$, the function  $\frac{e^{\sqrt{-1}\theta_{\beta}}}{e^{\sqrt{-1}\theta_{\alpha}}}$ is equal to $\frac{Z_{\beta}}{Z_{\alpha}}\cdot \sqrt{\frac{\phi_{\beta}}{\phi_{\alpha}}}$, it is pulled back from the open set  $U_{\beta,\mathbb{P}^{2}}\cap U_{\alpha,\mathbb{P}^{2}}$ in $\mathbb{P}^{2}$. Because $\nu$ is globally defined,  given a $\beta$ and $\alpha$ among $0,1,2$, for any integer $k$,  on the overlap 
$U_{\beta,\mathbb{S}^{5}}\cap U_{\alpha,\mathbb{S}^{5}}$, the $k-$th terms of the Fourier-Series satisfy the following.  \begin{equation}\label{equ transition condition for v}\nu_{\beta}(k)e^{\sqrt{-1}k\theta_{\beta}}=\nu_{\alpha}(k)e^{\sqrt{-1}k\theta_{\alpha}}.\end{equation} 

Pointwisely, for any $[Y]\in \mathbb{P}^{2}$, let $(X_{0},X_{1},X_{2})\in \mathbb{C}^{3}\setminus O$ belong to the line  $[Y]$. Tautologically,  we define the dual basis  $(X_{0},X_{1},X_{2})^{\vee}$  as the functional on $[Y]$ whose value at 
$(X_{0},X_{1},X_{2})$ is equal to $1$. Then 
\begin{equation}(\lambda X_{0},\lambda X_{1}, \lambda X_{2})^{\vee}=\frac{1}{\lambda}(X_{0},X_{1},X_{2})^{\vee}.
\end{equation}

\begin{Def}\label{Def section of the pullback universal bundle}  Let 
\begin{equation}
s_{-1}\triangleq (X_{0},X_{1},X_{2})
\end{equation}
be the standard unitary trivialization of $\pi^{\star}_{5,4}[O(-1)]\rightarrow \mathbb{S}^{5}$. Then for any integer $l$, 
\begin{equation}\label{equ def SObeta} s_{l}\triangleq \left\{ \begin{array}{c}s_{-1}^{\otimes -l}\ \textrm{when}\ l\leq 0\\
s_{-1}^{\vee,\otimes l}\ \textrm{when}\ l>0 \end{array} \right. \end{equation}
is the unitary trivialization of $\pi^{\star}_{5,4}[O(l)]\rightarrow \mathbb{S}^{5}$ with respect to the standard metric (see Definition \ref{Def standard Hermitian metric on Obeta}). 

We understand $s_{-1}^{\otimes 0}$ and ${[s_{-1}^{\vee}]}^{\otimes 0}$ as the constant function $1$  which trivializes the trivial rank $1$ complex  bundle on $\mathbb{S}^{5}$. 
\end{Def}

Let $s_{k}\otimes s_{-k}$ denote the section of the trivial complex line bundle $\pi^{\star}_{5,4} O(k)\otimes \pi^{\star}_{5,4} O(-k)$ over $\mathbb{S}^{5}$. Viewing $EndE$ as the tensor product between itself and $\pi^{\star}_{5,4} O(k)\otimes \pi^{\star}_{5,4} O(-k)$, we find for any fixed $k$ that 
\begin{equation}\label{equ local =global expression Fourier series}\nu_{\beta}(k)e^{\sqrt{-1}k\theta_{\beta}}= [\nu_{\beta}(k)e^{\sqrt{-1}k\theta_{\beta}}s_{k}]\otimes  s_{-k}.
\end{equation}

By the transition condition \eqref{equ transition condition for v}, the section $\nu_{k}$ of $EndE\rightarrow \mathbb{S}^{5}$ defined piece-wisely by  $\nu_{\beta}(k)e^{\sqrt{-1}k\theta_{\beta}}s_{k}$ on $U_{\beta,\mathbb{S}^{5}}$, is independent of the coordinate neighborhood $U_{\beta,\mathbb{S}^{5}}$ chosen. Moreover, $\nu_{k}$ is independent of $\theta_{\beta}$ in $U_{\beta}$, thus it descends to $\mathbb{P}^{2}$. This means the section $\nu_{k}$  is a globally defined section pulled back from $\mathbb{P}^{2}$. We then have the global Fourier series. 
\begin{equation}\label{equ local =global expression Fourier series}\nu=\Sigma_{k\in \mathbb{Z}}\nu_{k}\otimes s_{-k}.
\end{equation}
The negative sign in ``$-k$" is to be consistent with the usual local Fourier-Series in \eqref{equ local Fourier series} i.e. for any integer $k$, and any $\beta$ among $0,1,2$, the following is true in $U_{\beta,\mathbb{S}^{5}}$. 
\begin{equation}\label{equ local F series is equal to the global one}\nu_{\beta}(k)e^{\sqrt{-1}k\theta_{\beta}}=\nu_{k}\otimes  s_{-k}.
\end{equation}
\begin{Def}\label{Def Sasaki Fourier}Henceforth, the series \eqref{equ local =global expression Fourier series} is called the \textit{Sasakian-Fourier series} of $\nu$. This is not the same object as the eigen Fourier expansion in Definition \ref{Def eigen F} below.
\end{Def}
We are ready for the main lemma of this section. 
\begin{lem}\label{lem global F series} Given a holomorphic Hermitian triple $(E,h,A_{O})$ on $\mathbb{P}^{2}$,  let $\nu\in C^{10}(\mathbb{S}^{5},\pi^{\star}_{5,4}EndE)$. For any $k\in \mathbb{Z}$, there is an unique section $\nu_{k}$ of $(EndE)(k)$ such that the following holds.

Under the pullback Hermitian metric, the Sasaki-Fourier series $\Sigma_{k}\nu_{k}\otimes s_{-k}$ converges uniformly to $\nu$ on $\mathbb{S}^{5}$. Moreover,  this series can be differentiated term by term by the Reeb Lie derivative $L_{\xi}$ and the rough Laplacian $\nabla^{\star}\nabla$ i.e. 
\begin{itemize}\item  $\Sigma_{k}L_{\xi}(\nu_{k}\otimes s_{-k})$ is the Sasaki-Fourier series of  $L_{\xi}\nu$, and converges uniformly on $\mathbb{S}^{5}$ to $L_{\xi}\nu$.
\item  $\Sigma_{k}\nabla^{\star}\nabla(\nu_{k}\otimes s_{-k})$ is the Sasaki-Fourier series of  $\nabla^{\star}\nabla \nu$, and converges uniformly on $\mathbb{S}^{5}$ to $\nabla^{\star}\nabla \nu$.
\end{itemize}

The Sasaki-Fourier co-efficient $\nu_{k}$ is $(End_{0}E)-$valued if $\nu$ is. 
\end{lem}

Please see the remark above Claim \ref {clm derivative of Fourier term is = Fourier term of derivative} that not every operator can differentiate 
the Sasaki-Fourier Series term by term.
\begin{proof}[Proof of Lemma \ref{lem global F series}:] It is a straight forward combination of the uniform convergence in Lemma \ref{lem term by term differentiation of F series} and the term by term-wise differentiation in Claim \ref {clm derivative of Fourier term is = Fourier term of derivative} below. 
\end{proof}

\subsection{Characterizing some $P-$invariant subspaces of sections on $\mathbb{S}^{5}$ by sheaf cohomologies on $\mathbb{P}^{2}$\label{sect characterizing invariant subspaces}}
In this sub-section, we study the special class of eigensections of the operator $P$ consisted of $\pi^{\star}_{5,4}(adE)-$valued semi-basic $1-$forms i.e. an eigensection of which the first $4$ endomorphism components are $0$ (regarding the decomposition in \eqref{equ  5 element basis}). These turn out to be a ``building-block" of $SpecP$ (as in Theorem \ref{Thm 1} above, also see Theorem \ref{Thm eigenvalue Sasakian}  below).
\begin{Def}Let $V_{l}\triangleq \{a_{0}\in C^{10}[\mathbb{S}^{5}, D^{\star}\otimes \pi^{\star}_{5,4}adE]|Pa_{0}=l a_{0}\}$. This means  \begin{eqnarray}\label{equ eigenspace}& &V_{l}
\\&= & \{a_{0}\in C^{10}[\mathbb{S}^{5}, D^{\star}\otimes \pi^{\star}_{5,4}adE]| d_{0}a_{0}\lrcorner H=d_{0}a_{0}\lrcorner G=d_{0}^{\star_{0}}a_{0}=d_{0}a_{0}\lrcorner \frac{d\eta}{2}=0,\nonumber \\ & & L_{\xi}(J_{0}a_{0})=-l a_{0}\},\nonumber
\\&= & \{a_{0}\in C^{10}[\mathbb{S}^{5}, D^{\star}\otimes \pi^{\star}_{5,4}adE]| d^{\star_{0}}_{0}J_{H}(a_{0})=d^{\star_{0}}_{0}J_{G}(a_{0})=d_{0}^{\star_{0}}a_{0}=d^{\star_{0}}_{0}J_{0}(a_{0})=0,\nonumber \\ & & L_{\xi}(J_{0}a_{0})=-l a_{0}\}.\nonumber
\end{eqnarray} 
\end{Def}
The above definition says that $V_{l}$ is a subspace of the eigenspace $\mathbb{E}_{l}P$. Elliptic regularity implies that any $a_{0}\in V_{l}$ is smooth. 

This subsection is devoted to the proof of the following characterization of $V_{l}$.
\begin{prop}\label{cor eigenspace and harmonic forms} Given a holomorphic Hermitian triple on $\mathbb{P}^{2}$,  for any integer $l$,  the sub space $V_{l}$ (of the eigenspace) is isomorphic to $$\mathcal{H}^{0,1}[\mathbb{P}^{2}, (End_{0}E)(l)]\ \textrm{(space of}\  \bar{\partial}-\textrm{harmonic forms}).$$ 
 Consequently, with respect to the complex structure $J_{0}$,  $V_{l}$ is complex isomorphic to the sheaf cohomology $H^{1}[\mathbb{P}^{2}, (End_{0}E)(l)]$.
\end{prop}

To prove the above proposition and for other purposes, it is useful to set the following convention. 
\begin{Notation} In conjunction with Notation Convention \ref{Notation = I}, 
\label{Notation = II} 
when the two vector spaces are complex vector spaces of sections of a complex vector bundle (like the twisted endomorphism bundles), or when they are sheaf cohomologies etc,  the ``$=$" means a complex isomorphism. Otherwise, to say it is a complex isomorphism, the complex structure should be specified in a manner similar to Proposition \ref{cor eigenspace and harmonic forms}. 
\end{Notation}
\subsubsection{The two term Sasaki-Fourier series for elements in $V_{l}$}

We decompose any $\pi_{5,4}^{\star}adE-$valued semi-basic $1-$form $a_{0}$ into the $(1,0)$ and $(0,1)-$components
\begin{equation}\label{equ 10 01 decomposition}
a_{0}=a^{1,0}_{0}+a^{0,1}_{0}.
\end{equation} 
Then the  condition \begin{equation}\label{equ condition Lxi eigen beta} -L_{\xi}J_{0}(a_{0})=l a_{0}\  (\textrm{which is part of}\ \eqref{equ eigenspace})\end{equation} yields that
\begin{equation*}
l a^{1,0}_{0}+l a^{0,1}_{0}=l a_{0}=-L_{\xi}J_{0}(a_{0})=-L_{\xi}(-\sqrt{-1}a^{1,0}_{0}+\sqrt{-1}a^{0,1}_{0})=\sqrt{-1}L_{\xi}a^{1,0}_{0}-\sqrt{-1}L_{\xi}a^{0,1}_{0}.
\end{equation*} 
Comparing $(1,0)$ and $(0,1)-$part of both sides, we find 
\begin{equation}\label{equ 1 proof Thm Sasakian eigen}
L_{\xi}a^{0,1}_{0}=\sqrt{-1}l a^{0,1}_{0},\ L_{\xi}a^{1,0}_{0}=-\sqrt{-1}l a^{1,0}_{0}.
\end{equation}
Consider the Fourier-expansions 
\begin{equation}\label{equ 1.5 proof Thm Sasakian eigen}
a^{0,1}_{0}=\Sigma_{k}a^{0,1}_{0}(k)s_{-k},\ a^{1,0}_{0}=\Sigma_{k}a^{1,0}_{0}(k)s_{-k}.
\end{equation} 
where the summations are over all integers, and each coefficient is a $End_{0}E-$valued $1-$form pulled back from $\mathbb{P}^{2}$.

Since 
\begin{equation}\label{equ Lie derivative of s-k}L_{\xi}s_{-k}=\sqrt{-1}ks_{-k},\end{equation}  the following holds.
\begin{equation}\label{equ 2 proof Thm Sasakian eigen}
L_{\xi}a^{0,1}_{0}=\Sigma_{k}\sqrt{-1}k a^{0,1}_{0}(k)s_{-k},\ \  L_{\xi}a^{1,0}_{0}=\Sigma_{k}\sqrt{-1}k a^{1,0}_{0}(k)s_{-k}.
\end{equation}
Compare the Sasaki-Fouriercoefficients of the first equation in \eqref{equ 2 proof Thm Sasakian eigen} with the first identity 
in \eqref{equ 1 proof Thm Sasakian eigen}, we find that the eigenvalue $l$ must be an integer, and $a^{0,1}_{0}(k)=0$ if $k\neq l$. Consequently, 
$$a^{0,1}_{0}=a^{0,1}_{0}(l) s_{-l}.\ \ \ \textrm{Similarly}, \ a^{1,0}_{0}=a^{1,0}_{0}(-l) s_{l}.$$    In summary, we have found

\begin{clm}\label{clm Fourier expansion of eigensection} Suppose $a_{0}\in C^{10}[\mathbb{S}^{5}, D^{\star}\otimes \pi_{5,4}^{\star}adE]$  satisfies equation \eqref{equ condition Lxi eigen beta} and $a_{0}\neq 0$, then the $l$ therein is an integer. Moreover, in view of the $(1,0)\oplus (0,1)-$decomposition \eqref{equ 10 01 decomposition}, we have   \begin{equation}\label{equ clm Fourier expansion of eigensection} a^{1,0}_{0}=c^{1,0} s_{l}+c^{0,1} s_{-l},\end{equation} where $c^{1,0}$ is an  $(End_{0}E)(-l)-$valued  $(1,0)-$form on $\mathbb{P}^{2}$, and $c^{0,1}$ is an  $(End_{0}E)(l)-$valued $(0,1)-$form on $\mathbb{P}^{2}$.
\end{clm}
The above claim particularly means that there are only two non-zero terms in the Sasaki-Fourier series of $a_{0}$ (if it satisfies  \eqref{equ condition Lxi eigen beta}).

\subsubsection{Equivalence between the conditions on $d_{0}b$ and that $b^{0,1}$ is $\bar{\partial}_{0}-$harmonic\label{sect Equivalence of the dstar closeness and dbar harmonicity}}

The purpose of this section is to prove Lemma \ref{lem d0 formulas} on the equivalent characterization of the $4-$different ``$d^{\star_{0}}_{0}-$closeness" in the defining conditions of \eqref{equ eigenspace}. 

We need the following simple algebraic fact. The $2-$form $H-\sqrt{-1}G=\Theta \in \Lambda^{2,0}D^{\star,\mathbb{C}}$ is  nowhere vanishing, and the complex dimension of $D^{\star,(1,0)}$ is $2$. Then the following holds elementarily. 
\begin{fact}\label{fact 0,2 and 2,0 contraction}Let $p\in \mathbb{S}^{5}$,  $\theta_{1}\in \wedge^{(0,2)}D^{\star,\mathbb{C}}|_{p}$, and $\theta_{2}\in \wedge^{(2,0)}D^{\star,\mathbb{C}}|_{p}$. 
\begin{itemize}\item $\theta_{1}\lrcorner (H-\sqrt{-1}G)=0$ at $p$ if and only if $\theta=0$ at $p$. 

\item $\theta_{2}\lrcorner (H+\sqrt{-1}G)=0$ at $p$ if and only if $\theta=0$ at $p$.
\end{itemize}

\end{fact}
Again, because $\Theta$ is $(2,0)$, the contraction between any form in  $\wedge^{(2,0)}D^{\star,\mathbb{C}}$ with\\ $\Theta=H-\sqrt{-1}G$ is automatically $0$. The same holds for the contraction between any form in  $\wedge^{(0,2)}D^{\star,\mathbb{C}}$ with $\overline{\Theta}=H+\sqrt{-1}G$.

In our convention, 
the Hodge star $\star_{0}$ is extended complex linearly. On semi-basic $1-$forms, we define \begin{equation}\label{equ def of ad operators} d_{0}^{\star_{0}}\triangleq-\star_{0} d_{0} \star_{0},\ \ \partial_{0}^{\star_{0}}\triangleq-\star_{0} \partial_{0} \star_{0};\ \ \bar{\partial}_{0}^{\star_{0}}\triangleq-\star_{0} \bar{\partial}_{0} \star_{0}.
\end{equation}
Then the adjoint  $\bar{\partial}_{0}$ with respect to the Hermitian inner-product is $\partial_{0}^{\star_{0}}$, and that of  $\partial_{0}$  is $\bar{\partial}_{0}^{\star_{0}}$ (cf. the other notation convention in \cite{GH}).

\begin{lem}\label{lem d0 formulas}Given a holomorphic Hermitian triple on $\mathbb{P}^{2}$, let $b$ be a smooth section of  $D^{\star}\otimes\pi_{5,4}^{\star}adE\rightarrow \mathbb{S}^{5}$, the following holds true.
\begin{equation}\label{equ lem do formulas H and G}d_{0}b\lrcorner H=d_{0}b\lrcorner G=0\ \Leftrightarrow\  \bar{\partial}_{0}b^{0,1}=
\partial_{0}b^{1,0}=0;
\end{equation}
\begin{equation}\label{equ lem do formulas deta and dstar}d_{0}^{\star_{0}}b=d_{0}b\lrcorner \frac{d\eta}{2}=0\ \Leftrightarrow\ \partial_{0}^{\star_{0}}b^{0,1}=\bar{\partial}_{0}^{\star_{0}}b^{1,0}=0.
\end{equation}
\end{lem}
\begin{proof}[Proof of Lemma \ref{lem d0 formulas}:] It is by simple and routine comparison of types of the forms. For the reader's convenience, we still provide the detail.  
As follows, we can split $d_{0}b$  into $(0,2)$, $(2,0)$, and $(1,1)$ components. 
 \begin{equation}
d_{0}b=\bar{\partial}_{0}b^{0,1}+
\partial_{0}b^{1,0}
+(\partial_{0}b^{0,1}+
\bar{\partial}_{0}b^{1,0}). \label{equ formula for d0}
\end{equation}
 The form $H-\sqrt{-1}G$ is $(2,0)$, thus among the $4$ terms in \eqref{equ formula for d0}, only the $(0,2)-$component $\bar{\partial}_{0}b^{0,1}$ might have non-zero contraction with $H-\sqrt{-1}G$, the contraction between each of the other $3$ terms and  $H-\sqrt{-1}G$ vanishes. Combining the  vanishing criteria in Fact \ref{fact 0,2 and 2,0 contraction}, we find
\begin{equation}d_{0}b\lrcorner (H-\sqrt{-1}G)=0 \Leftrightarrow \bar{\partial}_{0}b^{0,1}=0.
\end{equation}
Similarly, because $H+\sqrt{-1}G$ is $(0,2)$, the following holds true.
\begin{equation}d_{0}b\lrcorner (H+\sqrt{-1}G)=0 \Leftrightarrow \partial_{0}b^{1,0}=0.
\end{equation}
The proof of \eqref{equ lem do formulas H and G} is complete. 

To prove \eqref{equ lem do formulas deta and dstar}, we first observe that 
\begin{equation}\label{equ 0 proof lem harmonic form}
\partial^{\star_{0}}_{0}b^{1,0}=0,\  \ \bar{\partial}^{\star_{0}}_{0}b^{0,1}=0. 
\end{equation}
To prove the first identity in \eqref{equ 0 proof lem harmonic form}, it suffices to notice that $\star_{0}b^{1,0}$ is a $(2,1)-$form and the complex dimension of $D^{\star, (1,0)}$ (and of $D^{\star, (0,1)}$) is $2$, then $\partial_{0} \star_{0}b^{1,0}=0$. The proof for the other one is similar. Therefore
\begin{equation}\label{equ 1 proof lem harmonic form}d^{\star_{0}}_{0}b=\partial^{\star_{0}}_{0}b^{0,1}+
\bar{\partial}^{\star_{0}}_{0}b^{1,0}.
\end{equation}

Contracting $d_{0}b$ with $\frac{d\eta}{2}$ using the decomposition \eqref{equ formula for d0}, still using the vanishing in \eqref{equ 0 proof lem harmonic form},  we find 
\begin{equation}d_{0}b\lrcorner \frac{d\eta}{2}=d_{0}^{\star_{0}}J_{0}b=d_{0}^{\star_{0}}(\sqrt{-1}b^{0,1}-\sqrt{-1}b^{1,0})=\sqrt{-1}\partial_{0}^{\star_{0}}b^{0,1}-\sqrt{-1}\bar{\partial}_{0}^{\star_{0}}b^{1,0}.\label{equ 2 proof lem harmonic form}
\end{equation}
Via the two different identities \eqref{equ 1 proof lem harmonic form} and \eqref{equ 2 proof lem harmonic form}, the condition $d^{\star_{0}}_{0}b=d_{0}b\lrcorner \frac{d\eta}{2}=0$ is equivalent to that 
$\sqrt{-1}\partial_{0}^{\star_{0}}b^{0,1}=0=\sqrt{-1}\bar{\partial}_{0}^{\star_{0}}b^{1,0}$. The sign difference between \eqref{equ 1 proof lem harmonic form} and \eqref{equ 2 proof lem harmonic form} (caused by the complex structure $J_{0}$) is crucial. The proof for \eqref{equ lem do formulas deta and dstar} is  complete. \end{proof}

\subsubsection{$d_{0}-$parallel of the trivialization of $\pi^{\star}_{5,4}O(-1)\rightarrow \mathbb{S}^{5}$, and the proof of Proposition \ref{cor eigenspace and harmonic forms}}

The purpose of this section is to show that the map sending $a_{0}$ to $c^{0,1}$ (see Claim \ref{clm Fourier expansion of eigensection}) is the desired isomorphism in Proposition \ref{cor eigenspace and harmonic forms} (identifying $V_{l}$ to the space of $\bar{\partial}-$harmonic $(0,1)$ $End_{0}E-$valued forms on $\mathbb{P}^{2}$). 

We need the following Lemma saying that the section $s_{-l}$ of $\pi^{\star}_{5,4}O(l)\rightarrow \mathbb{S}^{5}$ is $d_{0}-$closed (parallel). It is not parallel under the (full) connection on $\pi^{\star}_{5,4}O(l)$ unless $l=0$. 
\begin{lem}\label{lem d0 parallel section of pullback O(-1)}Under the pullback standard connection on $\pi^{\star}_{5,4}[O(l)]\rightarrow \mathbb{S}^{5}$, the standard trivialization $s_{l}$ (see \eqref{equ def SObeta} and Definition \ref{Def section of the pullback universal bundle}) is $d_{0}-$closed i.e.
$$\partial_{0}s_{l}=\bar{\partial}_{0}s_{l}=0,\ \ \textrm{and}\ d_{0}s_{l}=0.$$
\end{lem}  
\begin{proof}[Proof of Lemma \ref{lem d0 parallel section of pullback O(-1)}:] We only prove it when $l=-1$. When $l=1$, it follows by dualizing. For arbitrary integer $l$, it follows by Leibniz-rule with respect to tensor product. 

It suffices to prove $\partial_{0}s_{-1}=\bar{\partial}_{0}s_{-1}=0$ in $U_{0,\mathbb{S}^{5}}$. The vanishing on the whole $\mathbb{S}^{5}$ follows by continuity. 

First, we routinely verify the following. 
\begin{formula}\label{formula d0 of Z0} In $U_{0,\mathbb{S}^{5}}$, $\partial_{0}Z_{0}=-Z_{0}\partial_{\mathbb{P}^{2}}\log \phi_{0}$, and $ \bar{\partial}_{0}Z_{0}=0.$
\end{formula}
For the reader's convenience, we still give the proof of Formula \ref{formula d0 of Z0}. It is routine to verify the following two identities by  Formula \eqref{equ norm of Z0} for $Z_{0}$ (which particularly says $L_{\xi}Z_{0}=\sqrt{-1}Z_{0}$), Fact \ref{fact Reeb is angular} on the Reeb vector field, and Formula \eqref{equ d0 dbar0 partial0} for $\partial_{0},\ \bar{\partial}_{0}$ etc. 
\begin{eqnarray}& &\partial_{0}Z_{0}=\partial_{\mathbb{P}^{2}}Z_{0}+\frac{\sqrt{-1}}{2}(\partial_{\mathbb{P}^{2}}\log\phi_{0})\wedge (L_{\xi}Z_{0})=-\frac{Z_{0}}{2}\partial_{\mathbb{P}^{2}}\log\phi_{0}-\frac{Z_{0}}{2}\partial_{\mathbb{P}^{2}}\log\phi_{0}\nonumber
\\&=&-Z_{0}\partial_{\mathbb{P}^{2}}\log\phi_{0},\nonumber
\\& &\bar{\partial}_{0}Z_{0}=\bar{\partial}_{\mathbb{P}^{2}}Z_{0}-\frac{\sqrt{-1}}{2}(\bar{\partial}_{\mathbb{P}^{2}}\log\phi_{0})\wedge (L_{\xi}Z_{0})=-\frac{Z_{0}}{2}\bar{\partial}_{\mathbb{P}^{2}}\log\phi_{0}+\frac{Z_{0}}{2}\bar{\partial}_{\mathbb{P}^{2}}\log\phi_{0}\nonumber
\\&=&0.\nonumber
\end{eqnarray}
We continue proving Lemma \ref{lem d0 parallel section of pullback O(-1)}. In $U_{0,\mathbb{S}^{5}}$, the trivialization $(1,u_{1},u_{2})$ of $\pi^{\star}_{5,4}O(-1)$ descends to $U_{0,\mathbb{P}^{2}}$. Then, under the Chern-connection of the standard metric, we find
\begin{equation}
\bar{\partial}_{0}(1,u_{1},u_{2})=\bar{\partial}_{\mathbb{P}^{2}}(1,u_{1},u_{2})=0,
\end{equation}
and
\begin{equation}
\partial_{0}(1,u_{1},u_{2})=\partial_{\mathbb{P}^{2}}(1,u_{1},u_{2})=(\partial_{\mathbb{P}^{2}}\log\phi_{0})(1,u_{1},u_{2}).
\end{equation}
Therefore,
 \begin{equation}\bar{\partial}_{0}(Z_{0},Z_{1},Z_{2})=\bar{\partial}_{0}[Z_{0}(1,u_{1},u_{2})]=(\bar{\partial}_{0}Z_{0})(1,u_{1},u_{2})+Z_{0}[\bar{\partial}_{\mathbb{P}^{2}}(1,u_{1},u_{2})]=0.
\end{equation}
\begin{eqnarray}& & \partial_{0}(Z_{0},Z_{1},Z_{2})=\partial_{0}[Z_{0}(1,u_{1},u_{2})]=(\partial_{0}Z_{0})(1,u_{1},u_{2})+Z_{0}[\partial_{\mathbb{P}^{2}}(1,u_{1},u_{2})]\nonumber
\\&=&-Z_{0}\partial_{\mathbb{P}^{2}}\log\phi_{0}+Z_{0}\partial_{\mathbb{P}^{2}}\log\phi_{0} \nonumber
\\&=& 0.
\end{eqnarray}
The proof is complete. \end{proof}

To extend the usual conjugate transpose of endomorphisms  to twisted endomorphisms,  the following convention helps, for example,  in the proof of  Proposition \ref{cor eigenspace and harmonic forms} and Proposition \ref{prop inj not integer} below. \begin{Notation} For any integer $k$, let $^{\rule{0.5cm}{0.1mm}}$ denote the conjugate linear map from $\pi^{\star}_{5,4}O(k)$ to $\pi^{\star}_{5,4}O(-k)$ defined by $$\overline{s}_{k}\triangleq s_{-k}.$$
Let it applies distributively to a tensor product of the line bundles.

Let $(E,h,A_{O})$ be a holomorphic Hermitian triple on $\mathbb{P}^{2}$, it is obvious that we can take conjugate transpose $^{{\rule{0.5cm}{0.1mm}}_{t}}$ of any endomorphism. Using the above conjugation, let the transpose only applies to the $EndE-$part but not the line bundle part, we can also take $^{{\rule{0.5cm}{0.1mm}}_{t}}$ of any twisted endomorphism. 

This is why we only work with endomorphisms and twisted endomorphisms. 

Under the  identification of the following $3$ objects: 
\begin{itemize} \item $1$ as a section of the trivial line bundle, 
\item $s_{k}\otimes s_{-k}$ as a section of  $[\pi^{\star}_{5,4}O(k)]\otimes [\pi^{\star}_{5,4}O(-k)]$,
\item  and $s_{-k}\otimes s_{k}$ as a section of  $[\pi^{\star}_{5,4}O(-k)]\otimes [\pi^{\star}_{5,4}O(k)]$,
\end{itemize}
 the following diagram commutes (where the vertical maps are the conjugation). 
\begin{equation}\label{equ tikzpicture 1}  \begin{tikzpicture}
 \node at (-2.7, 0.5) {$1$};
    \node at (-5.8, 0.5) {$s_{k}\otimes s_{-k}$};
     \node at (0.3, 0.5) {$s_{-k}\otimes s_{k}$};
\draw[->,semithick] (-2.7,1.7) -- (-2.7,0.7);
\draw[<-,semithick] (-5,2) -- (-2.9,2);
\draw[<-,semithick] (-5,0.5) -- (-2.9,0.5);
\draw[->,semithick] (-2.6,2) -- (-0.5,2);
\draw[->,semithick] (-2.6,0.5) -- (-0.5,0.5);
\draw[->,semithick] (-5.8,1.8) -- (-5.8,0.7);
\draw[->,semithick] (0.3,1.8) -- (0.3,0.7);
\node at (-2.7,2) {$1$} ;
  \node  at (-5.8,2.1) {$s_{-k}\otimes s_{k}$};
   \node at (0.3, 2.1) {$s_{k}\otimes s_{-k}$};
     \node at (0.7, 1.3) {$^{\rule{0.5cm}{0.1mm}}$};
         \node at (-2.2, 1.3) {$^{\rule{0.5cm}{0.1mm}}$};
             \node at (-6.2, 1.3) {$^{\rule{0.5cm}{0.1mm}}$};
\end{tikzpicture}.
\end{equation}
Therefore, the conjugate transpose of a  Fourier-Series of local form (left side of \eqref{equ local F series is equal to the global one}) is equal to that of a series of global form (right side of \eqref{equ local F series is equal to the global one}). 
\end{Notation}

The conventions and equations established so far are at our disposal to characterize the  subspaces $V_{l}$.
\begin{proof}[\textbf{Proof of Proposition} \ref{cor eigenspace and harmonic forms}:]  We show that the map 
\begin{equation}\Gamma(b)\triangleq   b^{0,1}s_{l}
\end{equation}
is the desired isomorphism $V_{l}\rightarrow \mathcal{H}^{0,1}[\mathbb{P}^{2}, (End_{0}E)(l)]$. Because $J_{0}d^{0,1}=\sqrt{-1}d^{0,1}$ for any ($\pi^{\star}_{5,4}End_{0}E-$valued) $(0,1)$ semi-basic form $d^{0,1}$, the $\Gamma$ above is complex linear. 

Let $a_{0}=b$ in Claim \ref{clm Fourier expansion of eigensection} , $\Gamma(b)$ is precisely the ``$c^{0,1}$" in the splitting  \eqref{equ clm Fourier expansion of eigensection}. 

Step 1: $\Gamma(b)$ is a priori a $(EndE)(l)$-valued $(0,1)-$form.  We show that it is $\bar{\partial}_{\mathbb{P}^{2}}-$harmonic. By the $\partial_{0}$ and $\bar{\partial}_{0}-$closeness of $s_{l}$ in Lemma \ref{lem d0 parallel section of pullback O(-1)}, we compute
\begin{equation}
\bar{\partial}_{0}(b^{0,1}s_{l})=(\bar{\partial}_{0}b^{0,1})s_{l}+b^{0,1}(\bar{\partial}_{0}s_{l})=0,\ \end{equation}and 
\begin{eqnarray}& &\partial^{\star_{0}}_{0}(b^{0,1}s_{l})=-\star_{0}\partial_{0}[s_{l}\star_{0}(b^{0,1})]\nonumber
=-\star_{0}(\partial_{0}s_{l})\wedge\star_{0}(b^{0,1})-s_{l}[\star_{0}\partial_{0}\star_{0}(b^{0,1})]\nonumber
\\&=&-s_{l}[\star_{0}\partial_{0}\star_{0}(b^{0,1})]=s_{l}\partial_{0}^{\star_{0}}b^{0,1}\nonumber
\\&=&0.
\end{eqnarray}

 Because $b^{0,1}s_{l}$ descends to $\mathbb{P}^{2}$, the above vanishing implies that 
 $$\bar{\partial}_{\mathbb{P}^{2}}(b^{0,1}s_{l})=\partial^{\star_{\mathbb{P}^{2}}}_{\mathbb{P}^{2}}(b^{0,1}s_{l})=0\ i.e.\ \Gamma(b)\in  \mathcal{H}^{0,1}[\mathbb{P}^{2},O(l)\otimes End_{0}E].$$  

Step 2:  We show that the following map $\underline{\Gamma}: \mathcal{H}^{0,1}[\mathbb{P}^{2},O(l)\otimes End_{0}E]\rightarrow V_{l}$  is the (two-sided) inverse of $\Gamma$.
\begin{equation}\label{equ 0 proof eigenspace and harmonic form}\underline{\Gamma}(d^{0,1})\triangleq  d^{0,1}\otimes s_{-l}-\overline{[d^{0,1}\otimes s_{-l}]}^{t},
\end{equation}
where $d^{0,1}\otimes s_{-l}$ is viewed as an $\pi^{\star}_{5,4}(End_{0}E)-$valued semi-basic $1-$form on $\mathbb{S}^{5}$.

We notice that $\alpha$ is a $\pi^{\star}_{5,4}adE-$valued semi-basic $1-$form if and only if $\alpha^{1,0}=-\overline{\alpha^{0,1}}^{t}$ (cf. \cite[(2.15) VII]{Kobayashi}). Hence $\Gamma$ is injective. By \eqref{equ 0 proof eigenspace and harmonic form}, $\Gamma \underline{\Gamma}=Id$ automatically holds true. 
Then for any $b\in V_{l}$, we find $(\underline{\Gamma}\Gamma)b-b\in Ker \Gamma=\{0\}$. The injectivity says $(\underline{\Gamma} \Gamma )b-b=0$. Then $\underline{\Gamma} \Gamma =Id$, and $\underline{\Gamma}$ is the two-sided inverse of $\Gamma$. The complex linearity of $\Gamma$ then gives the complex linearity of $\underline{\Gamma}$.

Then proof is complete. \end{proof}
\subsection{Describing $SpecP$: proof of Theorem \ref{Thm 1}.$\mathbb{I}$}
\subsubsection{The orthogonal complement of the eigen cohomology space}

Again, we stress that $V_{l}$ is a subspace of the eigenspace $\mathbb{E}_{l}P$.

Based on the above, we first restrict  the operator $P$ onto a closed subspace of $L^{2}(\mathbb{S}^{5},Dom_{\mathbb{S}^{5}})$ of finite co-dimension. 
\begin{Def}\label{Def Vcoh}In view of the invariant subspace $V_{l}$ that is isomorphic to the cohomology (Proposition \ref{cor eigenspace and harmonic forms}), we define the \textit{eigen cohomology space} $V_{coh}\triangleq \oplus_{l\in \mathbb{Z}}V_{l}$, which is isomorphic to the finite dimensional vector space $\oplus_{l}H^{1}[\mathbb{P}^{2}, (End_{0}E)(l)]$. We view 
 $V_{coh}$ as a subspace of the Hilbert space $L^{2}(\mathbb{S}^{5},Dom_{\mathbb{S}^{5}})$. 
 \end{Def}
 \begin{rmk}\label{rmk specPvcoh=scoh} Apparently, $SpecP|_{V_{coh}}=S_{coh}$. Moreover, any element in $L^{2}(\mathbb{S}^{5},Dom_{\mathbb{S}^{5}})$ with vanishing 5th row (bottom entry) lies in $V^{\perp}_{coh}$.
 \end{rmk}
 Because $P$ is self-adjoint, the orthogonal complement $V^{\perp}_{coh}$, a closed subspace of finite co-dimension,  is still $P-$invariant. 
We let  $P|_{V^{\perp}_{coh}}$ denote the restriction of $P$ to $V^{\perp}_{coh}$. It is  convenient to work with   $P|_{V^{\perp}_{coh}}$.

\subsubsection{The crucial intermediate theorem}
The theory for the $P-$invariant spaces $V_{l}$ is at our disposal to show the following result, which is a crucial building block for Theorem \ref{Thm 1}.$\mathbb{I}$. 
\begin{thm}\label{Thm eigenvalue Sasakian}Given a Hermitian Yang-Mills triple on $\mathbb{P}^{2}$, $Spec P|_{V^{\perp}_{coh}}=S_{\nabla^{\star}\nabla}.$

\end{thm}

The above theorem means the spectrum of $P$ on $V^{\perp}_{coh}$ is exactly the part induced by the spectrum of the rough Laplacian. Before proving it, to intuitively understand the feature of the Bochner formulas of $P$ (Lemma \ref{lem Bochner}), we introduce the notion of ``autonomous". 
\begin{Def} For any $i=1,\ 2,\ 3,$ or $4$,

\begin{itemize}\item  let $Row^{i}$ denote the injection from $\pi^{\star}_{5,4} adE$ (or $\pi^{\star}_{7,4} adE$) to the $i$-th row of $Dom_{\mathbb{S}^{5}}$ ($Dom_{7}$).

\item Let $v_{i}$ be the $i-$th variable (in the $i-$th row) of $Dom_{\mathbb{S}^{5}}$.
A linear differential operator $L$ on $Dom_{\mathbb{S}^{5}}$ is said to be autonomous with respect to  row $i$, if row $i$ of $L$ is $$(\nabla^{\star}\nabla+kId)v_{i}$$ for some real constant $k$, and the other rows does not depend on $v_{i}$.
\end{itemize}
\end{Def} 

Another ingredient we need is the spectrum counted with multiplicities. 
\begin{Def}\label{Def spec mul}Let the operator be $P,\ P|_{V^{\perp}_{coh}}$, or $\nabla^{\star}\nabla  |_{\mathbb{S}^{5}}$ (which means $\nabla^{\star}\nabla  |_{\pi_{5,4}^{\star}adE \rightarrow \mathbb{S}^{5}}$). Let  $Spec^{mul}(\ \cdot \ )$ denote the set of eigenvalues counted with real multiplicity. This means if $\mu$ is an eigenvalue and the real dimension of the eigenspace is  $m_{\mu}$, $\mu$ appears in $Spec^{mul}(\cdot)$ $m_{\mu}$ times.

Similarly, when the operator is $\nabla^{\star}\nabla  |_{\pi_{5,4}^{\star}End_{0}E \rightarrow \mathbb{S}^{5}}$ or $\nabla^{\star}\nabla|_{(End_{0}E)(l)\rightarrow \mathbb{P}^{2}}$,  let  $Spec^{mul_{\mathbb{C}}}(\ \cdot \ )$ denote the set of eigenvalues  counted with complex multiplicity. This means if $\mu$ is an eigenvalue and the complex dimension of the eigenspace is  $m_{\mu}$, $\mu$ appears in $Spec^{mul_{\mathbb{C}}}(\cdot)$ $m_{\mu}$ times.
   \end{Def}
   \begin{rmk}  The complex bundle $End_{0} E$ is the complexification of $ad E$. Hence, for any $\lambda\in Spec\nabla^{\star}\nabla  |_{\mathbb{S}^{5}}$, $\mathbb{E}_{\lambda} \nabla^{\star}\nabla  |_{\pi_{5,4}^{\star}End_{0}E \rightarrow \mathbb{S}^{5}}$ is the complexification of $$\mathbb{E}_{\lambda} \nabla^{\star}\nabla  |_{\pi_{5,4}^{\star}adE \rightarrow \mathbb{S}^{5}}\triangleq \mathbb{E}_{\lambda} \nabla^{\star}\nabla  |_{\mathbb{S}^{5}}.$$ Because the real dimension of a vector space is equal to the complex dimension of its complexification, the following holds. 
   
 \begin{equation}\label{equ specmul = spec mulC} Spec^{mul}\nabla^{\star}\nabla  |_{\pi_{5,4}^{\star}adE \rightarrow \mathbb{S}^{5}}=Spec^{mul_{\mathbb{C}}}\nabla^{\star}\nabla  |_{\pi_{5,4}^{\star}End_{0}E \rightarrow \mathbb{S}^{5}}. \end{equation}\end{rmk}

\subsubsection{Proving the crucial intermediate theorem}

To prove Theorem \ref{Thm eigenvalue Sasakian}, we need two simple facts   on irreducibility of the connection in a Hermitian Yang-Mills triple.  
\begin{fact}\label{fact stable=irred}(\cite[VII, Proposition 4.14]{Kobayashi}) Given a Hermitian Yang-Mills triple $(E,h,A_{O})$ on $\mathbb{P}^{2}$, $E$ is stable if and only if $A_{O}$ is irreducible on $\mathbb{P}^{2}$ (i.e. there is no non-zero parallel section of $adE$). 
\end{fact}
\begin{fact}\label{fact irred on P2 = irred on S5} In the setting of Fact \ref{fact stable=irred}, $A_{O}$ is irreducible on $\mathbb{P}^{2}$ if and only if $\pi^{\star}_{5,4}A_{O}$ is irreducible on $\mathbb{S}^{5}$.\end{fact}
Fact \ref{fact irred on P2 = irred on S5} is a direct corollary of the spectral decomposition in Formula \ref{formula laplace on S5 vs laplace on CP2} below.

\begin{proof}[\textbf{Proof of Theorem \ref{Thm eigenvalue Sasakian}}:] The crucial observation is that by the Bochner formulas in Lemma \ref{lem Bochner}, $P^{2}+2P$ is autonomous with respect to row $1,\ 2$,  and $P^{2}+4P$ is autonomous with respect to row $3,\ 4$.

Step 1: $Spec P|_{V^{\perp}_{coh}}\subseteq S_{\nabla^{\star}\nabla}$. 

\begin{Def}\label{Def eigen F}Let $\{\phi_{\mu},\ \mu\in Spec^{mul}(P|_{V^{\perp}_{coh}})\}$ be an eigenbasis with respect to $P|_{V^{\perp}_{coh}}$. The eigen  expansion of any $L^{2}-$section of $Dom_{\mathbb{S}^{5}}$ is called the \textit{$P|_{V^{\perp}_{coh}}-$eigen Fourier expansion}. Similar definition of  eigen Fourier expansion applies to the operator $P$ itself (no restriction) and the other rough Laplacians. \textit{This is not the same object as the Sasaki-Fourier series in Definition \ref{Def Sasaki Fourier} above}. 
\end{Def}

     Let $\lambda\in Spec(\nabla^{\star}\nabla|_{\mathbb{S}^{5}})$,  for any non-zero $u_{\lambda}\in \mathbb{E}_{\lambda}(\nabla^{\star}\nabla|_{\mathbb{S}^{5}})$, in view of Remark \ref{rmk specPvcoh=scoh}, we  consider the $P|_{V^{\perp}_{coh}}-$eigen Fourier expansion: $$\left[\begin{array}{c}u_{\lambda} \\ 0 \\  0 \\ 0 \\ 0 \end{array}\right]=\Sigma_{\mu\in Spec^{mul}(P|_{V^{\perp}_{coh}})}u_{\lambda,\mu}\phi_{\mu}.$$  Using the Bochner formula \eqref{equ formula for P2+2P separation},  we calculate that
\begin{eqnarray} \label{equ 0 proof Thm SpecP} \nonumber
& &\Sigma_{\mu\in Spec^{mul}(P|_{V^{\perp}_{coh}})}(\mu^{2}+2\mu )u_{\lambda, \mu}\phi_{\mu}= [P^{2}+2P]\left[\begin{array}{c}u_{\lambda} \\ 0 \\  0 \\ 0 \\ 0 \end{array}\right]
=\left[\begin{array}{c}(\nabla^{\star}\nabla u_{\lambda})+3u_{\lambda} \\ 0 \\  0 \\ 0 \\ 0 \end{array}\right] \nonumber
\\&=&(\lambda+3)\left[\begin{array}{c}u_{\lambda} \\ 0 \\  0 \\ 0 \\ 0 \end{array}\right]\nonumber\\&=&\Sigma_{\mu\in Spec^{mul}(P|_{V^{\perp}_{coh}})}(\lambda+3)u_{\lambda, \mu}\phi_{\mu}.
\end{eqnarray}
Comparing the non-zero coefficients $u_{\lambda,\mu}$ (which must exist because $u_{\lambda}\neq 0$),  we find 
$$\mu^{2}+2\mu =\lambda+3.$$ This means that for any $\lambda \in Spec(\nabla^{\star}\nabla|_{\mathbb{S}^{5}})$, 
$$Row^1[\mathbb{E}_{\lambda}(\nabla^{\star}\nabla|_{\mathbb{S}^{5}})]
 \subseteq  \oplus^{L^{2}}_{\mu,\mu^{2}+2\mu -3=\lambda}\mathbb{E}_{\mu}P|_{V^{\perp}_{coh}}.$$

Because $u_{\lambda}$ is an arbitrary non-zero eigenvector, and that the eigenbasis with respect to $\nabla^{\star}\nabla$ is complete in $L^{2}[\mathbb{S}^{5},\pi^{\star}_{5,4}adE]$, we find 
\begin{eqnarray}\nonumber& &\{\left[\begin{array}{c}u\\ 0 \\  0 \\ 0 \\ 0 \end{array}\right] |u\in L^{2}[\mathbb{S}^{5},\pi^{\star}_{5,4}adE]\}\subseteq  [\oplus^{L^{2}}_{\lambda \in Spec(\nabla^{\star}\nabla|_{\mathbb{S}^{5}})}Row^{1}(\mathbb{E}_{\lambda}\nabla^{\star}\nabla|_{\mathbb{S}^{5}})]
\\& \subseteq  &\oplus^{L^{2}}_{\mu,\mu^{2}+2\mu -3\in Spec(\nabla^{\star}\nabla|_{\mathbb{S}^{5}})}\mathbb{E}_{\mu}P|_{V^{\perp}_{coh}}. 
\end{eqnarray}
Similarly, using \eqref{equ formula for P2+2P separation} again, we verify that 
\begin{equation}\{\left[\begin{array}{c}0\\ a_{s} \\  0 \\ 0 \\ 0 \end{array}\right] |a_{s}\in L^{2}[\mathbb{S}^{5},\pi^{\star}_{5,4}adE]\} \subseteq \oplus^{L^{2}}_{\mu,\mu^{2}+2\mu -3\in Spec(\nabla^{\star}\nabla|_{\mathbb{S}^{5}})}\mathbb{E}_{\mu}P|_{V^{\perp}_{coh}}.
\end{equation}
Using \eqref{equ formula for P2+4P separation} instead of \eqref{equ formula for P2+2P separation}, we verify 
\begin{equation}\{\left[\begin{array}{c}0\\ 0 \\  a_{r} \\ 0 \\ 0 \end{array}\right] |a_{r}\in L^{2}[\mathbb{S}^{5},\pi^{\star}_{5,4}adE]\} \subseteq  \oplus^{L^{2}}_{\mu,\mu^{2}+4\mu \in Spec(\nabla^{\star}\nabla|_{\mathbb{S}^{5}})}\mathbb{E}_{\mu}P|_{V^{\perp}_{coh}},
\end{equation}
and
\begin{equation}\{\left[\begin{array}{c}0\\ 0 \\ 0\\ a_{\eta} \\ 0 \end{array}\right] |a_{\eta}\in L^{2}[\mathbb{S}^{5},\pi^{\star}_{5,4}adE]\} \subseteq  \oplus^{L^{2}}_{\mu,\mu^{2}+4\mu \in Spec(\nabla^{\star}\nabla|_{\mathbb{S}^{5}})}\mathbb{E}_{\mu}P|_{V^{\perp}_{coh}}.
\end{equation}
In summary, by Definition \ref{Def two parts of spec} of  $S_{\nabla^{\star}\nabla}$, we find 
\begin{eqnarray}& &\nonumber\{\left[\begin{array}{c}u\\ a_{s} \\  a_{r} \\ a_{\eta} \\ 0 \end{array}\right] |u,a_{s},a_{r},a_{\eta}\in L^{2}[\mathbb{S}^{5},\pi^{\star}_{5,4}adE]\}
 \subseteq   \oplus^{L^{2}}_{\mu\in S_{\nabla^{\star}\nabla}} \mathbb{E}_{\mu}P|_{V^{\perp}_{coh}}.
\end{eqnarray}
Then  \begin{eqnarray}& & \oplus^{L^{2}}_{\mu \notin S_{\nabla^{\star}\nabla}}\mathbb{E}_{\mu}P|_{V^{\perp}_{coh}}
\nonumber
 \subseteq \{\left[\begin{array}{c}u\\ a_{s} \\  a_{r} \\ a_{\eta} \\ 0 \end{array}\right] |u,a_{s},a_{r},a_{\eta}\in L^{2}[\mathbb{S}^{5},\pi^{\star}_{5,4}adE]\}^{\perp}\\& \subseteq &\{\left[\begin{array}{c}0\\ 0 \\  0 \\ 0\\ a_{0} \end{array}\right] |a_{0}\in L^{2}[\mathbb{S}^{5},D^{\star}\otimes \pi^{\star}_{5,4}adE]\}.
\end{eqnarray}
The above implies that if $\mu$ is an eigenvalue of $P|_{V^{\perp}_{coh}}$ and $\mu \notin S_{\nabla^{\star}\nabla}$,  any eigensection of $\mu$ must be of the form $\left[\begin{array}{c}0\\ 0 \\  0 \\ 0\\ a_{0} \end{array}\right]$ i.e. row $1$ to row $4$ must all vanish. Again, by  Formula \eqref{equ formula for P}, if $a_{0}$ does not vanish,  $a_{0}\in V_{\mu}$ for some integer $\mu$. But  $a_{0}\perp V_{\mu}$ 
by assumption. Then $a_{0}=0$.

The statement $Spec P|_{V^{\perp}_{coh}}\subseteq S_{\nabla^{\star}\nabla}$ is  proved.\\

Step 2: .  $Spec P|_{V^{\perp}_{coh}}\supseteq S_{\nabla^{\star}\nabla}$.

Actually, Step 1 has already hinted on this direction. 
Equation \eqref{equ 0 proof Thm SpecP} can be written as  
\begin{equation}\left[\begin{array}{c}u_{\lambda} \\ 0 \\  0 \\ 0 \\ 0 \end{array}\right]=\Sigma_{\mu \in Spec^{mul}(P|_{V^{\perp}_{coh}}),\mu =-1+\sqrt{4+\lambda}} u_{\lambda,\mu }\phi_{\mu }+\Sigma_{\mu \in Spec^{mul}(P|_{V^{\perp}_{coh}}),\mu =-1-\sqrt{4+\lambda}} u_{\lambda,\mu }\phi_{\mu }.\end{equation}
Because $u_{\lambda}$ is non-zero, the above means at least one of $-1-\sqrt{4+\lambda}$ and $-1+\sqrt{4+\lambda}$ is an eigenvalue of $P|_{V^{\perp}_{coh}}$. We show by contradiction that both of them must be eigenvalues. If not, $\left[\begin{array}{c}u_{\lambda} \\ 0 \\  0 \\ 0 \\ 0 \end{array}\right]$ is an eigensection of $P$, thus $d_{0}u_{\lambda}=L_{\xi}u_{\lambda}=0$ by  Formula \eqref{equ formula for P}. By Fact  \ref{fact irred on P2 = irred on S5}, irreducibility of the connection on $\mathbb{P}^{2}$ implies the irreducibility of the pullback connection on $\mathbb{S}^{5}$,  which yields $u_{\lambda}=0$. This contradicts the hypothesis that $u_{\lambda}\neq 0$. 

Likewise, regarding the other polynomial $\mu^{2}+4\mu$, corresponding to the $3$rd and $4$th row of $Dom_{\mathbb{S}^{5}}$, $\left[\begin{array}{c}0 \\ 0 \\  u_{\lambda} \\ 0 \\ 0 \end{array}\right]$ yields the two eigenvalues $-2-\sqrt{4+\lambda}$ and $-2+\sqrt{4+\lambda}$. 

We proved $Spec P|_{V^{\perp}_{coh}}\supseteq S_{\nabla^{\star}\nabla}$. The proof of  Theorem \ref{Thm eigenvalue Sasakian} is complete. 
\end{proof}

\subsubsection{Proving Theorem \ref{Thm 1}.$\mathbb{I}$}
 To complete the proof, we need the formula for the sheaf cohomologies.  \begin{lem}\label{lem h1}Let $E$ be a holomorphic vector bundle on $\mathbb{P}^{2}$. For any integer $k$, we have  \begin{eqnarray}\label{equ 0 lem h1} & & h^{1}[\mathbb{P}^{2},\ (EndE)(k)]\\&=& h^{0}[\mathbb{P}^{2},\ (EndE)(k)]+h^{0}[\mathbb{P}^{2},\ (EndE)(-k-3)]\nonumber
\\& &+2rc_{2}(E)-(r-1)c^{2}_{1}(E)-\frac{r^{2}(k+1)(k+2)}{2}.\nonumber
\end{eqnarray}

Therefore, suppose additionally that $E$ is stable and  $rankE\geq 2$, the following is true. 
\begin{equation}\label{equ 1 lem h1}h^{1}[\mathbb{P}^{2},\ (End E)(-1)]=h^{1}[\mathbb{P}^{2},\ (End E)(-2)]=c_{2}(EndE)>0. 
\end{equation}

\end{lem}
The Chern number inequality \eqref{equ 1 lem h1} is crucial in showing $-1$ and $-2$ are eigenvalues of $P$.

 The proof of the above Lemma is simply by Riemann-Roch formula, Serre-duality,  and a vanishing theorem proved by Kobayashi. We defer it to Appendix \ref{Appendix Some algebro-geometric calculations}.
\begin{proof}[\textbf{Proof of Theorem \ref{Thm 1}}.$\mathbb{I}$:] The splitting $SpecP=S_{\nabla^{\star}\nabla}\cup S_{coh}$ follows simply from Remark \ref{rmk specPvcoh=scoh} and Theorem \ref{Thm eigenvalue Sasakian}.

 Fact \ref{fact stable=irred} says that the bundle $E$ on $\mathbb{P}^{2}$ must be stable under the irreducible Hermitian Yang-Mills conduction. The calculation of cohomology in Lemma \ref{lem h1}.\eqref{equ 1 lem h1}, the obvious fact that $V_{\mu }\subseteq E_{\mu }P$, and the identification in Proposition \ref{cor eigenspace and harmonic forms} says that $-1$ and $-2$ must be eigenvalues. Because of irreducibility of the connection on $\mathbb{P}^{2}$ and therefore on $\mathbb{S}^{5}$,  the lowest eigenvalue in $Spec(\nabla^{\star}\nabla |_{\mathbb{S}^{5}})$ is positive. Then the eigenvalues of $P$ other than $-1$ and $-2$ are either nonnegative or no larger than  -3. 

The proof is complete. 
\end{proof}

 \subsection{The dimension reduction for the spectrum of bundle rough Laplacians}
 To study the multiplicity of each eigenvalue of $P$, we need the following reduction. 
 \begin{formula}\label{formula laplace on S5 vs laplace on CP2} Given a holomorphic Hermitian triple $(E,h,A_{O})$ on $\mathbb{P}^{2}$, 
under the associated data setting and counted with multiplicity, the following holds. \begin{eqnarray}\nonumber & &Spec^{mul}\nabla^{\star}\nabla  |_{\pi_{5,4}^{\star}adE \rightarrow \mathbb{S}^{5}}=Spec^{mul_{\mathbb{C}}}\nabla^{\star}\nabla  |_{\pi_{5,4}^{\star}End_{0}E \rightarrow \mathbb{S}^{5}}
 \\&=&\{\alpha_{l}+l^{2}\ |\ \alpha_{l}\in Spec^{mul_{\mathbb{C}}}\nabla^{\star}\nabla|_{(End_{0}E)(l)}\}.\label{equ spec mul}\end{eqnarray}
 Consequently, without counting multiplicity, 
 $$Spec \nabla^{\star}\nabla  |_{\mathbb{S}^{5}} \triangleq Spec\nabla^{\star}\nabla  |_{\pi_{5,4}^{\star}adE\rightarrow \mathbb{S}^{5}}=\{\alpha_{l}+l^{2}\ |\ \alpha_{l}\in Spec\nabla^{\star}\nabla|_{End_{0}E(l)}\}.$$

 \end{formula}
\begin{proof}[Proof of Formula  \ref{formula laplace on S5 vs laplace on CP2}:] In view of  identity \eqref{equ specmul = spec mulC}, it suffices to show 

\begin{equation}Spec^{mul_{\mathbb{C}}}\nabla^{\star}\nabla  |_{\pi_{5,4}^{\star}End_{0}E \rightarrow \mathbb{S}^{5}}
=\{\alpha_{l}+l^{2}\ |\ \alpha_{l}\in Spec^{mul_{\mathbb{C}}}\nabla^{\star}\nabla|_{(End_{0}E)(l)}\}.\end{equation}
 
 Given an arbitrary $p\in \mathbb{S}^{5}$,  let $v_{i},\ i=1...4$ be the transverse geodesic coordinate vector fields near $p$ in  Lemma \ref{lem Kahler geodesic frame} below. Because $\nabla_{v_{i}}v_{i}=0$ at $p$, the following formula for the rough Laplacian is true
 \begin{equation}\label{equ 0 proof formula laplace on S5 vs laplace on CP2}-(\nabla^{\star}\nabla u)(p)=[\nabla_{v_{i}}(\nabla_{v_{i}}u)](p)+(L^{2}_{\xi}u)(p).
 \end{equation}
 
 Suppose $u\in C^{10}[\mathbb{S}^{5},\pi^{\star}_{5,4}End_{0}E]$,  the Fourier-expansion $u=\Sigma_{l\in \mathbb{Z}}u_{l}\otimes s_{-l}$ converges point-wisely on $\mathbb{S}^{5}$. Moreover, we can calculate the Laplacian term by term i.e. using that $d_{0}s_{-l}=0$ and that $u_{l}$ descends to $\mathbb{P}^{2}$, we find 
 \begin{eqnarray}\nonumber & &\nabla_{v_{i}}(\nabla_{v_{i}}u)=\Sigma_{l,i} \nabla_{v_{i}}\nabla_{v_{i}}(u_{l}\otimes s_{-l})=\Sigma_{l,i}(\nabla_{v_{i}}\nabla_{v_{i}}u_{l})\otimes s_{-l}
 \\&=&- \Sigma_{l}(\nabla^{\star,\mathbb{P}^{2}} \nabla^{\mathbb{P}^{2}} u_{l})\otimes s_{-l}. \nonumber
 \end{eqnarray}
 
Using $L_{\xi}u_{l}=0$ again, and $L^{2}_{\xi}s_{-l}=-l^{2}s_{-l}$, we find 
\begin{equation}\label{equ 0 formula laplace on S5 vs laplace on CP2}\nabla^{\star}\nabla u= \Sigma_{l}[(\nabla^{\star,\mathbb{P}^{2}} \nabla^{\mathbb{P}^{2}} +l^{2})u_{l}]\otimes s_{-l}.
\end{equation}
Then the desired result \eqref{equ spec mul} follows by plugging an arbitrary eigensection $u$ into \eqref{equ 0 formula laplace on S5 vs laplace on CP2} and comparing the Sasaki-Fourier series of both hand sides. For the reader's convenience, we still provide the full detail. 

Let $\lambda\in Spec\nabla^{\star}\nabla |_{\pi_{5,4}^{\star}End_{0}E\rightarrow \mathbb{S}^{5}}$, and $u_{\lambda}$ is a (non-zero) eigenvector. 
Then \eqref{equ 0 formula laplace on S5 vs laplace on CP2} yields that 
\begin{equation}\lambda u_{\lambda}=\Sigma_{l}u_{\lambda,l}s_{l}=\Sigma_{l}[(\nabla^{\star,\mathbb{P}^{2}} \nabla^{\mathbb{P}^{2}} +l^{2})u_{\lambda, l}]\otimes s_{-l}.
\end{equation}
Then either $u_{\lambda, l}=0$ or $\nabla^{\star,\mathbb{P}^{2}} \nabla^{\mathbb{P}^{2}}u_{\lambda, l}= (\lambda-l^{2})u_{\lambda, l}$. Therefore, we obtain a linear injection $$i:\ \mathbb{E}_{\lambda}(\nabla^{\star}\nabla |_{\pi_{5,4}^{\star}End_{0}E\rightarrow \mathbb{S}^{5}})\rightarrow \oplus_{l,\lambda-l^{2}\in Spec\nabla^{\star}\nabla|_{(End_{0}E)(l)}}(\mathbb{E}_{\lambda-l^{2}}\nabla^{\star}\nabla|_{(End_{0}E)(l)})$$ sending $u_{\lambda}$ to its Sasaki-Fourier co-efficients. We note that only finitely many $l$ yield  non-zero eigenspace $\mathbb{E}_{\lambda-l^{2}}\nabla^{\star}\nabla|_{(End_{0}E)(l)}$.

The map $i$ is obviously surjective because for any set of Sasaki-Fourier co-efficients $$\Sigma^{\oplus}_{l}u_{\lambda,l}\in \oplus_{l,\lambda-l^{2}\in Spec\nabla^{\star}\nabla|_{(End_{0}E)(l)}}(\mathbb{E}_{\lambda-l^{2}}\nabla^{\star}\nabla|_{(End_{0}E)(l)}),$$ we have  $$\Sigma_{l} u_{\lambda,l}s_{l}\in \mathbb{E}_{\lambda}(\nabla^{\star}\nabla |_{\pi_{5,4}^{\star}End_{0}E\rightarrow \mathbb{S}^{5}}), \textrm{and}\ i(\Sigma_{l}u_{\lambda,l}s_{l})=\Sigma^{\oplus}_{l}u_{\lambda,l}.$$
 The proof of \eqref{equ spec mul} is complete. 
\end{proof}
In conjunction with the splitting \eqref{equ 0 proof formula laplace on S5 vs laplace on CP2} by the frame, we call the operator $$\Delta_{0}\triangleq \nabla^{\star}\nabla+L^{2}_{\xi}$$ the transverse rough Laplacian. 

\subsection{On the multiplicities of the eigenvalues of $P$: proof of Theorem \ref{Thm 1}.$\mathbb{II}$}

The first thing we can say on the multiplicities is the following. 
\begin{lem}\label{lem JH is serre duality} (Kodaira-Serre duality for eigenspaces) In conjunction with Remark \ref{rmk EP is complex}, for any eigenvalue $\mu $ of $P$, both $K,\ \underline{T}$ are real isomorphisms $:\ Eigen_{\mu }(P)\rightarrow Eigen_{-(\mu +3)}(P)$ that anti-commute with the complex structure $I$ of the eigenspaces of $P$.
\end{lem}
\begin{proof}[Proof of Lemma \ref{lem JH is serre duality}:] This is a direct corollary of the commutators between $P$ and the isomorphisms $I,\ \underline{T},\ K$ in \eqref{equ commutators for PIJK}.\end{proof}
That $K,\ \underline{T}$ both anti-commute with  $I$ is consistent with that the Serre-duality map is conjugate  $\mathbb{C}-$linear. 

We can go much further than Lemma \ref{lem JH is serre duality}:   the multiplicity of each eigenvalue of $P$ can be determined.

\subsubsection{The definition and formula for the projection map  $\parallel_{\mathbb{E}_{\mu }P|_{V^{\perp}_{coh}}}$}

 The purpose of this section is the formula for a certain projection map  (Lemma \ref{lem proj formula}). This map  is decisive in determining the multiplicities. Before defining it, we stress the following. 
\begin{Notation} In conjunction with that each number in $Spec^{mul} \nabla^{\star}\nabla|_{\mathbb{S}^{5}}$ generates $4$ eigenvalues of $P$ (see Theorem \ref{Thm 1}.$\mathbb{I}$), for any $\lambda\in Spec (\nabla^{\star}\nabla|_{\mathbb{S}^{5}})$,  let \begin{equation}\label{equ the 4 eigenvalues generated}\mu_{\lambda, +}\triangleq -1+\sqrt{4+\lambda},\ \mu_{\lambda, -}\triangleq -1-\sqrt{4+\lambda},\ \underline{\mu}_{\lambda, +}=-2+\sqrt{4+\lambda},\ \underline{\mu}_{\lambda, -}=-2-\sqrt{4+\lambda}.\end{equation}

For any $\mu \in S_{\nabla^{\star}\nabla} \subseteq Spec P$, at least one of the following holds.
\begin{itemize}\item  $\mu^{2}+2\mu-3\in Spec (\nabla^{\star}\nabla|_{\mathbb{S}^{5}})$.

\item $\mu^{2}+4\mu\in Spec (\nabla^{\star}\nabla|_{\mathbb{S}^{5}})$. 
\end{itemize}
\end{Notation}
We are ready to define the projection.

\begin{Def}\label{Notation proj to Vcohperp} Let $(E,h,A_{O})$ be an irreducible Hermitian Yang-Mills triple on $\mathbb{P}^{2}$. Suppose $\mu\in Spec P|_{V^{\perp}_{coh}}$.  Let $\parallel_{\mathbb{E}_{\mu }P|_{V^{\perp}_{coh}}}$ denote the orthogonal projection from\\ $(\mathbb{E}_{\lambda_{1}}\nabla^{\star}\nabla|_{\mathbb{S}^{5}})^{\oplus 2}\oplus (\mathbb{E}_{\lambda_{2}}\nabla^{\star}\nabla|_{\mathbb{S}^{5}})^{\oplus 2}$ to $\mathbb{E}_{\mu }P|_{V^{\perp}_{coh}}$ that is factored as follows.   $$ (\mathbb{E}_{\lambda_{1}}\nabla^{\star}\nabla|_{\mathbb{S}^{5}})^{\oplus 2}\oplus (\mathbb{E}_{\lambda_{2}}\nabla^{\star}\nabla|_{\mathbb{S}^{5}})^{\oplus 2}\rightarrow V^{\perp}_{coh}\rightarrow L^{2}(\mathbb{S}^{5},Dom_{\mathbb{S}^{5}})\rightarrow \mathbb{E}_{\mu }P|_{V^{\perp}_{coh}}.$$ 
\end{Def}  
In conjunction with Remark \ref{rmk specPvcoh=scoh}, the range of the first arrow above consists of sections of the form $\left[\begin{array}{c}v \\ h \\  g \\ w \\ 0\end{array}\right]$ which are automatically in $V^{\perp}_{coh}$, where $$v,\ h\in \mathbb{E}_{\lambda_{1}}\nabla^{\star}\nabla|_{\mathbb{S}^{5}},\ \textrm{and} \ g,\ w\in \mathbb{E}_{\lambda_{2}}\nabla^{\star}\nabla|_{\mathbb{S}^{5}}.$$

  We start from the first component of the domain bundle $Dom_{\mathbb{S}^{5}}$.
  
 Suppose $v\in \mathbb{E}_{\lambda_{1}}\nabla^{\star}\nabla|_{\mathbb{S}^{5}}$, still by the $P|_{V^{\perp}_{coh}}-$eigen Fourier expansion,  there is an (unique) orthogonal splitting
\begin{equation}\label{equ splitting with respect to P}\left[\begin{array}{c}v \\ 0 \\  0 \\ 0 \\ 0\end{array}\right]=\left[\begin{array}{c}u \\ a_{s} \\  a_{r} \\ a_{\eta} \\ a_{0}\end{array}\right]+\left[\begin{array}{c}\widetilde{u} \\ -a_{s} \\  -a_{r} \\ -a_{\eta} \\ -a_{0}\end{array}\right],\end{equation} such that $\left[\begin{array}{c}u \\ a_{s} \\  a_{r} \\ a_{\eta} \\ a_{0}\end{array}\right]\in \mathbb{E}_{\mu _{\lambda_{1},+}}P|_{V^{\perp}_{coh}}$, and $\left[\begin{array}{c}\widetilde{u} \\ -a_{s} \\  -a_{r} \\ -a_{\eta} \\ -a_{0}\end{array}\right]\in \mathbb{E}_{\mu _{\lambda_{1},-}}P|_{V^{\perp}_{coh}}$.
 

By the fine formula \eqref{equ formula for P} for $P$, the eigensection conditions give the following system. 
 \begin{eqnarray}\label{equ eigenvalue condition system}u-L_{\xi}a_{s}-(d_{0}a_{0})\lrcorner H=\mu _{\lambda_{1},+}u&,&\widetilde{u}+L_{\xi}a_{s}+(d_{0}a_{0})\lrcorner H=\mu _{\lambda_{1},-}\widetilde{u}.\nonumber
\\ L_{\xi}u+a_{s}+(d_{0}a_{0})\lrcorner G=\mu _{\lambda_{1},+}a_{s} &,& L_{\xi}\widetilde{u}-a_{s}-(d_{0}a_{0})\lrcorner G=-\mu _{\lambda_{1},-}a_{s}. \nonumber
\\ -4a_{r}-L_{\xi}a_{\eta}+d_{0}^{\star_{0}}a_{0}=\mu _{\lambda_{1},+}a_{r}&,&4a_{r}+L_{\xi}a_{\eta}-d_{0}^{\star_{0}}a_{0}=-\mu _{\lambda_{1},-}a_{r}.\label{equ eigenvector projection}
\\ L_{\xi}a_{r}-4a_{\eta}-(d_{0}a_{0})\lrcorner \frac{d\eta}{2}=\mu _{\lambda_{1},+}a_{\eta}&,&-L_{\xi}a_{r}+4a_{\eta}+(d_{0}a_{0})\lrcorner \frac{d\eta}{2}=-\mu _{\lambda_{1},-}a_{\eta}. \nonumber
\\ \left.\begin{array}{c}J_{H}d_{0}u-J_{G}d_{0}a_{s}+d_{0}a_{r}\\ +J_{0}d_{0}a_{\eta}-L_{\xi}J_{0}(a_{0}) 
\end{array}\right\} = \mu _{\lambda_{1},+}a_{0}  &,&  \left.\begin{array}{c} J_{H}d_{0}\widetilde{u}+J_{G}d_{0}a_{s}-d_{0}a_{r}\\ -J_{0}d_{0}a_{\eta}+L_{\xi}J_{0}(a_{0})\end{array}\right\}  = -\mu _{\lambda_{1},-}a_{0}.\nonumber
\end{eqnarray}
The column on the left of the comma corresponds to that $\left[\begin{array}{c}u \\ a_{s} \\  a_{r} \\ a_{\eta} \\ a_{0}\end{array}\right]\in \mathbb{E}_{\mu _{\lambda_{1},+}}P|_{V^{\perp}_{coh}}$, the column on the right of the comma corresponds to that $\left[\begin{array}{c}\widetilde{u} \\ -a_{s} \\  -a_{r} \\ -a_{\eta} \\ -a_{0}\end{array}\right]\in \mathbb{E}_{\mu _{\lambda_{1},-}}P|_{V^{\perp}_{coh}}$.

In view of formula \eqref{equ the 4 eigenvalues generated} for the $4$ generated eigenvalues, we have $$\mu _{\lambda_{1},+}-\mu _{\lambda_{1},-}=2\sqrt{4+\lambda_{1}}\neq 0.$$ Then, summing up the two equations in row 3 of \eqref{equ eigenvector projection}, we find $a_{r}=0$. Similar operation for row $4$ yields that $a_{\eta}=0$.

Then the last row of \eqref{equ eigenvector projection} simply becomes the following two equations.
\begin{equation}J_{H}d_{0}u-J_{G}d_{0}a_{s}-L_{\xi}J_{0}(a_{0}) = \mu _{\lambda_{1},+}a_{0},\ \  J_{H}d_{0}\widetilde{u}+J_{G}d_{0}a_{s}+L_{\xi}J_{0}(a_{0}) = -\mu _{\lambda_{1},-}a_{0}.
\end{equation}
Summing them up and using  $v=u+\widetilde{u}$ (which is evident from \eqref{equ splitting with respect to P}), we find $$J_{H}d_{0}v=2\sqrt{4+\lambda_{1}} a_{0}\  \ \textrm{i.e.}\  \ a_{0}=\frac{J_{H}d_{0}v}{2\sqrt{4+\lambda_{1}} }.$$ 

On the other hand, summing up the two equations in the first row of \eqref{equ eigenvector projection}, we find $$v=\mu _{\lambda_{1},+}u+\mu _{\lambda_{1},-}\widetilde{u}.$$ Using  $v=u+\widetilde{u}$ again, we conclude that 
$$u=\frac{1-\mu _{\lambda_{1},-}}{\mu _{\lambda_{1},+}-\mu _{\lambda_{1},-}}v=(\frac{1}{\sqrt{4+\lambda_{1}}}+\frac{1}{2})v,\ \widetilde{u}=\frac{\mu _{\lambda_{1},+}-1}{\mu _{\lambda_{1},+}-\mu _{\lambda_{1},-}}v=(-\frac{1}{\sqrt{4+\lambda_{1}}}+\frac{1}{2})v.$$

Next, summing up the two equations in the second row of \eqref{equ eigenvector projection}, we obtain $$a_{s}=\frac{L_{\xi}v}{2\sqrt{4+\lambda_{1}} }.$$ Then 
\begin{equation}\label{equ proj 1st compenent}\left[\begin{array}{c}v \\ 0 \\  0 \\ 0 \\ 0\end{array}\right]^{\parallel_{\mathbb{E}_{\mu _{\lambda_{1},+}}P|_{V^{\perp}_{coh}}}}=\left[\begin{array}{c}(\frac{1}{\sqrt{4+\lambda_{1}}}+\frac{1}{2})v \\ \frac{L_{\xi}v}{2\sqrt{4+\lambda_{1}} }  \\  0 \\ 0 \\ \frac{J_{H}d_{0}v}{2\sqrt{4+\lambda_{1}} }\end{array}\right],\ \left[\begin{array}{c}v \\ 0 \\  0 \\ 0 \\ 0\end{array}\right]^{\parallel_{\mathbb{E}_{\mu _{\lambda_{1},-}}P|_{V^{\perp}_{coh}}}}=\left[\begin{array}{c}(-\frac{1}{\sqrt{4+\lambda_{1}}}+\frac{1}{2})v \\ -\frac{L_{\xi}v}{2\sqrt{4+\lambda_{1}} }  \\  0 \\ 0 \\ -\frac{J_{H}d_{0}v}{2\sqrt{4+\lambda_{1}} }\end{array}\right].\end{equation}
The same method as above yields the following projection formulas for the other rows of $Dom_{\mathbb{S}^{5}}$. 
\begin{itemize}
\item Suppose $\lambda_{1}\in Spec (\nabla^{\star}\nabla|_{\mathbb{S}^{5}})$, for any $h\in \mathbb{E}_{\lambda_{1}}\nabla^{\star}\nabla|_{\mathbb{S}^{5}}$ such that  $h\neq 0$, 
 \begin{equation}\left[\begin{array}{c}0 \\ h \\  0 \\ 0 \\ 0\end{array}\right]^{\parallel_{\mathbb{E}_{\mu _{\lambda_{1},+}}P|_{V^{\perp}_{coh}}}}=\left[\begin{array}{c}-\frac{L_{\xi}h}{2\sqrt{4+\lambda_{1}} }  \\ (\frac{1}{\sqrt{4+\lambda_{1}}}+\frac{1}{2})h \\  0 \\ 0 \\ -\frac{J_{G}d_{0}h}{2\sqrt{4+\lambda_{1}} }\end{array}\right],\ \left[\begin{array}{c}0 \\ h \\  0 \\ 0 \\ 0\end{array}\right]^{\parallel_{\mathbb{E}_{\mu _{\lambda_{1},-}}P|_{V^{\perp}_{coh}}}}=\left[\begin{array}{c}\frac{L_{\xi}h}{2\sqrt{4+\lambda_{1}} }  \\ (-\frac{1}{\sqrt{4+\lambda_{1}}}+\frac{1}{2})h \\  0 \\ 0 \\ \frac{J_{G}d_{0}h}{2\sqrt{4+\lambda_{1}} }\end{array}\right].\end{equation}
 \item Suppose $\lambda_{2}\in Spec (\nabla^{\star}\nabla|_{\mathbb{S}^{5}})$, for any $g\in \mathbb{E}_{\lambda_{2}}\nabla^{\star}\nabla|_{\mathbb{S}^{5}}$ such that $g\neq 0$, 
 \begin{equation}\left[\begin{array}{c}0 \\ 0 \\  g \\ 0 \\ 0\end{array}\right]^{\parallel_{\mathbb{E}_{\underline{\mu }_{\lambda_{2},+}}P|_{V^{\perp}_{coh}}}}=\left[\begin{array}{c}  0 \\ 0 \\ (-\frac{1}{\sqrt{4+\lambda_{2}}}+\frac{1}{2})g\\ \frac{L_{\xi}g}{2\sqrt{4+\lambda_{2}} }   \\ \frac{d_{0}g}{2\sqrt{4+\lambda_{2}} }\end{array}\right],\ \left[\begin{array}{c}0 \\ 0 \\  g \\ 0 \\ 0\end{array}\right]^{\parallel_{\mathbb{E}_{\underline{\mu }_{\lambda_{2},-}}P|_{V^{\perp}_{coh}}}}=\left[\begin{array}{c}  0 \\ 0 \\ (\frac{1}{\sqrt{4+\lambda_{2}}}+\frac{1}{2})g\\ -\frac{L_{\xi}g}{2\sqrt{4+\lambda_{2}} }   \\ -\frac{d_{0}g}{2\sqrt{4+\lambda_{2}} }\end{array}\right].\end{equation}
  \item Suppose $\lambda_{2}\in Spec (\nabla^{\star}\nabla|_{\mathbb{S}^{5}})$, for any $w\in \mathbb{E}_{\lambda_{2}}\nabla^{\star}\nabla|_{\mathbb{S}^{5}}$ such that $w\neq 0$, 
 \begin{equation}\label{equ proj 4th compenent}\left[\begin{array}{c}0 \\ 0 \\  0 \\ w \\ 0\end{array}\right]^{\parallel_{\mathbb{E}_{\underline{\mu }_{\lambda_{2},+}}P|_{V^{\perp}_{coh}}}}=\left[\begin{array}{c}  0 \\ 0 \\ -\frac{L_{\xi}w}{2\sqrt{4+\lambda_{2}} }   \\ (-\frac{1}{\sqrt{4+\lambda_{2}}}+\frac{1}{2})w\\  \frac{J_{0}d_{0}w}{2\sqrt{4+\lambda_{2}} }\end{array}\right],\ \left[\begin{array}{c}0 \\ 0 \\  0 \\ w \\ 0\end{array}\right]^{\parallel_{\mathbb{E}_{\underline{\mu }_{\lambda_{2},-}}P|_{V^{\perp}_{coh}}}}=\left[\begin{array}{c}  0 \\ 0 \\ \frac{L_{\xi}w}{2\sqrt{4+\lambda_{2}} }   \\ (\frac{1}{\sqrt{4+\lambda_{2}}}+\frac{1}{2})w\\  -\frac{J_{0}d_{0}w}{2\sqrt{4+\lambda_{2}} }\end{array}\right].\end{equation}
\end{itemize}

Summing up the formulas \eqref{equ proj 1st compenent}--\eqref{equ proj 4th compenent}, we arrive on target.   
\begin{lem}\label{lem proj formula} In the setting of Definition \ref{Notation proj to Vcohperp}, for any $\mu \in S_{\nabla^{\star}\nabla}\subset Spec P$, let $$\lambda_{1}\triangleq \mu ^{2}+2\mu -3\  \textrm{and}\ \lambda_{2}\triangleq \mu ^{2}+4\mu .$$ Suppose $v,\ h \in \mathbb{E}_{\lambda_{1}}\nabla^{\star}\nabla|_{\mathbb{S}^{5}}$ and  $g,\ w \in \mathbb{E}_{\lambda_{2}}\nabla^{\star}\nabla|_{\mathbb{S}^{5}}$, the following projection formula is true. 
\begin{equation}\label{equ 0 lem proj formula} \left[\begin{array}{c}v \\ h \\  g \\ w \\ 0\end{array}\right]^{\parallel_{\mathbb{E}_{\mu }P|_{V^{\perp}_{coh}}}}=\left[\begin{array}{c}-\frac{L_{\xi}h}{2(\mu+1)} +(\frac{1}{ \mu+1}+\frac{1}{2})v \\ \frac{L_{\xi}v}{2(\mu+1)} +(\frac{1}{ \mu+1}+\frac{1}{2})h\\ -\frac{L_{\xi}w}{2 (\mu+2) } +(-\frac{1}{ \mu+2}+\frac{1}{2})g \\ \frac{L_{\xi}g}{2 (\mu+2) } +(-\frac{1}{ \mu+2}+\frac{1}{2})w \\ -\frac{J_{G}d_{0}h}{2(\mu+1)}+\frac{J_{H}d_{0}v}{2 (\mu+1)}+\frac{d_{0}g}{2 (\mu+2) }+\frac{J_{0}d_{0}w}{2 (\mu+2) }\end{array}\right] .\end{equation}

\end{lem}
If  $\lambda_{1}\notin Spec\nabla^{\star}\nabla|_{\mathbb{S}^{5}}$, then  $v,\ h$ must be $0$. This is precisely an advantage in defining  eigenspaces for all real numbers. The rationale is that if it is not an eigenvalue, then all ``eigensections" are $0$. The same applies to  $\lambda_{2}$.


\subsubsection{Surjectivity of the projection map $\parallel_{\mathbb{E}_{\mu }P|_{V^{\perp}_{coh}}}$}

The purpose of this section is to prove the following. 
\begin{prop}\label{prop surjective in multiplicity} In the setting of Definition \ref{Notation proj to Vcohperp},  the orthogonal projection $\parallel_{\mathbb{E}_{\mu }P|_{V^{\perp}_{coh}}}$ is surjective. 
\end{prop}
\begin{rmk}\label{rmk form of eigensection} The surjectivity particularly means that any eigensection of $P$ must be of the form \eqref{equ 0 lem proj formula}.
\end{rmk}
\begin{proof}[Proof of Proposition \ref{prop surjective in multiplicity}:] The condition \begin{equation}\label{equ coker condition 1} \left[\begin{array}{c}u \\ a_{s} \\  a_{r} \\ a_{\eta} \\ a_{0}\end{array}\right]\perp \left[\begin{array}{c}v \\ h \\  g \\ w \\ 0\end{array}\right]^{\parallel_{\mathbb{E}_{\mu }P|_{V^{\perp}_{coh}}}}\end{equation} is equivalent to that \begin{eqnarray*}0&=& \langle u,-\frac{L_{\xi}h}{2(\mu+1)}+(\frac{1}{\mu+1}+\frac{1}{2})v \rangle+  \langle a_{s},\frac{L_{\xi}v}{2(\mu+1)} +(\frac{1}{\mu+1}+\frac{1}{2})h\rangle \\& & +\langle a_{r},-\frac{L_{\xi}w}{2 (\mu+2) } +(-\frac{1}{\mu+2}+\frac{1}{2})g \rangle+\langle a_{\eta}, \frac{L_{\xi}g}{2 (\mu+2) } +(-\frac{1}{\mu+2}+\frac{1}{2})w\rangle \\& &+\langle a_{0}, -\frac{J_{G}d_{0}h}{2(\mu+1)}+\frac{J_{H}d_{0}v}{2 (\mu+1)}+\frac{d_{0}g}{2 (\mu+2) }+\frac{J_{0}d_{0}w}{2 (\mu+2) }\rangle.\end{eqnarray*} 

Using that the adjoint of $L_{\xi}$ is $-L_{\xi}$, and that $J_{0}d_{0}$, $J_{H}d_{0}$, $J_{G}d_{0}$ are adjoint to $-d^{\star_{0}}_{0}J_{0}$, $-d^{\star_{0}}_{0}J_{H}$, $-d^{\star_{0}}_{0}J_{G}$ respectively,  we find \begin{eqnarray}\label{equ 0 prop surjective in multiplicity}0&=&<h,\frac{L_{\xi}u}{2(\mu+1)} +(\frac{1}{\mu+1}+\frac{1}{2})a_{s}+\frac{d^{\star_{0}}_{0}J_{G}a_{0}}{2(\mu+1)}> \label{equ coker condition 2}
\\&&+<v,-\frac{L_{\xi}a_{s}}{2(\mu+1)} +(\frac{1}{\mu+1}+\frac{1}{2})u-\frac{d^{\star_{0}}_{0}J_{H}a_{0}}{2(\mu+1)}>\nonumber\\& &+<w,\frac{L_{\xi}a_{r}}{2 (\mu+2) } +(-\frac{1}{\mu+2}+\frac{1}{2})a_{\eta}-\frac{d^{\star_{0}}_{0}J_{0}a_{0}}{2 (\mu+2) }>\nonumber
\\& &+ <g,-\frac{L_{\xi}a_{\eta}}{2 (\mu+2)} +(-\frac{1}{\mu+2}+\frac{1}{2})a_{r}+\frac{d^{\star_{0}}_{0}a_{0}}{2 (\mu+2) }>.\nonumber
\end{eqnarray}
Therefore, the condition $\left[\begin{array}{c}u \\ a_{s} \\  a_{r} \\ a_{\eta} \\ a_{0}\end{array}\right]\in Coker\parallel_{\mathbb{E}_{\mu }P|_{V^{\perp}_{coh}}}$ is equivalent to that \eqref{equ 0 prop surjective in multiplicity} holds for all $\left[\begin{array}{c}v \\ h \\  g \\ w \\ 0\end{array}\right]\in (\mathbb{E}_{\lambda_{1}}\nabla^{\star}\nabla|_{\mathbb{S}^{5}})^{\oplus 2}\oplus (\mathbb{E}_{\lambda_{2}}\nabla^{\star}\nabla|_{\mathbb{S}^{5}})^{\oplus 2}$ and $\left[\begin{array}{c}u \\ a_{s} \\  a_{r} \\ a_{\eta} \\ a_{0}\end{array}\right]\in \mathbb{E}_{\mu }P|_{V^{\perp}_{coh}}$.

The eigensection condition (the left of comma in system \eqref{equ eigenvalue condition system}) again says that 
\begin{eqnarray}\label{equ mu lambda}-L_{\xi}a_{s}-(d_{0}a_{0})\lrcorner H=(\mu -1) u,\ & &  L_{\xi}u+(d_{0}a_{0})\lrcorner G=(\mu -1)a_{s}.
\\ -L_{\xi}a_{\eta}+d_{0}^{\star_{0}}a_{0}=(\mu +4)a_{r}, & & L_{\xi}a_{r}-(d_{0}a_{0})\lrcorner \frac{d\eta}{2}=(\mu +4)a_{\eta}.\nonumber\end{eqnarray}

Plugging \eqref{equ mu lambda} into \eqref{equ coker condition 2}, we find 
\begin{equation}<h,a_{s}>+<v,u>+<w,a_{\eta}>+<g,a_{r}>=0. 
\end{equation}
Because $\left[\begin{array}{c}v \\ h \\  g \\ w \\ 0\end{array}\right]\in (\mathbb{E}_{\lambda_{1}}\nabla^{\star}\nabla|_{\mathbb{S}^{5}})^{\oplus 2}\oplus (\mathbb{E}_{\lambda_{2}}\nabla^{\star}\nabla|_{\mathbb{S}^{5}})^{\oplus 2}$ is arbitrary, we find $$u=a_{s}=a_{r}=a_{\eta}=0.$$ 
That $\left[\begin{array}{c}u \\ a_{s} \\  a_{r} \\ a_{\eta} \\ a_{0}\end{array}\right] =\left[\begin{array}{c}0 \\ 0 \\  0 \\ 0 \\ a_{0}\end{array}\right]\in V^{\perp}_{coh}$ implies $a_{0}=0$ as well, therefore $Coker \parallel_{\mathbb{E}_{\mu }P|_{V^{\perp}_{coh}}}=\{0\}$. \end{proof}

\subsubsection{Kernel of the projection  $\parallel_{\mathbb{E}_{\mu }P|_{V^{\perp}_{coh}}}$}
The purpose of this section is to describe the kernel of the projection. 

\begin{prop}\label{prop inj not integer} In the setting of Definition \ref{Notation proj to Vcohperp}, the following holds. 
\begin{enumerate}\item When $\mu$ is not an integer, the orthogonal projection $\parallel_{\mathbb{E}_{\mu }P|_{V^{\perp}_{coh}}}$ is injective. 
\item When $\mu$ is an integer, 
\begin{equation}\label{equ Ker of proj and holo sections}Ker\parallel_{\mathbb{E}_{\mu }P|_{V^{\perp}_{coh}}}=H^{0}[\mathbb{P}^{2},(End_{0}E)(\mu)]\oplus H^{0}[\mathbb{P}^{2},(End_{0}E)(-\mu-3)].\end{equation}
\end{enumerate}
\end{prop}
\begin{proof}[Proof of Proposition \ref{prop inj not integer} :]Suppose 
\begin{equation}\label{equ vanishing of the proj to Ebeta P}\left[\begin{array}{c}v \\ h \\  g \\ w \\ 0\end{array}\right]\in (\mathbb{E}_{\lambda_{1}}\nabla^{\star}\nabla|_{\mathbb{S}^{5}})^{\oplus 2}\oplus (\mathbb{E}_{\lambda_{2}}\nabla^{\star}\nabla|_{\mathbb{S}^{5}})^{\oplus 2},\ \textrm{and}\ \left[\begin{array}{c}v \\ h \\  g \\ w \\ 0\end{array}\right]^{\parallel_{\mathbb{E}_{\mu }P|_{V^{\perp}_{coh}}}}=0.\end{equation} 

By projection formula \eqref{equ 0 lem proj formula}, the vanishing \eqref{equ vanishing of the proj to Ebeta P} is equivalent to 
\begin{eqnarray}\label{equ expanded vanishing of the proj to Ebeta P} \nonumber -\frac{L_{\xi}h}{2(\mu+1)} +(\frac{1}{\mu+1}+\frac{1}{2})v=0 & ,& \frac{L_{\xi}v}{2(\mu+1)} +(\frac{1}{\mu+1}+\frac{1}{2})h=0,\\ -\frac{L_{\xi}w}{2 (\mu+2) } +(-\frac{1}{\mu+2}+\frac{1}{2})g=0 & ,&\nonumber \frac{L_{\xi}g}{2 (\mu+2) } +(-\frac{1}{\mu+2}+\frac{1}{2})w=0, \\ \textrm{and}\ \ \ -\frac{J_{G}d_{0}h}{2(\mu+1)}+\frac{J_{H}d_{0}v}{2 (\mu+1)}& +&\frac{d_{0}g}{2 (\mu+2) }+\frac{J_{0}d_{0}w}{2 (\mu+2)}=0.
\end{eqnarray}
We note again that Theorem \ref{Thm 1},$\mathbb{I}$ says none of $\mu,\ \mu+1,\ \mu+2,\ \mu+3$ is $0$. Hence row $1$ of \eqref{equ expanded vanishing of the proj to Ebeta P} is equivalent to
\begin{equation}\label{equ row 1 of expanded vanishing of the proj to Ebeta P}L^{2}_{\xi}h=-(\mu +3)^{2}h,\ \ \ v=\frac{L_{\xi}h}{\mu +3}. \end{equation}
Similarly, row $2$ of \eqref{equ expanded vanishing of the proj to Ebeta P} is equivalent to
\begin{equation}\label{equ row 2 of expanded vanishing of the proj to Ebeta P}L^{2}_{\xi}w=-\mu^{2}w,\ \ \ \ g=\frac{L_{\xi}w}{\mu}. \end{equation}

\subsubsection*{Part I: Suppose $\mu$ is not an integer.}  Because the eigenvalues of $-L^{2}_{\xi}$ are squares of integers, but $\mu$ is not an integer, we find by \eqref{equ row 1 of expanded vanishing of the proj to Ebeta P} and \eqref{equ row 2 of expanded vanishing of the proj to Ebeta P} that
$$\left[\begin{array}{c}v \\ h \\  g \\ w \\ 0\end{array}\right]=0.$$

\subsubsection*{Part II: Suppose $\mu$ is  an integer.} 

 In this case, \eqref{equ row 1 of expanded vanishing of the proj to Ebeta P} and \eqref{equ row 2 of expanded vanishing of the proj to Ebeta P}  do not force the eigensection to vanish. We show that this is how the space of holomorphic sections come into play.

  Similarly to the proof of the two term expansion in Claim \ref{clm Fourier expansion of eigensection},  equation \eqref{equ row 1 of expanded vanishing of the proj to Ebeta P}  implies that the Sasaki-Fourier series of $h$ only has $2-$terms i.e. 
\begin{equation}\label{equ two expansion for h} h=h_{\mu +3}s_{-(\mu +3)}+h_{-(\mu +3)}s_{\mu +3}.
\end{equation}
Moreover, because $L_{\xi}s_{k}=-\sqrt{-1}k s_{k}$, we  find 
\begin{equation}\label{equ two expansion for v} v=\frac{L_{\xi}h}{\mu +3}=\sqrt{-1}[h_{\mu +3}s_{-(\mu +3)}-h_{-(\mu +3)}s_{\mu +3}]. 
\end{equation}
Consequently, the transverse differential of $h$ is 
\begin{equation}\label{equ d0h}d_{0}h=[d_{\mathbb{P}^{2}}h_{\mu +3}]s_{-(\mu +3)}+[d_{\mathbb{P}^{2}}h_{-(\mu +3)}]s_{\mu +3},
\end{equation}
and that of $v$ is
\begin{equation}d_{0}v=\sqrt{-1}\{[d_{\mathbb{P}^{2}}h_{\mu +3}]s_{-(\mu +3)}-[d_{\mathbb{P}^{2}}h_{-(\mu +3)}]s_{\mu +3}\}. 
\end{equation}
Hence, 
\begin{eqnarray}\label{equ j0d0v}J_{0}d_{0}v&=&\sqrt{-1}\{[J_{\mathbb{P}^{2}}d_{\mathbb{P}^{2}}h_{\mu +3}]s_{-(\mu +3)}-[J_{\mathbb{P}^{2}}d_{\mathbb{P}^{2}}h_{-(\mu +3)}]s_{\mu +3}\}.  \nonumber
\\&=&[\partial_{\mathbb{P}^{2}}h_{\mu +3}-\bar{\partial}_{\mathbb{P}^{2}}h_{\mu +3}]s_{-(\mu +3)}-[\partial_{\mathbb{P}^{2}}h_{-(\mu +3)}-\bar{\partial}_{\mathbb{P}^{2}}h_{-(\mu +3)}]s_{\mu +3}.\end{eqnarray}
Equation \eqref{equ d0h} and \eqref{equ j0d0v} amount to
\begin{equation}\label{equ d0h+j0d0v} d_{0}h+J_{0}d_{0}v
=2[(\partial_{\mathbb{P}^{2}}h_{\mu +3})s_{-(\mu +3)}+(\bar{\partial}_{\mathbb{P}^{2}}h_{-(\mu +3)})s_{\mu +3}].
\end{equation}
Because $G$ is minus the imaginary part of the form $\Theta$ i.e. 
\begin{equation}G=\frac{\overline{\Theta}-\Theta}{2\sqrt{-1}},
\end{equation}
we calculate that 
\begin{eqnarray}\label{equ JGd0h etc}& &-J_{G}d_{0}h+J_{H}d_{0}v=-J_{G}(d_{0}h+J_{0}d_{0}v)\nonumber
\\&=&-2\{[(\partial_{\mathbb{P}^{2}}h_{\mu +3})\lrcorner G] s_{-(\mu +3)}+[(\bar{\partial}_{\mathbb{P}^{2}}h_{-(\mu +3)})\lrcorner G]s_{\mu +3}\}\nonumber
\\&=&\sqrt{-1}\{[(\partial_{\mathbb{P}^{2}}h_{\mu +3})\lrcorner \bar{\Theta}] s_{-(\mu +3)}-[(\bar{\partial}_{\mathbb{P}^{2}}h_{-(\mu +3)})\lrcorner \Theta]s_{\mu +3}\}.
\end{eqnarray}
In the above, we used again that the contraction between a $(1,0)-$form and a $(2,0)-$form vanishes, and that the contraction between a $(0,1)-$form and a $(0,2)-$form vanishes. 

Using \eqref{equ row 2 of expanded vanishing of the proj to Ebeta P}, the following two term expansions hold as well.  \begin{equation}\label{equ two expansion for w and g}w=w_{\mu }s_{-\mu }+w_{-\mu }s_{\mu },\ \ g=\sqrt{-1}(w_{\mu }s_{-\mu }-w_{-\mu }s_{\mu}).
\end{equation}

By similar derivation as of \eqref{equ d0h+j0d0v}, we find 
\begin{equation}\label{equ d0g+j0d0w} d_{0}g+J_{0}d_{0}w
=2\sqrt{-1}[(\bar{\partial}_{\mathbb{P}^{2}}w_{\mu})s_{-\mu}-(\partial_{\mathbb{P}^{2}}w_{-\mu})s_{\mu}].
\end{equation}

In the light of \eqref{equ JGd0h etc} and \eqref{equ d0g+j0d0w}, the last equation in \eqref{equ expanded vanishing of the proj to Ebeta P}   
reads 
\begin{eqnarray}\label{equ 0 proof prop inj}& &\frac{\sqrt{-1}}{2 (\mu+1)}\{[(\partial_{\mathbb{P}^{2}}h_{\mu +3})\lrcorner \bar{\Theta}] s_{-(\mu +3)}-[(\bar{\partial}_{\mathbb{P}^{2}}h_{-(\mu +3)})\lrcorner \Theta]s_{\mu +3}\}
\\& &+\frac{\sqrt{-1}}{ (\mu+2)}[(\bar{\partial}_{\mathbb{P}^{2}}w_{\mu})s_{-\mu}-(\partial_{\mathbb{P}^{2}}w_{-\mu})s_{\mu}]\nonumber
\\&=&0.\nonumber\end{eqnarray}
The  $(1,0)$ and $(0,1)$ part of \eqref{equ 0 proof prop inj} should both vanish. This  yields the following.  
\begin{equation}\label{equ 1 proof prop inj}-[(\bar{\partial}_{\mathbb{P}^{2}}h_{-(\mu +3)})\lrcorner \Theta]s_{\mu +3}-\frac{2 (\mu+1)}{ (\mu+2)}(\partial_{\mathbb{P}^{2}}w_{-\mu})s_{\mu}=0;\end{equation}
\begin{equation}\label{equ 2 proof prop inj}[(\partial_{\mathbb{P}^{2}}h_{\mu +3})\lrcorner \bar{\Theta}]s_{-(\mu +3)}+\frac{2 (\mu+1)}{ (\mu+2)}(\bar{\partial}_{\mathbb{P}^{2}}w_{\mu})s_{-\mu}=0.\end{equation}
Using that
\begin{itemize}\item  $d_{0}\Theta=d_{0}\bar{\Theta}=0$ (see Formula \ref{formula d0 closedness of the 3 forms}),
\item any power of $s_{-1}$ or its conjugation is $d_{0}-$closed (Lemma \ref{lem d0 parallel section of pullback O(-1)}),
\item and that the curvature form $F_{A_{O}}$ is $(1,1)$,
\end{itemize}
 we find that 
 $$[(\bar{\partial}_{\mathbb{P}^{2}}h_{-(\mu +3)})\lrcorner \Theta]s_{\mu +3}=\star_{0}[(\bar{\partial}_{\mathbb{P}^{2}}h_{-(\mu +3)})\wedge \Theta]s_{\mu +3}\ \textrm{is}\ \bar{\partial}^{\star_{0}}_{0}-\textrm{closed},$$
 and
 $$[(\partial_{\mathbb{P}^{2}}h_{\mu +3})\lrcorner \bar{\Theta}]s_{-(\mu +3)}=\star_{0}[(\partial_{\mathbb{P}^{2}}h_{\mu +3})\wedge \bar{\Theta}]s_{-(\mu +3)}\ \textrm{is}\ \partial^{\star_{0}}_{0}-\textrm{closed}.$$
Plugging the above into \eqref{equ 1 proof prop inj}  and \eqref{equ 2 proof prop inj}, we see that  $$(\partial_{\mathbb{P}^{2}}w_{-\mu})s_{\mu}\  \textrm{is}\ \bar{\partial}^{\star_{0}}_{0}-\textrm{closed, and}\ 
(\bar{\partial}_{\mathbb{P}^{2}}w_{\mu})s_{-\mu}\  \textrm{is}\ \partial^{\star_{0}}_{0}-\textrm{closed}.$$ Because both $\partial_{\mathbb{P}^{2}}w_{-\mu}$ and 
$\bar{\partial}_{\mathbb{P}^{2}}w_{\mu}$ are forms on $\mathbb{P}^{2}$, and $\partial^{\star_{0}}_{0}$ ($\bar{\partial}^{\star_{0}}_{0}$) are equal to the usual $\partial^{\star_{\mathbb{P}^{2}}}_{\mathbb{P}^{2}}$ ($\bar{\partial}^{\star_{\mathbb{P}^{2}}}_{\mathbb{P}^{2}}$) on such forms, we find 
\begin{equation} \bar{\partial}^{\star_{\mathbb{P}^{2}}}_{\mathbb{P}^{2}}\partial_{\mathbb{P}^{2}}w_{-\mu}=0\ \textrm{and}\  \partial^{\star_{\mathbb{P}^{2}}}_{\mathbb{P}^{2}}\bar{\partial}_{\mathbb{P}^{2}}w_{\mu}=0.
\end{equation}
Integrating by parts on $\mathbb{P}^{2}$ shows 
\begin{equation}\label{equ 3 proof prop inj}  \partial_{\mathbb{P}^{2}}w_{-\mu}=0\ \textrm{and}\  \bar{\partial}_{\mathbb{P}^{2}}w_{\mu}=0.
\end{equation}

Plugging \eqref{equ 3 proof prop inj} back into \eqref{equ 1 proof prop inj}  and \eqref{equ 2 proof prop inj}, because the $(2,0)-$form $\Theta$ is no-where vanishing (its real and imaginary parts are both complex structures), we find
\begin{equation}\bar{\partial}_{\mathbb{P}^{2}}h_{-(\mu+3)}=0,\ \textrm{and}\  \partial_{\mathbb{P}^{2}}h_{\mu+3}=0.
\end{equation}
Then $$w_{\mu}\in H^{0}[\mathbb{P}^{2}, (End_{0}E)(\mu)]\ \ \textrm{and}\ \ h_{-(\mu+3)}\in H^{0}[\mathbb{P}^{2}, (End_{0}E)(-\mu-3)].$$ \\

The derivation so far has given a map 
$$ Q: Ker \parallel_{\mathbb{E}_{\mu }P|_{V^{\perp}_{coh}}}\rightarrow H^{0}[\mathbb{P}^{2},(End_{0}E)(\mu)]\oplus H^{0}[\mathbb{P}^{2},(End_{0}E)(-\mu-3)]$$
defined by 
$$Q\left[\begin{array}{c}v \\ h \\  g \\ w \\ 0\end{array}\right]\triangleq (w_{\mu},\ h_{-(\mu+3)}).$$

In a similar manner to \eqref{equ 0 proof eigenspace and harmonic form}, reversing the above arguments yields the obvious inverse of $Q$. For the reader's convenience, we still give the detail.  

For any $$w_{\mu}\in H^{0}[\mathbb{P}^{2}, (End_{0}E)(\mu)]\ \textrm{and}\ h_{-(\mu+3)}\in H^{0}[\mathbb{P}^{2}, (End_{0}E)(-\mu-3)],$$  let $$w_{-\mu}\triangleq \overline{w}^{t}_{\mu},\ \ \ h_{\mu+3}\triangleq \overline{h}^{t}_{-(\mu+3)}.$$ The inverse $Q^{-1}(w_{\mu},\ h_{-(\mu+3)})$ is simply the  $\left[\begin{array}{c}v \\ h \\  g \\ w \\ 0\end{array}\right]$ defined by  conditions \eqref{equ two expansion for h}, \eqref{equ two expansion for v}, and \eqref{equ two expansion for w and g}: 
\begin{itemize}\item the spectral reduction in Formula \ref{formula laplace on S5 vs laplace on CP2} and the identification in Lemma \ref{lem a holomorphic section is an eigensection of the rough Laplacian} below says that  
$$v\ \textrm{and}\ h\in \mathbb{E}_{\mu^{2}+2\mu-3}(\nabla^{\star}\nabla|_{\mathbb{S}^{5}})=\mathbb{E}_{\lambda_{1}}(\nabla^{\star}\nabla|_{\mathbb{S}^{5}}),\  g \ \textrm{and}\ w\in  \mathbb{E}_{\mu^{2}+4\mu}(\nabla^{\star}\nabla|_{\mathbb{S}^{5}})=\mathbb{E}_{\lambda_{2}}(\nabla^{\star}\nabla|_{\mathbb{S}^{5}});$$

\item $\left[\begin{array}{c}v \\ h \\  g \\ w \\ 0\end{array}\right]$ satisfies  conditions \eqref{equ row 1 of expanded vanishing of the proj to Ebeta P} and \eqref{equ row 2 of expanded vanishing of the proj to Ebeta P}, therefore it satisfies the whole  system \eqref{equ expanded vanishing of the proj to Ebeta P} which means $\left[\begin{array}{c}v \\ h \\  g \\ w \\ 0\end{array}\right]\in Ker\parallel_{\mathbb{E}_{\mu }P|_{V^{\perp}_{coh}}}$.
\end{itemize}

The desired identification \eqref{equ Ker of proj and holo sections} is proved. 
\end{proof}
\subsubsection{Proof of Theorem \ref{Thm 1}.$\mathbb{II}$}
Our understanding of the projection map  determines the multiplicities. 
\begin{proof}[\textbf{Proof of Theorem \ref{Thm 1}.$\mathbb{II}$}:] 

In view of Definition \ref{Def Vcoh}, and the characterization of $V_{l}$ in Proposition \ref{cor eigenspace and harmonic forms}, $Spec P|_{V_{coh}}$ are those integers $l$ such that $V_{l}=H^{1}[\mathbb{P}^{2}, (End_{0}E)(l)]$ is nonzero, and the eigenspace is $V_{l}$. It suffices to determine the multiplicities of $P$ restricted to $V^{\perp}_{coh}$. 

In conjunction with Theorem \ref{Thm eigenvalue Sasakian}, and the last paragraph of the proof of Theorem \ref{Thm 1}.$\mathbb{I}$, neither $-1$ nor $-2$ is in $Spec P|_{V^{\perp}_{coh}}=S_{\nabla^{\star}\nabla}$.  The eigenspaces of $-1$ and $-2$ are $V_{-1}$ and $V_{-2}$ respectively.  Then in view of Remark \ref{rmk EP is complex} and the complex isomorphism in Proposition \ref{cor eigenspace and harmonic forms},  the eigenspaces of $-1$ and $-2$ are complex isomorphic to $H^{1}[\mathbb{P}^{2}, (End E)(-1)]$ and $H^{1}[\mathbb{P}^{2}, (End E)(-2)]$ respectively. Theorem \ref{Thm 1}.$\mathbb{II}.1$ is proved.

In view of the spectral advantage in Theorem \ref{Thm eigenvalue Sasakian}, restricted to $V^{\perp}_{coh}$, the multiplicity of an eigenvalue $$\mu\in S_{\nabla^{\star}\nabla}\subset Spec P|_{V^{\perp}_{coh}}$$
is completely determined  by the surjective map
\begin{equation}\label{equ proj is iso}\parallel_{\mathbb{E}_{\mu }P|_{V^{\perp}_{coh}}}:\ (\mathbb{E}_{\lambda_{1}}\nabla^{\star}\nabla|_{\mathbb{S}^{5}})^{\oplus 2}\oplus (\mathbb{E}_{\lambda_{2}}\nabla^{\star}\nabla|_{\mathbb{S}^{5}})^{\oplus 2}\rightarrow \mathbb{E}_{\mu }P|_{V^{\perp}_{coh}}\ \textrm{in  Definition}\ \ref{Notation proj to Vcohperp}. \end{equation}

When $\mu$ is not an integer, the vanishing of $Ker\parallel_{\mathbb{E}_{\mu }P|_{V^{\perp}_{coh}}}$ in Proposition \ref{prop inj not integer}.$1$ says that \eqref{equ proj is iso} is an isomorphism. This finishes the proof of Theorem \ref{Thm 1}.$\mathbb{II}.2$. 

When $\mu$ is  an integer, the characterization \eqref{equ Ker of proj and holo sections} of $Ker\parallel_{\mathbb{E}_{\mu }P|_{V^{\perp}_{coh}}}$ in Proposition \ref{prop inj not integer}.$2$  finishes the proof of Theorem \ref{Thm 1}.$\mathbb{II}.3$.

\end{proof}

  \section{Spectral theory of the rough Laplacian on the bundle over $\mathbb{S}^{5}$, and the proof of Theorem \ref{Thm spec of rough laplacian P2}, Proposition \ref{prop multiplicity TP2}, and Corollary \ref{Cor 1}.\label{sect proof of Cor 1}}
  
  Theorem \ref{Thm 1} reduces the spectrum of $P$ to spectrum of the rough Laplacian on $\mathbb{S}^{5}$. The purpose of this section is to prove Theorem \ref{Thm spec of rough laplacian P2} and Proposition \ref{prop multiplicity TP2} i.e. to completely characterize the eigenvalues of the rough Laplacian on $\pi_{5,4}^{\star}(End_{0}T^{\prime}\mathbb{P}^{2})\rightarrow \mathbb{S}^{5}$. This directly leads to Corollary \ref{Cor 1}.
   \begin{prop}\label{prop multiplicity TP2} (Multiplicities of the eigenvalues in Theorem \ref{Thm spec of rough laplacian P2}) 
  
Under the special data in Theorem \ref{Thm spec of rough laplacian P2}, for any number $\lambda\in Spec\nabla^{\star}\nabla |_{(End_{0}T^{\prime}\mathbb{P}^{2})(l)}$, let the set  $S^{l}_{\lambda}$ be defined by 
\begin{eqnarray*}S^{l}_{\lambda}\triangleq   & & \{(a,b) \in \mathbb{Z}^{\geq 0}\times \mathbb{Z}^{\geq 0}| \frac{4}{3}(a^{2}+b^{2}+ab+3a+3b)-\frac{4}{3}l^{2}-8=\lambda,\ \textrm{and}
\\& &\max(3-a-2b,b-a-3) \leq l\leq \min(2a+b-3,3+b-a)\}.\end{eqnarray*}

The (complex) multiplicity of any  $\lambda_{l}\in Spec\nabla^{\star}\nabla |_{(End_{0}T^{\prime}\mathbb{P}^{2})(l)}$ is 
\begin{equation}\label{equ 0 prop multiplicity TP2}\Sigma_{S^{l}_{\lambda_l}}\frac{(a+1)(b+1)(a+b+2)}{2}.\end{equation}

The (real) multiplicity of any  $\lambda\in Spec\nabla^{\star}\nabla |_{\mathbb{S}^{5}}$ is \begin{equation} \label{equ 1 prop multiplicity TP2}\Sigma_{l, \lambda-l^{2}\in Spec\nabla^{\star}\nabla |_{(End_{0}T^{\prime}\mathbb{P}^{2})(l)}}\Sigma_{S^{l}_{\lambda-l^{2}}}\frac{(a+1)(b+1)(a+b+2)}{2}.\end{equation}

  \end{prop}



Section \ref{sect reductive homogeneous spaces}--\ref{sect proof of spec} below are devoted to the proof of Theorem \ref{Thm spec of rough laplacian P2} and Proposition \ref{prop multiplicity TP2}. 

The proof of Corollary \ref{Cor 1} is completed in Section \ref{sect proof of Cor 1}.
\subsection{Background on reductive homogeneous spaces and homogeneous vector bundles\label{sect reductive homogeneous spaces}}
The representation theoretic method for Theorem \ref{Thm spec of rough laplacian P2} has been well recorded in literature. For example, see \cite{MS} and \cite{CH}. To be self-contained, we recall the background tailored for our purpose. 
\subsubsection{Killing reductive homogeneous spaces}

Our references for this section are \cite[Section 2: geometry of homogeneous spaces]{Koda} and \cite[Section 5, page 13]{Snow}. 

All the group actions below will be left actions unless otherwise specified. All the $G-$actions below are smooth. The $``\cdot"$ between a group element and a vector (in a representation space) means the underlying action (which should be clear from the context).

\begin{Def}\label{Def reductive homogeneous space}(Reductive homogeneous space) Let $G$ be a compact semi-simple matrix Lie group, and $K$ be a closed matrix Lie subgroup of $G$. Let $\mathfrak{g}$ and $\mathfrak{k}$ denote the Lie algebras of $G$ and $K$ respectively. Let $m$ be a subspace of $g$ such that $\mathfrak{g}=m\oplus\mathfrak{k}$.  The manifold $M=G/K$ is called a \textit{reductive homogenous space with respect to $m$} if $Ad_{K}m\subseteq m$ (which means that for any $k\in K$ and $X\in m$, $Ad_{k}X\in m$).

 In practice, we hide the $``m"$ and abbreviate it to  \textit{reductive homogeneous space}. 
\end{Def}

At an arbitrary point $gK\in M$, any $X\in \mathfrak{g}$ generates a tangent vector  $X^{\star}$  in the following way. 
\begin{equation}\label{equ Xstar} X^{\star}(gK)=\frac{d}{dt}|_{t=0}(\exp{tX})gK. 
\end{equation}

Let $E$ be a homogeneous bundle over a reductive homogeneous space $M=G/K$. \textit{Let $e$ denote the identity element in $G$ (and $K$)}. Let the base point  $o\in M$ be $eK$. The natural map $\rho:\ G\times_{K} E_{o}\rightarrow E$ defined by $\rho (g,v)=g\cdot v$ is a $G-$equivariant isomorphism (covering identity diffeomorphism  of $M$). 
  
  On the tangent bundle, let $\tau$ denote the natural isomorphism $G\times_{K,ad}m\rightarrow TM$ defined by 
    \begin{equation}\label{equ tangent bundle iso to associated bundle}\tau(g,X)\triangleq g_{\star}[X^{\star}|_{eK}]=(Ad_{g}X)^{\star}|_{gK}.\end{equation} 
  The tautological isomorphism $\tau_{taut}:\ m\rightarrow T_{o}M$ is defined by $\tau_{taut}(X)=X^{\star}|_{o}$.

 The set of $G-$invariant Riemannian metrics on $M$ is bijective to the set of $Ad_{K}-$invariant inner products on $m$. For example, under the semi-simple condition on $G$, the restriction to $m$ of a negative scalar multiple of the Killing form of $G$ yields a $G-$invariant metric on $M$. 
 This example leads to the notion of a Killing reductive homogeneous space, which is a Riemannian manifold.
 \begin{Def}\label{Def killing reductive homogeneous space}A reductive homogeneous space $M=G/K$ 
  with a $G-$invariant Riemannian metric $(\ ,\ )$ is called a \textit{Killing reductive homogeneous space} with respect to $(\ ,\ )$ and $m$ if the following holds. 
\begin{itemize}\item Let $B$ be the Killing form on $\mathfrak{g}$. With respect to the inner product $-B$, $m$ is perpendicular to the Lie algebra $\mathfrak{k}$ of $K$.
\item The restriction of $-B$ on $m$ is a (constant) real scalar multiple of the inner product $\langle,\rangle_{m}$ corresponding to $(\ ,\ )$. 
\end{itemize}
We usually abbreviate it to  \textit{Killing reductive homogeneous space}, and denote it by $[M,\ \langle,\rangle_{m}]$.
\end{Def}

The following existence of a certain kind of frames enables us to calculate the bundle rough Laplacian by Casimir operators. 
\begin{lem}\label{lem geodesic frame on KRHS} Let $[M,\ \langle,\rangle_{m}]$ be a Killing reductive homogeneous space equipped with the Levi-Civita connection, and let $(e_{i},\ i=1,...,dimM)$ be an orthonormal basis of $m$. Then for any $i$,
\begin{equation}\label{equ geodesic frame at the origin on KRHS}\nabla_{e^{\star}_{i}}e^{\star}_{i}=0\  \textrm{at}\ o=eK.\end{equation}
Consequently,  for any $g\in G$, $[Ad_{g}(e_{i})]^{\star}=g_{\star}(e^{\star}_{i})$ is an orthonormal frame at $gK$ such that 
\begin{equation}\label{equ geodesic frame on KRHS}\nabla_{[Ad_{g}(e_{i})]^{\star}}[Ad_{g}(e_{i})]^{\star}=0\ \textrm{at}\ gK.\end{equation}
\end{lem}

We do not know whether Lemma \ref{lem geodesic frame on KRHS} holds in general if the reductive homogeneous space is not Killing. 

The proof of Lemma \ref{lem geodesic frame on KRHS} is  routine Riemannian geometry calculation. It is deferred to Appendix \ref{Appendix good frame} below. 
\subsubsection{Homogeneous vector bundles}
  Based on the notion of a reductive homogeneous space, we briefly recall the homogeneous bundles. 
  \begin{Def} Let $M=G/K$ be a reductive homogeneous space. A (smooth) vector-bundle $E\rightarrow M$ is said to be $G-$homogeneous 
  if the left action of $G$ on $G/K$ can be lifted to a compatible action of $G$ on $E$. 
  
   We usually suppress the ``$G$" and call it  a homogeneous vector bundle. 
  \end{Def}

   The space of smooth sections of a homogeneous vector bundle can be identified to an $\infty-$dimensional $G-$representation. 
  \begin{Def} \label{Def homogeneous bundle}(Sections of a homogeneous bundle)
 Let $\rho_{\mathcal{E}}:\ K\rightarrow GL(\mathcal{E})$ be a (real or complex) $K-$representation.  We consider the associated vector bundle $E=G\times_{K,\rho_{\mathcal{E}}}\mathcal{E}$.
 
  Let $C^{\infty}(G,\mathcal{E})$ denote the space of all smooth $\mathcal{E}-$valued functions on $G$, and let  $C_{K,\rho_{\mathcal{E}}}^{\infty}(G,\mathcal{E})$ be the subspace of $K-$invariant functions i.e. the functions $f$ such that 
 \begin{equation} \label{equ K invariant functions}f(gk)=\rho_{\mathcal{E}}(k^{-1})f(g)\triangleq k^{-1}\cdot f(g).
 \end{equation}
 We sometimes suppress the representation $\rho_{\mathcal{E}}$ in the notation  $C_{K,\rho_{\mathcal{E}}}^{\infty}(G,\mathcal{E})$. 
 
 A section $u$ of  $E$ defines uniquely a $K-$invariant function in  $C_{K}^{\infty}(G,\mathcal{E})$, denoted by  $\widetilde{u}$. The converse is also true. The correspondence between $u$ and $\widetilde{u}$ is given by 
 \begin{equation}\label{equ correspondence of sections} u(gK)=(g,\widetilde{u})\ \  \textrm{for any}\ \ g\in G. 
 \end{equation}
The same correspondence also holds pointwisely: for any $u\in E|_{gK}$, there is an unique $K-$invariant function $\widetilde{u}$ defined on the $K-$orbit  passing through $g$ such that $u=(g,\widetilde{u})$.
 
  The left regular representation of $G$ on $C_{K}^{\infty}(G,\mathcal{E})$ is defined by
 $$[L(a)\cdot f](g)\triangleq f(a^{-1}g)\ \textrm{for any}\ a,\ g\in G. $$
 We also have the right regular representation of $G$ on $C^{\infty}(G,\mathcal{E})$:
  $$[R(a)\cdot f](g)=f(ga)\ \textrm{for any}\ a,\ g\in G. $$
  Though we do not have a proof in this paper, we expect that in general,\\ $C_{K}^{\infty}(G,\mathcal{E})\subset C^{\infty}(G,\mathcal{E})$ is not necessarily an invariant subspace under the right regular representation, though it apparently is invariant under the left regular representation. 
  
For further references,  see \cite[Section 5.1]{MS} and \cite[III.6]{BD}.
\end{Def}

Regarding the reductive homogeneous spaces and homogeneous bundles  defined so far, it does not harm to keep the following routine convention in mind. 

\textbf{Convention on the $G-$equivariant isomorphism}: from here to the end of Section \ref{sect proof of Cor 1}, 
\begin{itemize}\item the equal signs ``=" between reductive homogeneous spaces or sections of homogeneous bundles, and
\item any correspondence/identification between homogeneous connections, or between sections of homogeneous bundles, or between reductive homogeneous spaces
\end{itemize}
are via the underlying $G-$equivariant isomorphism or diffeomorphism. 
\subsubsection{A useful identity}
For any $X\in m$, the following  lemma calculates the invariant function corresponding to the vector field $X^{\star}$.

  Let $[\ \cdot \ ]_{m}$ be the projection to $m$ with respect to the directly sum $\mathfrak{g}=m\oplus \mathfrak{k}$.
\begin{lem}\label{lem mtildex and xstar}Under the conditions and setting in Definition \ref{Def reductive homogeneous space}, let $\widetilde{m}$ denote the linear map $m\rightarrow C^{\infty}(G, m)$ defined by $$[\widetilde{m}(X)](g)\triangleq [g^{-1}Xg]_{m}.$$ Then for any $X\in m$, $\widetilde{m}(X)$ is $AdK-$invariant. This means that $\widetilde{m}$ is actually a  linear map $m\rightarrow C_{K,Ad}^{\infty}(G, m)$.

As  a vector field on  $M=G/K$,  via the isomorphism \eqref{equ tangent bundle iso to associated bundle}, $\widetilde{m}(X)$ corresponds to $X^{\star}$ i.e. 
\begin{equation}\label{equ lem mtildex and xstar}X^{\star}(gK)=[Ad_{g}\widetilde{m}(X)]^{\star}(gK)=\tau(g,[\widetilde{m}(X)](g))\  \textrm{at any point}\ gK\in M.\end{equation}
\end{lem}
\begin{proof}[Proof of Lemma \ref{lem mtildex and xstar}:] Because $(Ad_{K})m\subseteq m$, $Ad_{K}$ preserves the splitting $Y=Y_{m}+Y_{\mathfrak{k}}$ i.e. \begin{equation}[(Ad_k)Y]_{m}=(Ad_k)[Y]_{m}\ \textrm{for any}\ k\in K.
\end{equation}
Thus $[\widetilde{m}(X)](gk)=[k^{-1}g^{-1}Xgk]_{m}=Ad_{k^{-1}}[g^{-1}Xg]_{m}$. 

To prove the second part, for any $Y\in \mathfrak{g}$, let $Y=[Y]_{m}+[Y]_{\mathfrak{k}}$ (where $[Y]_{\mathfrak{k}}$ is the  $\mathfrak{k}-$component of $Y$), we calculate 
\begin{eqnarray*}& & \{Ad_{g}[\widetilde{m}(X)]\}(g)=g[g^{-1}Xg]_{m}g^{-1}=gg^{-1}Xgg^{-1}-g[g^{-1}Xg]_{\mathfrak{k}}g^{-1}=X-g[g^{-1}Xg]_{\mathfrak{k}}g^{-1}.
\end{eqnarray*}

Because $[g^{-1}Xg]_{\mathfrak{k}}\in \mathfrak{k}$, the tangent vector  $\{g[g^{-1}Xg]_{\mathfrak{k}}g^{-1}\}^{\star}$ at $gK$ is equal to $0$. Then \eqref{equ lem mtildex and xstar} is proved.
\end{proof}

\subsubsection{Rough Laplacian  on a  homogeneous vector bundle and the Casimir operator\label{section rough laplacian and Cas}}
The purpose of this section is to show Formula \ref{formula Cas general} of the rough Laplacian in terms of the Casimir operator. 
\begin{Def} \label{Def Cas}(Casimir operator associated with a basis) Let $\mathfrak{g}$ be a Lie algebra, and $\mathcal{B}=(e_{i},\ i=1...dim\mathfrak{g})$ be a basis of $\mathfrak{g}$. Let $\rho:\ \mathfrak{g}\rightarrow gl(\mathcal{E})$ be a representation of $\mathfrak{g}$. Then we define the Casimir operator with respect to the basis $\mathcal{B}$ by $$Cas^{\mathcal{B}}_{\mathfrak{g},\rho}\triangleq \Sigma^{dim\mathfrak{g}}_{i=1}\rho(e_{i})\rho(e_{i}).$$

\textit{The notion of a Casimir operator associated with a basis is more general than the notion of an usual Casimir operator}. Our definition is tailored for our purpose: on the former,  $\Sigma^{dim\mathfrak{g}}_{i=1}e_{i}\cdot e_{i}$ is  an element in the universal enveloping algebra, but is not in general required to be in the center; on the contrary, the later must be in the center.

Moreover, in the underlying definition, we do not require $G$  to be semi-simple, though  it indeed is  in the case of interest. We do not need the Killing form either. All we need is a basis of the Lie algebra.

\end{Def}
\begin{Def}  Let $M=G/K$ be a reductive homogeneous space with respect to $m$. We view $G$ as a $K-$principal bundle over $M$.  We define the  \textit{left invariant principal connection of $m$} to be the connection of which the horizontal distribution at $g\in G$ is $gm$ (viewed as subspace of left invariant vector fields). On an associated bundle, the connection  given by this horizontal distribution is called  \textit{the connection induced by $m$}, or simply \textit{the induced connection}. 
\end{Def}
\begin{Def}\label{Def triply orthonormal} Let $(M,\ \langle,\rangle_{m})$ be a Killing reductive homogeneous space. A basis $B_{\mathfrak{g}}=(e_{i},\ i=1...dim\mathfrak{g})$ for the Lie algebra $\mathfrak{g}$ is called triply orthonormal if
\begin{itemize}\item $B_{\mathfrak{g}}$ is orthonormal with respect to a negative real scalar multiple of the Killing-form;
\item the set of vectors $B_{\mathfrak{k}}=(e_{i},\ i=1+dimm,...,dim\mathfrak{g})$ form a basis for the Lie algebra $\mathfrak{k}$ of $K$;
\item the set of vectors $B_{m}=(e_{i},\ i=1,...,dimm)$ form  an orthonormal basis of  $m$ with respect to $\langle,\rangle_{m}$. 
\end{itemize}
\end{Def}
In view of the above $3$ definitions  and the notation\ $\widetilde{.}$\ in Definition \ref{Def homogeneous bundle} for the invariant function in terms of a section, we prove the formula for the rough Laplacian.

\begin{formula}\label{formula Cas general} Let $(M,\ \langle,\rangle_{m})$ be a Killing reductive homogeneous space with a triply orthonormal basis $B_{\mathfrak{g}}$ for the Lie algebra $\mathfrak{g}$.  Let $\rho: K\rightarrow GL(\mathcal{E})$ be a (real or complex) representation of $K$. On the homogeneous bundle $G\times_{K,\rho}\mathcal{E}$, the following  holds with respect to induced connection.  $$-\widetilde{(\nabla^{\star}\nabla u)}=(Cas^{\mathcal{B}_{\mathfrak{g}}}_{\mathfrak{g},L}-Cas^{\mathcal{B}_{\mathfrak{k}}}_{\mathfrak{k},\rho})\widetilde{u}.$$ 

\end{formula}

\begin{rmk}\textit{The second  operator $Cas^{\mathcal{B}_{\mathfrak{k}}}_{\mathfrak{k},\rho}$ acts on the value of} $\widetilde{u}$. 
\end{rmk}
\begin{proof}[Proof of Formula \ref{formula Cas general}:]  Let $\widetilde{Y}$ denote the horizontal lift of a tangent vector $Y$, the Kobayashi-Nomizu formula \cite[Vol I, Chap III, page 115]{KN} says that 
\begin{equation}\label{equ KN formula}[\widetilde{Y}(\widetilde{u})](g)=[\widetilde{\nabla_{Y}u|_{gK}}](g).
\end{equation}
For any $e_{i}\in m$, the vector $ge_{i}$ (the value of the left invariant vector field) at $g$ is the horizontal lift of $[Ad_{g}(e_{i})]^{\star}$ at $gK$, using the vanishing \eqref{equ geodesic frame on KRHS}, we compute 
\begin{eqnarray}
& &-\widetilde{(\nabla^{\star}\nabla u)}(g)=\Sigma_{i=1}^{dim M}\{ [Ad_{g}(e_{i})]^{\star}\cdot [Ad_{g}(e_{i})]^{\star}]\widetilde{u} \}(g)=\Sigma_{i=1}^{dim M}[R_{\star}(e_{i})R_{\star}(e_{i})\widetilde{u}](g) \nonumber
\\&=&\Sigma_{i=1}^{dim G}[R_{\star}(e_{i})R_{\star}(e_{i})\widetilde{u}](g) -\Sigma_{i=1+dimM}^{dim G}[R_{\star}(e_{i})R_{\star}(e_{i})\widetilde{u}](g)\nonumber
\\&=&[Cas^{\mathcal{B}_{\mathfrak{g}}}_{\mathfrak{g},R}\widetilde{u}](g)-[Cas^{\mathcal{B}_{\mathfrak{k}}}_{\mathfrak{k},R}\widetilde{u}](g).\label{equ 0 proof Formula Cas}
\end{eqnarray}
For any $g$ and $\widetilde{u}\in C^{\infty}(G,\mathcal{E})$, we compute  
\begin{eqnarray}& &(Cas^{\mathcal{B}_{\mathfrak{g}}}_{\mathfrak{g},R} \cdot \widetilde{u})(g)=[\Sigma_{i=1}^{dimG}R_{\star}(e_{i})R_{\star}(e_{i})\widetilde{u}](g)=\frac{d^{2}}{dsdt}|_{t=s=0}\widetilde{u}(g\exp^{se_{i}}\exp^{te_{i}}) \nonumber\\&=&\frac{d^{2}}{dsdt}|_{t=s=0}\widetilde{u}(g\exp^{se_{i}}g^{-1}\cdot g\exp^{te_{i}}g^{-1}\cdot g) =\Sigma_{i=1}^{dim G}\{[L_{\star}(Ad_{g}e_{i})L_{\star}(Ad_{g}e_{i})]\widetilde{u}\}(g)\nonumber
\\&=&\Sigma_{i=1}^{dim G}\{[L_{\star}(e_{i})L_{\star}(e_{i})]\widetilde{u}\}(g)\ (\textrm{because}\ Ad_{g} \ \textrm{is an orthogonal transformation of}\ \mathfrak{g}) \nonumber
\\&=& (Cas^{\mathcal{B}_{\mathfrak{g}}}_{\mathfrak{g},L} \widetilde{u})(g) \label{equ 1 proof Formula Cas}
\end{eqnarray} 
 This means that on $C^{\infty}(G,\mathcal{E})$ (not requiring $K-$invariance), the Casimir operator of the right regular representation coincides with the Casimir of the left regular representation. 

Because of the $K-$invariance of $\widetilde{u}$, we have $R(e_{i})\widetilde{u}=-\rho(e_{i})\widetilde{u}$ (acting on the value of  $\widetilde{u}$). Hence
\begin{equation}
Cas^{\mathcal{B}_{\mathfrak{k}}}_{\mathfrak{k},R} \cdot\widetilde{u}=\Sigma_{i=1}^{dimK}R_{\star}(e_{i})R_{\star}(e_{i})\widetilde{u}=\Sigma_{i=1}^{dimK}\rho_{\star}(e_{i})\rho_{\star}(e_{i})\widetilde{u}=Cas^{\mathcal{B}_{\mathfrak{k}}}_{\mathfrak{k},\rho} \widetilde{u}. \label{equ 2 proof Formula Cas}
\end{equation}
Applying \eqref{equ 1 proof Formula Cas} and \eqref{equ 2 proof Formula Cas} to the two individual terms in \eqref{equ 0 proof Formula Cas}, the desired formula is proved. 
\end{proof}

 \subsection{The standard connection on the homogeneous bundle\\ $[EndT^{1,0}(\mathbb{P}^{2})](l)$\label{sect the connection on EndTP2}}
The purpose of this section is to interpret $EndT^{\prime}\mathbb{P}^{2}(l)$ as a homogeneous bundle over the homogeneous space $\mathbb{P}^{2}$, and show that the twisted Fubini-Study connection corresponds to the standard horizontal distribution $m_{\mathbb{P}^{2}}$ defined in Section \ref{sect horizontal distribution}. Please see Proposition \ref{prop the canonical connection} for the main statement of this section. 
\subsubsection{The horizontal distribution $m_{\mathbb{P}^{2}}$\label{sect horizontal distribution}}

Recall that  $\mathbb{P}^{2}=SU(3)/ S[U(1)\times U(2)]$. Let the subspace $m_{\mathbb{P}^{2}}\subset su(3)$ be spanned by the following $4$ matrices. 
\begin{eqnarray}\label{equ basis of su3 first 4}& &e_{1}\triangleq X_{1} \triangleq \left[\begin{array}{ccc}0  & 1&   0  \\  -1 & 0&0  \\ 0& 0 & 0 \end{array}\right],\ e_{2}\triangleq Y_{1} \triangleq \left[\begin{array}{ccc}0  & i&   0  \\  i & 0&0  \\ 0& 0 & 0 \end{array}\right],\ 
\\& & e_{3}\triangleq X_{3}\triangleq \left[\begin{array}{ccc}0  & 0&   1  \\  0 & 0&0  \\ -1& 0 & 0 \end{array}\right],\ e_{4}\triangleq Y_{3}\triangleq \left[\begin{array}{ccc}0  & 0&   i  \\  0 & 0&0  \\ i& 0 & 0 \end{array}\right]. \nonumber
\end{eqnarray}
It admits a natural complex structure 
\begin{equation}\label{equ J on mP2}JX_{1}=-Y_{1},\ JY_{1}=X_{1},\ JX_{3}=-Y_{3},\ JY_{3}=X_{3}.
\end{equation}
Then $m_{\mathbb{P}^{2}}$ is naturally isomorphic to  $m^{(1,0)}_{\mathbb{P}^{2}}$ (the $(1,0)-$part of the complexification of $m_{\mathbb{P}^{2}}$). The isomorphism is given by the natural injection $m_{\mathbb{P}^{2}}\rightarrow m_{\mathbb{P}^{2}}\otimes \mathbb{C}$ composed by the projection to the $(1,0)-$part.  $m^{(1,0)}_{\mathbb{P}^{2}}$ is spanned by the vectors 

\begin{equation}\label{equ s} s_{1}=\frac{1}{2}(X_{1}+iY_{1}),\ s_{2}=\frac{1}{2}(X_{3}+iY_{3}).\end{equation} 

It is routine to verify that  $m_{\mathbb{P}^{2}}$ is preserved by $Ad_{S[U(1)\times U(2)]}$. Thus, with respect to the Fubini-Study metric and $m_{\mathbb{P}^{2}}$, $\mathbb{P}^{2}$  is a Killing reductive homogeneous space.

 As complex vector bundles, the homogeneous bundle $SU(3)\times_{S[U(1)\times U(2)],ad} m^{(1,0)}_{\mathbb{P}^{2}}$ is\\ $SU(3)-$equivariantly isomorphic to the holomorphic tangent bundle $T^{1,0}(\mathbb{P}^{2})$. 
  \subsubsection{Interpreting the line bundle $O(l)\rightarrow \mathbb{P}^{n}$ as an associated bundle\label{sect interpreting universal bundle as an associated bundle}} 
    $SU(n+1)$ acts on $\mathbb{C}^{n+1}\setminus O$ which is the total space of $O(-1)$. Let  $S[U(1)\times U(n)]$ denote the subgroup of block-diagonal matrices in $SU(n+1)$ i.e. the matrices in $SU(n+1)$ such that the only non-zero entry in the  first row and the first column is the    $(1,1)-$entry. This means that the matrices have the form $\left[\begin{array}{ccccc}e^{\sqrt{-1}\theta}  & 0& ... & ...&   0 \\  0 & . & . & .& .  \\  \vdots & . &. & .& .  \\ \vdots& . & . & .& .\\ 0& . & . & .& .\end{array}\right]$. A natural group homorphism
$\tau_{S}:\ S[U(1)\times U(n)]\rightarrow U(1)$ maps a matrix in  $S[U(1)\times U(n)]$ to its $(1,1)-$entry. 
  
  For any integer $l$, let $\rho_{l}$ denote the $1-$dimensional complex representation of $U(1)$ i.e. $\rho_{l}(e^{\sqrt{-1}\theta})=e^{\sqrt{-1}l\theta}\in GL(1,\mathbb{C})$. Abusing notation, we still let $\rho_{l}$ denote the $S[U(1)\times U(2)$ representation $\rho_{l}\cdot \tau_{S}$. 
  
  As a homogeneous space, $\mathbb{P}^{n}$ is $SU(n+1)/ S[U(1)\times U(n)]$. The projection map is  
  defined by \begin{equation}\label{equ the pi for Pn}\pi:\ SU(n+1)\rightarrow \mathbb{P}^{n}.\ \ \ \pi(A)\triangleq A\left[\begin{array}{c}1 \\ \vdots \\ 0\end{array}\right]\ \textrm{for any}\ A\in SU(n).\end{equation}
  Obviously, $S[U(1)\times U(n)]$  is the isotropy group of the point $\left[\begin{array}{c}1 \\ \vdots \\ 0\end{array}\right]\in \mathbb{P}^{n}$.

  The action of the isotropy group $S[U(1)\times U(n)]$ on $Span[1,0,...,0]$, the fiber of $O(-1)$ at $[1,0,...,0]$, factors through the standard representation $\rho_{1}$ of $U(1)$. Hence, the universal bundle $$O(-1)\rightarrow \mathbb{P}^{n}$$ is $SU(n+1)-$equivariantly isomorphic to the homogeneous bundle $$SU(n+1)\times_{S[U(1)\times U(n)],\rho_{1}} \mathbb{C}.$$ More generally, for any integer $l$, $O(l)\rightarrow \mathbb{P}^{n}$ is $SU(n+1)-$equivariantly isomorphic to $SU(n+1)\times_{S[U(1)\times U(n)],\rho_{-l}} \mathbb{C}$.%
  \subsubsection{Characterizing the connection of interest}
  The main proposition of Section \ref{sect the connection on EndTP2} is a direct corollary of the following two lemmas addressing the horizontal distribution corresponding to the standard $SU(3)-$invariant connections. 
  \begin{lem}\label{lem connection of O-1} Via the $SU(3)-$equivariant isomorphism $$SU(3)\times_{S[U(1)\times U(2)],\rho_{1}} \mathbb{C}=O(-1)\rightarrow \mathbb{P}^{2},$$ the connection induced by  $m_{\mathbb{P}^{2}}$ corresponds to the standard connection (see Definition \ref{Def standard Hermitian metric on Obeta}).
\end{lem}
\begin{lem}\label{lem induced connection = Fubini Study connection}Via the $SU(3)-$equivariant isomorphism $$SU(3)\times_{S[U(1)\times U(2)],Ad} m^{(1,0)}_{\mathbb{P}^{2}}=T^{\prime}\mathbb{P}^{2},$$ the connection induced by $m_{\mathbb{P}^{2}}$ corresponds to  the Fubini-Study connection.
\end{lem}
The proof of the above two Lemmas is deferred to Appendix \ref{Appendix the standard connection on O(l)} and \ref{Appendix The horizontal distribution of the Fubini-Study connection on the holomorphic tangent bundle}. We are ready for our main proposition about the standard connection on $[EndT^{1,0}(\mathbb{P}^{2})](l)$.
  \begin{prop}\label{prop the canonical connection}On the homogeneous bundle $[EndT^{1,0}(\mathbb{P}^{2})](l)$, 
the tensor product of the Fubini-Study connections (on $T^{1,0}(\mathbb{P}^{2})$ and its dual) and the standard connection
on $O(l)$ is  induced  by the horizontal distribution $m_{\mathbb{P}^{2}}$.   \end{prop}
\begin{proof}[Proof of Proposition \ref{prop the canonical connection}:] Using Lemma \ref{lem connection of O-1}, the standard connection on $O(l)\rightarrow \mathbb{P}^{2}$, obtained by the dual and/or tensor product of the standard connection on $O(-1)$, is induced by  $m_{\mathbb{P}^{2}}$.     Using Lemma \ref{lem induced connection = Fubini Study connection}, the associated connection on the holomorphic co-tangent bundle $\Omega^{1}_{\mathbb{P}^{2}}$, therefore the (tensor product) Fubini-Study connection on  $[EndT^{1,0}(\mathbb{P}^{2})]$, are induced by $m_{\mathbb{P}^{2}}$. Then the tensor product connection on $[EndT^{1,0}(\mathbb{P}^{2})](l)$ is induced by $m_{\mathbb{P}^{2}}$.
\end{proof}

\subsection{Representation theory of $SU(3)$ and $S[U(1)\times U(2)]$, and the proof of Theorem \ref{Thm spec of rough laplacian P2}\label{sect proof of spec} and Proposition \ref{prop multiplicity TP2} }

From here to the end of Section \ref{sect proof of spec}, given a vector space $V$, the symbol $|_{V}$ means ``as an endomorphism of $V$" or ``as a representation on $V$".
\subsubsection{The representation of $S[U(1)\times U(2)]$ on $End_{0}m^{(1,0)}_{\mathbb{P}^{2}}\otimes \mathbb{C}$ and its ``Casimir" operator}
The purpose of this section is to prove Formula \ref{formula of Cas K} on the Casimir operator of the subgroup $S[U(1)\times U(2)]$ of $SU(3)$ (see Section \ref{sect interpreting universal bundle as an associated bundle} for the definition of the subgroup).

As a subgroup of $SU(3)$, the Lie algebra of $S[U(1)\times U(2)]$ is spanned by 
\begin{eqnarray}\label{equ basis of su3 last 4}& &e_{5}\triangleq \widehat{H}_{1}\triangleq\frac{1}{\sqrt{3}} \left[\begin{array}{ccc}2i  & 0&   0  \\  0 & -i&0  \\ 0& 0 & -i \end{array}\right],\ e_{6}\triangleq H_{2}\triangleq \left[\begin{array}{ccc}0  & 0&   0  \\ 0 & i&0  \\ 0& 0 & -i \end{array}\right],\ 
\\& & e_{7}\triangleq X_{2} \triangleq \left[\begin{array}{ccc}0  & 0&   0  \\  0 & 0&1  \\ 0& -1 & 0 \end{array}\right],\ e_{8}\triangleq Y_{2}\triangleq \left[\begin{array}{ccc}0  & 0&   0  \\  0 & 0&i  \\ 0& i & 0 \end{array}\right].\nonumber
\end{eqnarray}

\begin{Def}Let $B_{s[u(1)\times u(2)]}\triangleq \{e_{5}, e_{6},e_{7},e_{8}\}$ be the basis for the Lie algebra\\ $s[u(1)\times u(2)]$ of $S[U(1)\times U(2)]$.
\end{Def}

\begin{rmk}\label{rmk su(2)}$SU(2)$ is isomorphic to the subgroup in $S[U(1)\times U(2)]$ of  block diagonal matrices with $(1,1)-$entry equal to $1$. Henceforth, let $SU(2)$ denote this subgroup. Then $su(2)$ is spanned by $H_{2},\ X_{2},\ Y_{2}$ $(e_{6},\ e_{7},\ e_{8})$. We denote this basis of $su(2)$ by $\mathcal{B}_{su(2)}$. \end{rmk}

The first columns of $(e_{1},\ e_{2},\ e_{3},\ e_{4})$ form an orthonormal set of vectors in $\mathbb{R}^{6}$ (see \eqref{equ basis of su3 first 4}). Thus in view of the formula for the Euclidean metric (K\"ahler form) $\omega_{\mathbb{C}^{3}}$ in Table \eqref{equ tabular volume forms}, it  is straight forward to verify that the quadruple $(e_{1},\ e_{2},\ e_{3},\ e_{4})$ is an orthonormal basis of the inner product $\langle\ ,\ \rangle_{m_{\mathbb{P}^{2}}}$ induced by the Fubini-Study metric $\frac{d\eta}{2}$. 

Let $\mathcal{B}_{su(3)}$ denote the basis $(e_{i},\ i=1...8)$ of $su(3)$. According to the previous paragraph,  it is triply orthonormal on $\mathbb{P}^{2}$. That $\widehat{H}_{1}$ is of the form in \eqref{equ basis of su3 last 4} is important for this triple orthogonality.

Let  $V_{d}$ be the space of all degree $2$ homogeneous polynomials of $2-$complex variables. Let $\rho_{V_{d}}:\ su(2)\rightarrow gl(V_{d})$ be the irreducible representation of $su(2)$ on $V_{d}$. With respect to the notation convention in Definition \ref{Def Cas}, the Casimir operator obeys the following formula
\begin{equation}\label{equ Cas} Cas^{\mathcal{B}_{su(2)}}_{su(2),V_{d}}=-(d^{2}+2d)Id|_{V_{d}}.\ \ \ \textrm{When}\ d=2,\  Cas^{\mathcal{B}_{su(2)}}_{su(2),V_{2}}=-8Id|_{V_{2}}.\end{equation}

We routinely verify the following identities on $ad_{su(2)}|_{m^{(1,0)}_{\mathbb{P}^{2}}}$. 
\begin{equation}[H_{2},s_{1}]=is_{1},\ [H_{2},s_{2}]=-is_{2},\ [X_{2},s_{1}]=-s_{2},\ [X_{2},s_{2}]=s_{1},\ [Y_{2},s_{1}]=is_{2},\ [Y_{2},s_{2}]=is_{1}.
\end{equation}
Therefore, under the basis $(s_{1},s_{2})$ of $m^{(1,0)}_{\mathbb{P}^{2}}$, the representation $ad_{su(2)}$ is given by 
\begin{eqnarray}\label{equ ad on m is standard}& & ad_{H_{2}}\cdot (s_{1},s_{2})=(s_{1},s_{2})\left[\begin{array}{cc} i&0  \\  0 & -i \end{array}\right],\ ad_{X_{2}}\cdot (s_{1},s_{2})=(s_{1},s_{2})\left[\begin{array}{cc} 0&1 \\  -1 & 0\end{array}\right] 
\\& &ad_{Y_{2}}\cdot (s_{1},s_{2})=(s_{1},s_{2})\left[\begin{array}{cc} 0&i \\  i & 0 \end{array}\right].\nonumber
\end{eqnarray}
Let an element in $su(2)$ be represented by its lower block $2\times 2$ (which is exactly the standard form of $su(2)$), the above identities mean that under the basis $(s_{1},s_{2})$, $ad_{su(2)}|_{m^{(1,0)}_{\mathbb{P}^{2}}}$ is the standard representation of $su(2)$.

Based on the above discussion, we can characterize $End_{0}m^{(1,0)}_{\mathbb{P}^{2}}$ as an $su(2)-$representation. 
\begin{lem}\label{lem Endm as su2rep} In view of Remark \ref{rmk su(2)}, the $su(2)$ representation  on  $End_{0}m^{(1,0)}_{\mathbb{P}^{2}}$ inherited from $s[u(1)\times u(2)]$ is $3-$dimensional and irreducible. Consequently,  it is equivalent to $\rho_{V_{2}}$.
\end{lem}
\begin{proof}[Proof of Lemma \ref{lem Endm as su2rep}:] Because $ad_{su(2)}|_{m^{(1,0)}_{\mathbb{P}^{2}}}$ is (equivalent to) the standard representation of $su(2)$ (see \eqref{equ ad on m is standard}), it suffices to show  that  the $su(2)-$representation on $End_{0}\mathbb{C}^{2}$  induced by the standard representation is irreducible.

  By definition, we routinely verify that the $su(2)-$representation on $End_{0}\mathbb{C}^{2}$ extends complex linearly to the adjoint representation of $sl(2,\mathbb{C})$. Because $sl(2,\mathbb{C})$ is simple, the adjoint action must be irreducible. Therefore it is also irreducible as a $su(2)-$representation. The proof is complete.
\end{proof}

To calculate the Casimir operator of $S[U(1)\times U(2)]$, it remains to understand the adjoint action of $e_{5}=\widehat{H}_{1}$ on $m^{(1,0)}_{\mathbb{P}^{2}}$.
\begin{lem}\label{lem ad e5}$ad_{\widehat{H}_{1}}|_{m^{(1,0)}_{\mathbb{P}^{2}}}=-\sqrt{3}i Id|_{m^{(1,0)}_{\mathbb{P}^{2}}}$.
\end{lem}

\begin{proof}[Proof of Lemma \ref{lem ad e5}:] We straight-forwardly verify the following.
\begin{equation*}[\widehat{H}_{1},X_{1}]=\sqrt{3}Y_{1},\ [\widehat{H}_{1},Y_{1}]=-\sqrt{3}X_{1},\ [\widehat{H}_{1},X_{3}]=\sqrt{3}Y_{3},\ [\widehat{H}_{1},Y_{3}]=-\sqrt{3}X_{3}.
\end{equation*}
Then $[\widehat{H}_{1},s_{1}]=-\sqrt{3}is_{1}$, $[\widehat{H}_{1},s_{2}]=-\sqrt{3}is_{2}$.
\end{proof}

 Using that the representation of $su(2)$ on $\mathbb{C}$ is trivial (the image is the $0-$endomorphism), we compute the $su(2)-$Casimir on the representation of interest.  

\begin{eqnarray}& &\Sigma_{i=6}^{8}[(ad\otimes \rho_{-l})_{\star}(e_{i})(ad\otimes \rho_{-l})_{\star}(e_{i})] |_{End_{0}m^{(1,0)}_{\mathbb{P}^{2}}\otimes \mathbb{C}}\nonumber\\&=&\Sigma_{i=6}^{8}[(ad_{\star}(e_{i})(ad)_{\star}(e_{i})]|_{End_{0}m^{(1,0)}_{\mathbb{P}^{2}}}\otimes Id |_{\mathbb{C}}\nonumber
\\&=&-8Id|_{End_{0}m^{(1,0)}_{\mathbb{P}^{2}}\otimes \mathbb{C}}. \label{equ Cas su(2) on End otimes C}
\end{eqnarray}

 Elementary calculation  yields the action of $e_{5}$ via $\rho_{l}$:
\begin{equation}\label{equ action of rhol e5}\rho_{-l}(e_{5})|_{\mathbb{C}}=-\frac{2li}{\sqrt{3}}Id|_{\mathbb{C}},\ \textrm{consequently}\ [\rho_{-l}(e_{5})\rho_{-l}(e_{5})]|_{\mathbb{C}}=-\frac{4l^{2}}{3}Id|_{\mathbb{C}}. 
\end{equation}On the other hand, Lemma \ref{lem ad e5} says that $ad_{e_{5}}|_{m^{(1,0)}_{\mathbb{P}^{2}}}$ is a (complex) scalar multiple of the identity.  Thus  $ad_{e_{5}}|_{End_{0}m^{(1,0)}_{\mathbb{P}^{2}}}=0$. We obtain  
\begin{eqnarray}& & [(ad\otimes \rho_{-l})_{\star}(e_{5})(ad\otimes \rho_{-l})_{\star}(e_{5})] |_{End_{0}m^{(1,0)}_{\mathbb{P}^{2}}\otimes \mathbb{C}}=Id_{End_{0}m^{(1,0)}_{\mathbb{P}^{2}}}\otimes [ \rho_{-l,\star}(e_{5})\rho_{-l,\star}(e_{5})\mathbb{C}]\nonumber
\\&=&  -\frac{4l^{2}}{3}Id|_{End_{0}m^{(1,0)}_{\mathbb{P}^{2}}\otimes \mathbb{C}}.\label{Cas e5 rhol}
\end{eqnarray}

Combining \eqref{equ Cas su(2) on End otimes C} and \eqref{Cas e5 rhol}, we arrive at the following. 
 \begin{formula}\label{formula of Cas K} \begin{eqnarray}\nonumber& &Cas^{\mathcal{B}_{s[u(1)\times u(2)]}}_{s[u(1)\times u(2)],ad\otimes \rho_{-l}}|_{End_{0}m^{(1,0)}_{\mathbb{P}^{2}}\otimes \mathbb{C}}\triangleq \Sigma_{i=5}^{8}[(ad\otimes \rho_{-l})_{\star}(e_{i})(ad\otimes \rho_{-l})_{\star}(e_{i})]|_{End_{0}m^{(1,0)}_{\mathbb{P}^{2}}\otimes \mathbb{C}}
 \\&=&(-8 -\frac{4l^{2}}{3})Id |_{End_{0}m^{(1,0)}_{\mathbb{P}^{2}}\otimes \mathbb{C}}.\end{eqnarray}  \end{formula}

\subsubsection{The translation between two conventions of $SU(3)-$representations\label{set The translation between two conventions of SU(3) repre}}


Let $W_{1,0}$ be the standard representation of $su(3)$ on $\mathbb{C}^{3}$, and $W_{0,1}$ be the dual representation of $W_{1,0}$. Let $W_{a,b}$ 
be the irreducible representation generated by the highest weight vector in the tensor product representation $W^{\otimes a}_{1,0}\otimes W^{\otimes b}_{0,1}$ (see \cite[II.5]{Hall}). 
Any irreducible representation of $su(3)$ is equivalent to  $W_{a,b}$ for some integer $a,\ b\geq 0$ ($W_{0,0}$ is the $1-$dimensional trivial representation). 

\begin{Notation} In \cite{IT}, a $SU(3)$ irreducible representation is labelled by an integer linear combination of the two weights $x^{\star}_{1},\ x^{\star}_{2}$. We denoted it by $V^{SU(3)}_{m_{1}x^{\star}_{1}+m_{2}x^{\star}_{2}}$, and this is said to be the \textit{Ikeda-Taniguchi convention}. In \cite{IT}, such integer linear combinations also label the irreducible $S[U(1)\times U(2)]-$representations. We denote it by $V^{S[U(1)\times U(2)]}_{k_{1}x^{\star}_{1}+k_{2}x^{\star}_{2}}$.
\end{Notation}

We need the following translation from the Ikeda-Taniguchi convention to the (usual) $W_{a,b}-$convention. 
\begin{lem}\label{equ Wab and IT convention}The irreducible representation $W_{a,b}$ of $SU(3)$ is isomorphic to Ikeda-Taniguchi's $V^{SU(3)}_{(a+b)x^{\star}_{1}+b x^{\star}_{2}}$.
\end{lem}
\begin{proof}[Proof of Lemma \ref{equ Wab and IT convention}:] It is an algebra exercise  to verify that in Ikeda-Taniguchi convention, the standard representation $W_{1,0}$ of $su(3)$ has highest weight $x^{\star}_{1}$, the dual representation $W_{0,1}$ has highest weight $-x^{\star}_{3}$, which is equal to $x^{\star}_{1}+x^{\star}_{2}$. 

 The highest weight of a (possibly multiple) tensor product of irreducible $su(3)-$representations is the sum of the highest weight of each one. Thus, the highest weight of $W^{\otimes a}_{1,0}\otimes W^{\otimes b}_{0,1}$ (in Ikeda-Taniguchi convention) is $ax^{\star}_{1}-bx^{\star}_{3}$, which is equal to $(a+b)x^{\star}_{1}+bx^{\star}_{2}$. Because  $W_{a,b}$ is the irreducible representation generated by the highest weight vector in $W^{\otimes a}_{1,0}\otimes W^{\otimes b}_{0,1}$, the  highest weight of $W_{a,b}$  is the same i.e. $(a+b)x^{\star}_{1}+bx^{\star}_{2}$. 

\end{proof}

\subsubsection{The irreducible $S[U(1)\times U(2)]-$representation $End_{0}m^{(1,0)}_{\mathbb{P}^{2}}\otimes \mathbb{C}$ in Ikeda-Taniguchi convention}
The Cartan sub-algebra $\Upsilon_{su(3)}$ of $su(3)$ consists of diagonal traceless matrices with purely imaginary diagonal entries. 
Let $x^{\star}_{i}$ maps any matrix in $\Upsilon_{su(3)}$ to  its $i-$th diagonal entry. Then $x^{\star}_{1},\ x^{\star}_{2},\ x^{\star}_{3}$ are roots. They are subject to the relation $x^{\star}_{1}+x^{\star}_{2}+x^{\star}_{3}=0$. According to  \cite[Section 5, Page 529]{IT}, the partial ordering is determined by $$x^{\star}_{1}>x^{\star}_{2}>0>x^{\star}_{3}.$$  Moreover, $m^{(1,0)}_{\mathbb{P}^{2}}$ has highest weight $x^{\star}_{2}-x^{\star}_{1}$, and the dual $m^{(1,0),\star}_{\mathbb{P}^{2}}$ has highest weight $x^{\star}_{1}-x^{\star}_{3}$ (see \cite[page 532, (iii)]{IT}). Then the highest weight of $End m^{(1,0)}_{\mathbb{P}^{2}}=m^{(1,0)}_{\mathbb{P}^{2}}\otimes m^{(1,0),\star}_{\mathbb{P}^{2}}$ is $x^{\star}_{1}+2x^{\star}_{2}$: the sum of the highest weights of $m^{(1,0)}_{\mathbb{P}^{2}}$ and  $m^{\star}_{\mathbb{P}^{2}}$.

\begin{lem}\label{lem weights for SU1U2}The tensor product representation $$Ad\otimes \rho_{-l}:\ S[U(1)\times U(2)]\rightarrow GL(End_{0}m^{(1,0)}_{\mathbb{P}^{2}}\otimes \mathbb{C})$$
is equivalent to $V^{S[U(1)\times U(2)]}_{(-l+1)x^{\star}_{1}+2x^{\star}_{2}}$. 
\end{lem}
\begin{proof}[Proof of Lemma \ref{lem weights for SU1U2}:] Because $S[U(1)\times U(2)]$ is a subgroup of $SU(3)$ having the same Cartan sub-algebra $\Upsilon_{su(3)}$, the highest weight on $End_{0}m^{(1,0)}_{\mathbb{P}^{2}}$ is the same as the highest weight as a $SU(3)-$representation, which is equal to $x_{1}^{\star}+2x^{\star}_{2}$.
 Because $\mathbb{C}$ is $1-$dimensional, the only weight for $\rho_{-l}$ is $-lx^{\star}_{1}$. Then in Ikeda-Taniguchi convention, the representation $Ad\otimes \rho_{-l}$ is denoted by $V^{S[U(1)\times U(2)]}_{(-l+1)x^{\star}_{1}+2x^{\star}_{2}}$. 
\end{proof}


\subsubsection{The infinite dimensional $SU(3)-$representation of invariant functions}
Using the translation in \eqref{equ Wab and IT convention} between two different conventions, we re-state the result of Ikeda-Taniguchi in the following. 
\begin{fact}\label{fact IT} (Ikeda-Taniguchi \cite[Proposition 5.1, Proposition 1.1]{IT})\\ Let $l$ be an integer, and let $a,b$ be nonnegative integers.  $W_{a,b}$ appears as an irreducible summand in $C^{\infty}_{S[U(1)\times U(2)], Ad\otimes \rho_{-l}}(SU(3),End_{0}m^{(1,0)}_{\mathbb{P}^{2}}\otimes \mathbb{C})$ if and only if \begin{equation}\label{equ fact IT} \max(3-a-2b,b-a-3) \leq l\leq \min(2a+b-3,3+b-a).\end{equation}
\end{fact}
\begin{proof}[Proof of Fact \ref{fact IT}:] The representation $SU(3)-$representation $W_{a,b}$ is also a representation of the subgroup $S[U(1)\times U(2)]$ by restriction.  The Frobenius reciprocal theorem (for example, see  \cite[Proposition 1.1]{IT})) implies that  the following two conditions are equivalent. 
\begin{itemize}\item As $SU(3)-$representations,  $W_{a,b}$ appears as an irreducible summand in\\ $C^{\infty}_{S[U(1)\times U(2)], Ad\otimes \rho_{-l}}(SU(3),End_{0}m^{(1,0)}_{\mathbb{P}^{2}}\otimes \mathbb{C})$. 
\item As $S[U(1)\times U(2)]-$representations, $(End_{0}m^{(1,0)}_{\mathbb{P}^{2}}\otimes \mathbb{C},Ad\otimes \rho_{-l})$ appears as an irreducible summand in $W_{a,b}$. 
\end{itemize}
It suffices to determine for which $a,\ b$ the latter happens.  

\cite[Proposition 5.1]{IT} states that $V^{S[U(1)\times U(2)]}_{k_{1}x^{\star}_{1}+k_{2}x^{\star}_{2}}$ appear as an irreducible summand in $V^{SU(3))}_{m_{1}x^{\star}_{1}+m_{2}x^{\star}_{2}}$ if and only if the following holds.
\begin{equation}m_{1}\geq k_{2}+k\geq m_{2}\geq k\geq 0,\ \textrm{and}\ \ k_{1}=m_{1}+m_{2}-k_{2}-3k.
\end{equation}
Because of Lemma \ref{equ Wab and IT convention} and \ref{lem weights for SU1U2}, to verify the second bullet point above, it suffices to let $m_{1}=a+b$, $m_{2}=b$, $k_{1}=-l+1$, $k_{2}=2$. Then the second bullet point holds if and only if 
\begin{equation}\label{equ 1 proof fact IT}a+b\geq 2+k\geq b\geq k\geq 0,\ \textrm{and}\ \ 3k=a+2b-3+l.
\end{equation}
Elementary calculation shows \eqref{equ 1 proof fact IT} is equivalent to \eqref{equ fact IT}.
 \end{proof}
\subsubsection{Proof of Theorem \ref{Thm spec of rough laplacian P2} and Proposition \ref{prop multiplicity TP2} }
In conjunction with the notation convention in Definition \ref{Def Cas} above, the known formula for the quadratic Casimir operator of $su(3)$ states: 
\begin{formula}\label{equ Cas su3}$Cas^{\mathcal{B}_{su(3)}}_{su(3),W_{a,b}}=\Sigma_{i=1}^{8}(e_{i}|_{W_{a,b}})\cdot (e_{i}|_{W_{a,b}})=(-\frac{4}{3}a^{2}-\frac{4}{3}b^{2}-4a-4b-\frac{4}{3}ab)Id. $
\end{formula}

The tools at our disposal now  can be assembled to achieve our goal. 
\begin{proof}[\textbf{Proof of Theorem \ref{Thm spec of rough laplacian P2} and Proposition \ref{prop multiplicity TP2}}:] We first prove Theorem \ref{Thm spec of rough laplacian P2}.\eqref{equ spec Laplacian P2}. It is a direct corollary of  Fact \ref{fact IT} on the irreducible summand of the infinite dimensional representation, Formula \ref{equ Cas su3} for the Casimir operator of $su(3)$, Formula \ref{formula of Cas K} for $Cas^{\mathcal{B}_{s[u(1)\times u(2)]}}_{s[u(1)\times u(2)],ad\otimes \rho_{-l}}$, 
and the general Formula \ref{formula Cas general} for rough Laplacian on a homogeneous bundle over a Killing reductive homogeneous space. 

Because of the $SU(3)-$equivariant isomorphism: $$(End_{0}T^{\prime}\mathbb{P}^{2})(l)\rightarrow SU(3)\times_{S[U(1)\times U(2)], Ad\otimes
\rho_{-l}}[(End_{0}m^{(1,0)}_{\mathbb{P}^{2}})\otimes \mathbb{C}],$$ 
the general formula \ref{formula Cas general} for $G=SU(3)$, $K=S[U(1)\times U(2)]$, and $\rho=Ad\otimes \rho_{-l}$  says that the spectrum of the rough Laplacian is equal to the spectrum of \begin{equation}\label{equ -Cas + Cas}-Cas^{\mathcal{B}_{su(3)}}_{su(3),L}+ Cas^{\mathcal{B}_{s[u(1)\times u(2)]}}_{s[u(1)\times u(2)],Ad\otimes \rho_{-l}}\end{equation} on the space $C^{\infty}_{S[U(1)\times U(2)], Ad\otimes \rho_{-l}}(SU(3),End_{0}m^{(1,0)}_{\mathbb{P}^{2}}\otimes \mathbb{C})$ of invariant functions. 

On the whole  $C^{\infty}_{S[U(1)\times U(2)], Ad\otimes \rho_{-l}}(SU(3),End_{0}m^{(1,0)}_{\mathbb{P}^{2}}\otimes \mathbb{C})$, by Formula \ref{formula of Cas K}, \\ $Cas^{\mathcal{B}_{s[u(1)\times u(2)]}}_{s[u(1)\times u(2)],Ad\otimes \rho_{-l}}$ acts by $-(\frac{4}{3}l^{2}+8)Id$. In the Peter-Weyl formulation (see the presentation in \cite[Section 5.1]{MS}), as  $SU(3)-$representations, on each irreducible summand $W_{a,b}$ of $C^{\infty}_{S[U(1)\times U(2)], Ad\otimes \rho_{-l}}(SU(3),End_{0}m^{(1,0)}_{\mathbb{P}^{2}}\otimes \mathbb{C})$, Formula \ref{equ Cas su3}  says that the action of  $-Cas^{\mathcal{B}_{su(3)}}_{su(3),L}$ is the scalar multiplication by $\frac{4}{3}(a^{2}+b^{2}+ab+3a+3b)Id$. Then on the irreducible summand $W_{a,b}$, 
the action \eqref{equ -Cas + Cas} is the scalar multiplication by $$\frac{4}{3}(a^{2}+b^{2}+ab+3a+3b)-\frac{4}{3}l^{2}-8.$$

 Fact \ref{fact IT}  says that $W_{a,b}$ appears as an irreducible summand if and only if the condition on the right side of  \eqref{equ spec Laplacian P2} holds.  The proof of Theorem \ref{Thm spec of rough laplacian P2}.\eqref{equ spec Laplacian P2} is  complete. 

 Hence, Theorem \ref{Thm spec of rough laplacian P2}.\eqref{equ proof of Thm rough laplacian} directly follows from the spectral splitting in Formula \ref{formula laplace on S5 vs laplace on CP2}  and Theorem \ref{Thm spec of rough laplacian P2}.\eqref{equ spec Laplacian P2}: we only need to add $l^{2}$ to the $\frac{4}{3}(a^{2}+b^{2}+ab+3a+3b)-\frac{4}{3}l^{2}-8$ in \eqref{equ spec Laplacian P2}.
  
  Next, we address the multiplicities. It is evident from  the first $4$ paragraphs in the underlying proof that the eigenspace of any $\lambda_{l}\in Spec \nabla^{\star}\nabla |_{(End_{0}T^{\prime}\mathbb{P}^{2})(l)}$ is isomorphic to the direct sum of all those $W_{a,b}$ such that 
  \begin{itemize}\item $W_{a,b}$ is a summand in $C^{\infty}_{S[U(1)\times U(2)], Ad\otimes \rho_{-l}}(SU(3),End_{0}m^{(1,0)}_{\mathbb{P}^{2}}\otimes \mathbb{C})$ i.e. the conditions for $a,\ b,$ on the right side of \eqref{equ spec Laplacian P2} holds;
  \item the value of $\frac{4}{3}(a^{2}+b^{2}+ab+3a+3b)-\frac{4}{3}l^{2}-8$ is equal to $\lambda_{l}$. 
  \end{itemize}
  In the terminology of Proposition  \ref{prop multiplicity TP2}, the above precisely means that $(a,b)\in S^{l}_{\lambda_{l}}$. The proof of Proposition \ref{prop multiplicity TP2}.\eqref{equ 0 prop multiplicity TP2} is complete.  
  
  Hence, Proposition \ref{prop multiplicity TP2}.\eqref{equ 1 prop multiplicity TP2} follows by \eqref{equ 0 prop multiplicity TP2}, and the spectral splitting \eqref{equ spec mul} counting multiplicity. \end{proof}
  \subsection{The  proof of Corollary \ref{Cor 1}\label{sect proof of Cor 1}}
 
Corollary \ref{Cor 1} can be proved using Theorem \ref{Thm spec of rough laplacian P2}, \ref{Thm 1}, and the following Lemma on cohomology. 
 \begin{lem}\label{lem h1 EndTP2} $h^{1}[\mathbb{P}^{2},\ (EndT^{\prime}\mathbb{P}^{2})(l)]=\left\{ \begin{array}{cc} 3\ &  \textrm{if}\ l=-1\ \textrm{or}\ -2,\\
 0\ & \textrm{otherwise}. \end{array}\right.$
 
Consequently, $h^{0}[\mathbb{P}^{2},\ (End_{0}T^{\prime}\mathbb{P}^{2})(l)]=\left\{ \begin{array}{cc} \frac{3l(l+3)}{2}\ &  \textrm{if}\ l> 0,\\
 0\ &  \textrm{if}\ l\leq 0. \end{array}\right.$
 \end{lem}
 The proof of Lemma \ref{lem h1 EndTP2} is completely routine via Euler sequence and Bott formula for sheaf cohomology on $\mathbb{P}^{n}$ (see \cite{Okonek}). We defer it to Appendix \ref{Appendix Some algebro-geometric calculations}.
 

\begin{proof}[\textbf{Proof of Corollary \ref{Cor 1}}:]

 Under the setting of Theorem \ref{Thm 1}, when $$\pi^{\star}_{5,4}EndE= \pi^{\star}_{5,4}End(T^{\prime}\mathbb{P}^{2})$$ is equipped with the pullback Fubini-Study connection, Lemma \ref{lem h1 EndTP2} means that except when  $l\neq -1$ or $-2$, the sheaf cohomology   has no contribution to $Spec P$.

 On the other hand, Theorem \ref{Thm spec of rough laplacian P2} addresses the source $Spec\nabla^{\star}\nabla|_{\mathbb{S}^{5}}$ of the other part of $SpecP$. 
 
 We need the fact that any $W_{a,b}$ appears in the infinite-dimensional representation at most once. This is because the representation of the associated bundle is irreducible (see Lemma \ref{lem weights for SU1U2}). Please see the Frobenius reciprocal theorem (stated in \cite[Proposition 1.1]{IT}), and also \cite[Proposition 5.1]{IT}. 
 
 We seek for those eigenvalues of $\nabla^{\star}\nabla |_{\mathbb{S}^{5}}$ that is strictly less than $8$.  When $l\geq 3$, because of the ``$+l^{2}$" in Formula \ref{formula laplace on S5 vs laplace on CP2}.\eqref{equ spec mul}, the  eigenvalues of $\nabla^{\star}\nabla |_{\mathbb{S}^{5}}$ generated are $\geq 9$. Thus, it suffices to assume $-2\leq l\leq 2$ and  seek for those eigenvalues of $\nabla^{\star}\nabla |_{(End_{0}T^{\prime}\mathbb{P}^{2})(l)}$ that is strictly less than $8$.

  Under the conditions on $a,\ b$ in Theorem \ref{Thm spec of rough laplacian P2}.$\eqref{equ spec Laplacian P2}$, elementary calculation shows that this  can only happen for the following values of $l,\ a, \ b$.
 \begin{itemize}\item $l=0,\ (a,b)=(1,1)$. In this case, the corresponding eigenvalue of  $\nabla^{\star}\nabla |_{End_{0}T^{\prime}\mathbb{P}^{2}}$ is $4$, the eigenspace is isomorphic to $W_{1,1}$.
 \item $l=1,\ (a,b)=(2,0)$. In this case, the corresponding eigenvalue of  $\nabla^{\star}\nabla |_{(End_{0}T^{\prime}\mathbb{P}^{2})(1)}$ is $4$, the eigenspace is isomorphic to $W_{2,0}$.
  \item $l=-1,\ (a,b)=(0,2)$. In this case, the corresponding eigenvalue of  $\nabla^{\star}\nabla |_{(End_{0}T^{\prime}\mathbb{P}^{2})(-1)}$ is $4$, the eigenspace is isomorphic to $W_{0,2}$.
 \end{itemize}
 
 Then, still according to Formula \ref{formula laplace on S5 vs laplace on CP2}.\eqref{equ spec mul}, the above three cases generate the numbers $4$ and $5$ in  $Spec \nabla^{\star} \nabla |_{\mathbb{S}^{5}}$. The multiplicity of $4$ is equal to $dimW_{1,1}=8$, the multiplicity of $5$ is equal to $dimW_{2,0}+dimW_{0,2}=12$. 
 
The  number $4$ in  $Spec \nabla^{\star} \nabla |_{\mathbb{S}^{5}}$ generates the following values in $SpecP$. $$2\sqrt{2}-1,\ 2\sqrt{2}-2,\ -1-2\sqrt{2},\ -2-2\sqrt{2}.$$

The  number $5$ in  $Spec \nabla^{\star} \nabla |_{\mathbb{S}^{5}}$ generates the following values in $Spec P$. $$1,\ 2,\ -4,\ -5.$$

Apparently, among the above $8$ numbers, $2\sqrt{2}-2$ and $1$  are the only ones in the interval $(0,1]$. 

Because $2\sqrt{2}-2$ is not an integer, its multiplicity is 16 i.e.  twice of the multiplicity of $4\in Spec \nabla^{\star} \nabla |_{\mathbb{S}^{5}}$. 

The other number $1$ is an integer, in view of Lemma \ref{lem h1 EndTP2},  the multiplicity is  
$$2dim\mathbb{E}_{5} (\nabla^{\star} \nabla |_{\mathbb{S}^{5}})- 2h^{0}[\mathbb{P}^{2},\ (End_{0}T^{\prime}\mathbb{P}^{2})(1)]=24-12=12.$$

The proof of Table \eqref{equ tabular eigenvalue and multiplicity}  is complete. This is a sample of how the multiplicity of each eigenvalue of $P$ is determined. 
 
  \end{proof}

\appendix
\appendixpage 

\section{Elementary Sasakian geometry\label{sect appendix Sasakian}}Still let $D$ denote the contact distribution, and let $D^{\star}$ denote the contact co-distribution. In the main body and in the subsequent Appendix, we frequently  appeal to the following identities on Sasakian geometry of $\mathbb{S}^{5}\rightarrow \mathbb{P}^{2}$ (please also see \cite{Sparks}). 
\begin{equation}\label{equ 1 Appendix}
\textrm{For any section}\ X\ \textrm{of the contact distribution}\ D,\ \nabla_{X}\xi=J_{0}(X),\ \nabla_{X}\eta=[J_{0}(X)]^{\sharp}.
\end{equation}
We also have the following formulas for Hessians of the Reeb vector field and contact form. For any point $p\in \mathbb{S}^{5}$ and $X,\ Y \in D|_{p}$, the following is true.   \begin{equation}(\nabla^{2}\xi)(X,Y)=-<X,Y>\xi,\ (\nabla^{2}\eta)(X,Y)=-<X,Y>\eta.\  \textrm{Consequently},\
 \nabla^{\star}\nabla\eta=4\eta.\label{equ hessian of xi and eta}
\end{equation}

For the point-wise calculations in the proof of  Lemma \ref{lem formula of the model dirac deformation operator} and others, it is helpful to have a transverse geodesic frame in the following sense. 
\begin{lem}\label{lem Kahler geodesic frame} (Properties of a  transverse geodesic frame)  Let $(x_{i},\ i=1,...,4)$ be a K\"ahler geodesic coordinate with respect to the (Fubini-Study) metric $\frac{d\eta}{2}$ at (near) an arbitrary point $[p]\in \mathbb{P}^{2}$. Then, for any $\beta$ among $0,\ 1,\ 2$ such that $[p]\in U_{\beta,\mathbb{P}^{2}}$, the following vector fields $$[\xi;\ v_{i}\triangleq \frac{\partial}{\partial x_{i}}-\eta(\frac{\partial}{\partial x_{i}})\xi,i=1,2,3,4]$$ is a frame near the Reeb orbit $\pi^{-1}_{5,4}[p]$, and is orthonormal on $\pi^{-1}_{5,4}[p]$. Moreover, the following holds on the Reeb orbit.  $$(\nabla_{v_{i}}v_{i})|_{\pi^{-1}_{5,4}[p]}=0,\ [\nabla_{v_{i}}(J_{0}v_{i})]|_{\pi^{-1}_{5,4}[p]}=-\xi.$$ 
\end{lem}
Near the Reeb orbit, we call the $(v_{i},\ i=1,2,3,4)$ above a transverse geodesic frame. It is generated by the geodesic coordinate on $\mathbb{P}^{2}$. 
\begin{proof}[Proof of Lemma \ref{lem Kahler geodesic frame}:] We first show 
\begin{equation}\label{equ 0 proof Lem Kah geo coordinate}[v_{i},v_{j}]=[d\eta (\frac{\partial}{\partial x_{j}},\frac{\partial}{\partial x_{i}})]\xi.
\end{equation}
 Because the Reed vector-field $\xi$ is a coordinate vector field in $U_{\beta, \theta_{\beta}}$ for any $\beta=0,1,$ or $2$, we have the vanishing \begin{equation}\label{equ xi ddxi commute}[\xi,\frac{\partial}{\partial x_{i}}]=0\ \textrm{for any}\ i.\end{equation} Then we calculate
\begin{eqnarray}& &[v_{i},v_{j}]=[\frac{\partial}{\partial x_{i}}-\eta(\frac{\partial}{\partial x_{i}})\xi,\ \frac{\partial}{\partial x_{j}}-\eta(\frac{\partial}{\partial x_{j}})\xi]=[-\frac{\partial}{\partial x_{i}}\eta(\frac{\partial}{\partial x_{j}})+\frac{\partial}{\partial x_{j}}\eta(\frac{\partial}{\partial x_{i}})]\xi\nonumber
\\&=& [d\eta (\frac{\partial}{\partial x_{j}},\frac{\partial}{\partial x_{i}})]\xi.
\end{eqnarray}
The proof of \eqref{equ 0 proof Lem Kah geo coordinate} is complete.

The vanishing \eqref{equ xi ddxi commute} above implies the following vanishing. 
\begin{equation}\label{equ xi vi commute}[\xi,v_{i}]=0\ \textrm{for any}\ i.\end{equation} The identity \eqref{equ 0 proof Lem Kah geo coordinate} implies that the Lie bracket of $v_{i}$ and $v_{j}$ is perpendicular to both $v_{i}$ and $v_{j}$. Then, using the Koszul formula \cite[page 25]{Petersen}, we find 
\begin{equation}\label{equ 1 proof Lem Kah geo coordinate}2\langle\nabla_{v_{i}}v_{j}, v_{k}\rangle=v_{i}\langle v_{j},v_{k} \rangle-v_{k}\langle v_{i},v_{j}\rangle+v_{j}\langle v_{k},v_{i}\rangle. 
\end{equation}
We recall the following formula for the standard metric on $\mathbb{S}^{5}$. 
\begin{equation}g_{\mathbb{S}^{5}}=\pi^{\star}_{5,4}g_{FS}+\eta\otimes \eta. 
\end{equation}
Then \eqref{equ 1 proof Lem Kah geo coordinate} implies 
\begin{eqnarray}\nonumber\label{equ 2 proof Lem Kah geo coordinate}& &2\langle\nabla_{v_{i}}v_{j}, v_{k}\rangle=\frac{\partial}{\partial x_{i}}\langle \frac{\partial}{\partial x_{j}},\frac{\partial}{\partial x_{k}} \rangle_{\mathbb{P}^{2}}-\frac{\partial}{\partial x_{k}}\langle \frac{\partial}{\partial x_{i}},\frac{\partial}{\partial x_{j}}\rangle_{\mathbb{P}^{2}}+\frac{\partial}{\partial x_{j}}\langle \frac{\partial}{\partial x_{k}},\frac{\partial}{\partial x_{i}}\rangle_{\mathbb{P}^{2}}
\\&=& 2\langle \nabla^{FS}_{\frac{\partial}{\partial x_{i}}} \frac{\partial}{\partial x_{j}},\frac{\partial}{\partial x_{k}} \rangle_{\mathbb{P}^{2}}=2\langle \nabla^{FS}_{\pi_{5,4,\star}v_{i}}\pi_{5,4,\star}v_{j},\pi_{5,4,\star}v_{k} \rangle_{\mathbb{P}^{2}} \nonumber 
\\&=& 2\langle (\pi^{\star}_{5,4}\nabla^{FS})_{v_{i}} v_{j},v_{k} \rangle\end{eqnarray}
Moreover, the Lie bracket identity \eqref{equ xi vi commute} and the Koszul formula yield that 
\begin{equation}\label{equ 3 proof Lem Kah geo coordinate}2\langle\nabla_{v_{i}}v_{j}, \xi\rangle=<[v_{i},v_{j}],\xi>=d\eta(\frac{\partial}{\partial x_{j}},\frac{\partial}{\partial x_{i}}).
\end{equation}
Thus the identities  \eqref{equ 2 proof Lem Kah geo coordinate} and \eqref{equ 3 proof Lem Kah geo coordinate}  yield
\begin{equation}\nabla_{v_{i}} v_{j}=(\pi^{\star}\nabla^{FS})_{v_{i}} v_{j}+\xi[\frac{d\eta}{2}(v_{j},v_{i})].\end{equation}

Via the tangent map $\pi_{5,4,\star}D\rightarrow T\mathbb{P}^{2}$ which is an isometry, we verify that $J_{0}=\pi_{5,4}^{\star}J_{\mathbb{P}^{2}}$. This implies that $J_{0}v_{1}=v_{2},\ J_{0}v_{3}=v_{4}$, because $J_{\mathbb{P}^{2}}\frac{\partial}{\partial x_{1}}=\frac{\partial}{\partial x_{2}},\ J_{\mathbb{P}^{2}}\frac{\partial}{\partial x_{3}}=\frac{\partial}{\partial x_{4}}$. Using that $(\nabla^{FS}_{\frac{\partial}{\partial x_{i}}}\frac{\partial}{\partial x_{j}})|_{[p]}=0$, the following holds true. 
\begin{eqnarray}& &\nabla_{v_{i}} (J_{0}v_{i})=(\pi_{5,4}^{\star}\nabla^{FS})_{v_{i}}J_{0}v_{i}+\xi[\frac{d\eta}{2}(J_{0}v_{i},v_{i})] =\nabla^{FS}_{\frac{\partial}{\partial x_{i}}}(J_{\mathbb{P}^{2}}\frac{\partial}{\partial x_{i}})+\xi[\frac{d\eta}{2}(J_{0}v_{i},v_{i})]\nonumber
\\&=&-\xi\ \ \ \ \textrm{on the Reeb orbit}\ \pi^{-1}_{5,4}[p]. 
\end{eqnarray}
Similarly, we compute
\begin{equation}(\nabla_{v_{i}}v_{i})|_{q}=0+\xi[\frac{d\eta}{2}(v_{i},v_{i})]|_{q}=0.\nonumber
\end{equation}
The proof is complete. 
\end{proof}

The transverse geodesic frame helps in proving the following two formulas which are applied in the proof of the Bochner formulas (see Lemma \ref{lem Bochner} above).
\begin{formula}\label{formula 1 proof lem rough laplacian in eta and a0 component} In the setting of Lemma \ref{lem Bochner}, $(\nabla^{\star}\nabla a_{0})(\xi)=2d^{\star_{0}}_{0}J_{0}(a_{0})$. 
\end{formula}
\begin{proof}[Proof of Formula \ref{formula 1 proof lem rough laplacian in eta and a0 component}]: Using that $a_{0}$ is semi-basic i.e. $a_{0}(\xi)=0$, Leibniz-rule yields 
\begin{equation}0=(\nabla^{\star}\nabla) [a_{0}(\xi)]=(\nabla^{\star}\nabla a_{0})(\xi)-2tr(\nabla a_{0}\otimes \nabla \xi)+a_{0} (\nabla^{\star}\nabla \xi).
\end{equation}
Because $\nabla^{\star}\nabla \xi=4\xi$ (see \eqref{equ hessian of xi and eta}), we find $a_{0} (\nabla^{\star}\nabla \xi)
=0$, hence \begin{equation}(\nabla^{\star}\nabla a_{0})(\xi)=2tr(\nabla a_{0}\otimes \nabla \xi).\end{equation}

Therefore, at an arbitrary $p\in \mathbb{S}^{5}$, let $v_{i}$ be a transverse geodesic frame given by Lemma \ref{lem Kahler geodesic frame}, using $\nabla_{\xi}\xi=0$, we compute \begin{eqnarray}& &2tr(\nabla a_{0}\otimes \nabla \xi)|_{p}=2\Sigma_{i=1}^{4}(\nabla_{v_{i}} a_{0})(\nabla_{v_{i}} \xi)|_{p}=2\Sigma_{i=1}^{4}(\nabla_{v_{i}} a_{0})(J_{0}v_{i})|_{p}\nonumber
\\&=& (2\Sigma_{i=1}^{4}\nabla_{v_{i}} [a_{0}(J_{0}v_{i})]-2a_{0}[\Sigma_{i=1}^{4}\nabla_{v_{i}}(J_{0}v_{i})])|_{p}= 2\Sigma_{i=1}^{4}\nabla_{v_{i}} [a_{0}(J_{0}v_{i})]|_{p}\nonumber
\\& =&-2\Sigma_{i=1}^{4}\nabla_{v_{i}} [(J_{0}a_{0})(v_{i})]|_{p}\nonumber
\\&=&2d^{\star_{0}}_{0}J_{0}(a_{0})|_{p}.
\end{eqnarray}
The proof is complete by the above two identities.
\end{proof}

\begin{formula}\label{clm 3 proof lem rough laplacian in eta and a0 component}In the setting of Lemma \ref{lem Bochner},  $[\nabla^{\star}\nabla(\eta a_{\eta})]=\eta(4a_{\eta}+\nabla^{\star}\nabla a_{\eta})-2J_{0}(d_{0}a_{\eta}).$
\end{formula}
\begin{proof}[Proof of formula \ref{clm 3 proof lem rough laplacian in eta and a0 component}:] We still work with a transverse geodesic frame $v_{i}\ (i=1,2,3,4)$ at an arbitrary $p\in \mathbb{S}^{5}$. Using the fundamental identities \eqref{equ hessian of xi and eta}, we calculate
\begin{eqnarray}\label{equ laplace of aeta eta}
& &\nabla^{\star}\nabla (a_{\eta}\eta)=(\nabla^{\star}\nabla a_{\eta})\eta+a_{\eta}\nabla^{\star}\nabla \eta-2tr(\nabla a_{\eta}\otimes \nabla\eta)
\\&=&(\nabla^{\star}\nabla a_{\eta})\eta+4a_{\eta}\eta-2tr(\nabla a_{\eta}\otimes \nabla\eta).\nonumber
\end{eqnarray}

In view of the local formula \eqref{equ d0u Euc formula} for $d_{0}a_{\eta}$, the transverse geodesic frame yields 
\begin{equation}\label{equ 0 proof of clm 3 proof lem rough laplacian in eta and a0 component}(\nabla_{v_{i}}a_\eta)dx^{i}|_{p}=(d_{0}a_{\eta})|_{p}.
\end{equation}
Using the vanishing $\nabla_{\xi}\eta=0$ and the formula $\nabla_{X} \eta=[J_{0}(X^{\parallel_{0}})]^{\sharp}$, the trace term in the above identity can be additionally analyzed as follows. 
\begin{eqnarray*}& & tr(\nabla a_{\eta}\otimes \nabla\eta)|_{p}=[\Sigma_{i=1}^{4} \nabla_{v_{i}} a_{\eta}\otimes \nabla_{v_{i}}\eta+L_{\xi} a_{\eta}\otimes \nabla_{\xi}\eta]|_{p}=[\Sigma_{i=1}^{4} \nabla_{v_{i}} a_{\eta}\otimes \nabla_{v_{i}}\eta ]|_{p}
\\&=&J_{0}(dx^{i})(\nabla_{v_{i}}a_\eta)|_{p}
\\&=&J_{0}(d_{0}a_{\eta})|_{p} \ \ \ \textrm{by}\ \eqref{equ 0 proof of clm 3 proof lem rough laplacian in eta and a0 component}.
\end{eqnarray*}
 The desired identity follows.  
\end{proof}

Elementary calculations establish the formula for the contact form $\eta$, and the formula for  $G$ and $H$. 
 \begin{proof}[\textbf{Proof of Formula} \ref{formula eta}:]It suffices to check it in $U_{0,\mathbb{C}^{3}}$, the proof is similar in $U_{1,\mathbb{C}^{3}}$ and $U_{2,\mathbb{C}^{3}}$. We first have $Z_{0}\frac{\partial}{\partial Z_{1}}=\frac{Z_{0}\bar{Z}_{1}}{2r}\frac{\partial}{\partial r}+\frac{\partial}{\partial u_{1}}$ and $Z_{0}\frac{\partial}{\partial Z_{2}}=\frac{Z_{0}\bar{Z}_{2}}{2r}\frac{\partial}{\partial r}+\frac{\partial}{\partial u_{2}}$. Using 
 $$\eta=\frac{1}{2}d^{c}\log (|Z_{0}|^{2}+|Z_{1}|^{2}+|Z_{2}|^{2})\ \ (\textrm{see definition}\ \eqref{equ def eta}),$$
 we find
 \begin{equation}\eta (\frac{\partial}{\partial u_{1}})=\eta (Z_{0}\frac{\partial}{\partial Z_{1}})=-\frac{\sqrt{-1}\bar{u}_{1}}{2\phi_{0}}=-\frac{\sqrt{-1}}{2}\frac{\partial \log \phi_{0}}{\partial u_{1}}. 
 \end{equation}
 Similarly, we have $\eta (\frac{\partial}{\partial u_{2}})=-\frac{\sqrt{-1}}{2}\frac{\partial \log \phi_{0}}{\partial u_{2}}$. Taking conjugation, we then obtain 
$$\eta (\frac{\partial}{\partial \bar{u}_{1}})=\frac{\sqrt{-1}}{2}\frac{\partial \log \phi_{0}}{\partial \bar{u}_{1}},\ \ \eta (\frac{\partial}{\partial \bar{u}_{2}})=\frac{\sqrt{-1}}{2}\frac{\partial \log \phi_{0}}{\partial \bar{u}_{2}}.$$
The proof is complete by observing that $\eta$ coincides with $d\theta_{0}+\frac{d^{c}\log\phi_{0}}{2}$ on the basis $$\frac{\partial}{\partial \theta_{0}},\ \frac{\partial}{\partial u_{1}},\ \frac{\partial}{\partial u_{2}},\ \frac{\partial}{\partial \bar{u}_{1}},\ \frac{\partial}{\partial \bar{u}_{2}}\ \
\ \textrm{for} \ T^{\mathbb{C}}\mathbb{S}^{5}.$$ 
 \end{proof}

\begin{proof}[\textbf{Proof of Lemma} \ref{lem G and H}:] We routinely verify in $U_{0,\mathbb{C}^{3}}$ that
\begin{eqnarray*}& & \frac{dZ_{0}}{Z_{0}}=\frac{1}{Z_{0}}d(\frac{re^{\sqrt{-1}\theta_{0}}}{\sqrt{\phi_{0}}})=\frac{dr}{r}-\frac{d\log \phi_{0}}{2}+\sqrt{-1}d\theta_{0}
\\&=& \frac{dr}{r}-\frac{d\log \phi_{0}}{2}+\sqrt{-1}\eta-\frac{\sqrt{-1}}{2}(d^{c}\log \phi_{0})\ \textrm{by Formula}\ \ref{formula eta}.
\end{eqnarray*}
When $i=1,2$, we calculate
\begin{equation*}\frac{dZ_{i}}{Z_{0}}=\frac{d(Z_{0}u_{i})}{Z_{0}}=du_{i}+u_{i}(\frac{dZ_{0}}{Z_{0}}).
\end{equation*}
Then,
\begin{eqnarray*}& &\Omega_{\mathbb{C}^{3}}=dZ_{0}dZ_{1}dZ_{2}=Z^{3}_{0}\cdot \frac{dZ_{0}}{Z_{0}}\frac{dZ_{1}}{Z_{0}}\frac{dZ_{2}}{Z_{0}}=Z^{3}_{0}\cdot \frac{dZ_{0}}{Z_{0}}\wedge [du_{1}+u_{1}(\frac{dZ_{0}}{Z_{0}})]\wedge [du_{2}+u_{2}(\frac{dZ_{0}}{Z_{0}})].\nonumber
\\&=& Z^{3}_{0}\cdot \frac{dZ_{0}}{Z_{0}}\wedge du_{1}\wedge du_{2}
\\&=& Z^{3}_{0}\cdot \frac{dr}{r}\wedge du_{1}\wedge du_{2}+\sqrt{-1}Z^{3}_{0}\cdot \eta \wedge du_{1}\wedge du_{2}.
\end{eqnarray*}
The last inequality above uses that $d\log \phi_{0},\ d^{c}\log \phi_{0}$ are both pulled back from $U_{0,\mathbb{P}^{2}}\subset \mathbb{P}^{2}$, therefore $d\log \phi_{0}\wedge du_{1}\wedge du_{2}=d^{c}\log \phi_{0}\wedge du_{1}\wedge du_{2}=0$.

 In $U_{0,\mathbb{C}^{3}}$, the proof of \eqref{equ formula for Omega full} and the first row in \eqref{equ G and H} is complete. The proof is similar in $U_{1,\mathbb{C}^{3}}$ and $U_{2,\mathbb{C}^{3}}$.
\end{proof}

\section{The Sasaki-Quaternion coordinate on $\mathbb{S}^{5}$\label{Appendix SQ coordinate}}
The purpose of this section is to introduce a simple coordinate system under which Lemma \ref{lem formula of the model dirac deformation operator} can be proved.

Because both the $(1,0)$ vector field $\frac{1}{2r^{3}}(r\frac{\partial}{\partial r}-\sqrt{-1}\xi)$ and the standard $(3,0)-$form $dZ_{0}dZ_{1}dZ_{2}$ are invariant under the $SU(3)-$action on $\mathbb{C}^{3}$, by  definition \eqref{equ contraction def of H and G},  so is $H$ and $G$. 

\begin{fact} For any $\chi\in SU(3)$ acting on $\mathbb{S}^5$, $\chi^{\star}G=G$, $\chi^{\star}H=H$.
\end{fact}



\begin{Def} For any point $[Z]\in \mathbb{P}^{2}$, let $\chi\in SU(3)$ be an element mapping $[Z]$ to $[1,0,0]$, let $z_{1}=\chi^{\star}u_{1},\ z_{2}=\chi^{\star}u_{2}$ be the pullback  coordinate system near $[Z]$. Because the Fubini-Study metric is $SU(3)-$invariant,  $(z_{1},\ z_{2})$ is a K\"ahler geodesic coordinate at  $[Z]$ i.e. 
for any $i,j$ among $1,\ 2$, $\nabla^{FS}_{\frac{\partial }{\partial z_{i}}}\frac{\partial }{\partial z_{j}}=\nabla^{FS}_{\frac{\partial }{\partial z_{i}}}\frac{\partial }{\partial \bar{z}_{j}}=0$. 

On $\mathbb{S}^{5}$, we call $(z_{1},z_{2})$ a \textit{Sasaki-Quaternion coordinate} of the Reeb orbit $\pi^{-1}_{5,4}[Z]$.
 \end{Def} 
Under such a coordinate, on the Reeb orbit, both $G$ and $H$ take the canonical form 
\begin{equation}\label{equ G and H at any point under Sasakian canonical coordinate}
G=-Im(dz_{1}dz_{2}),\ H=Re(dz_{1}dz_{2}).
\end{equation}
 Moreover, it yields a transverse geodesic frame in the sense of Lemma \ref{lem Kahler geodesic frame}.

\section{The usual separation of variable: proof of Formula  \ref{formula pre splitting}}

\begin{proof}[\textbf{Proof of Formula \ref{formula pre splitting}}] It is completely routine. To be self-contained, we still show the detail.  In view of the splitting \eqref{equ 1st splitting of the 1form}, we find 
\begin{equation}\label{equ 0 formula pre splitting}d_{\mathbb{C}^{3}\times \mathbb{S}^{1}}a_{\mathbb{C}^{3}\times \mathbb{S}^{1}}=(d_{\mathbb{C}^{3}}\underline{a}_{s})\wedge ds+d_{\mathbb{C}^{3}}a_{\mathbb{C}^{3}}-\frac{\partial a_{\mathbb{C}^{3}}}{\partial s}\wedge ds
\end{equation}
and \begin{equation}d_{\mathbb{C}^{3}\times \mathbb{S}^{1}}^{\star_{\mathbb{C}^{3}\times \mathbb{S}^{1}}}a_{\mathbb{C}^{3}\times \mathbb{S}^{1}}=-\frac{\partial \underline{a}_{s}}{\partial s}+d_{\mathbb{C}^{3}}^{\star_{\mathbb{C}^{3}}}a_{\mathbb{C}^{3}}. 
\end{equation}
Using the tensor identity \begin{equation}\star_{\mathbb{C}^{3}\times \mathbb{S}^{1}}[d_{\mathbb{C}^{3}\times \mathbb{S}^{1}}a_{\mathbb{C}^{3}\times \mathbb{S}^{1}} \wedge \psi_{\mathbb{C}^{3}\times \mathbb{S}^{1}}]= (d_{\mathbb{C}^{3}\times \mathbb{S}^{1}}a_{\mathbb{C}^{3}\times \mathbb{S}^{1}})\lrcorner_{\mathbb{C}^{3}\times \mathbb{S}^{1}}  \phi_{\mathbb{C}^{3}\times \mathbb{S}^{1}},\end{equation}
it suffices to calculate the right  side of \eqref{equ 0 formula pre splitting} term-wisely as follows. 
\begin{equation}\label{equ 1 formula pre splitting}(d_{\mathbb{C}^{3}\times \mathbb{S}^{1}}a_{\mathbb{C}^{3}\times \mathbb{S}^{1}}) \lrcorner_{\mathbb{C}^{3}\times \mathbb{S}^{1}} (\omega_{\mathbb{C}^{3}}\wedge ds)=-(d_{\mathbb{C}^{3}}\underline{a}_{s})\lrcorner_{\mathbb{C}^{3}} \omega_{\mathbb{C}^{3}}+(d_{\mathbb{C}^{3}}a_{\mathbb{C}^{3}}\lrcorner_{\mathbb{C}^{3}} \omega_{\mathbb{C}^{3}})ds+\frac{\partial a_{\mathbb{C}^{3}}}{\partial s}\lrcorner_{\mathbb{C}^{3}} \omega_{\mathbb{C}^{3}}.
\end{equation} 
\begin{equation}\label{equ 2 formula pre splitting}
(d_{\mathbb{C}^{3}\times \mathbb{S}^{1}}a_{\mathbb{C}^{3}\times \mathbb{S}^{1}})\lrcorner_{\mathbb{C}^{3}\times \mathbb{S}^{1}} Re\Omega=(d_{\mathbb{C}^{3}}a_{\mathbb{C}^{3}})\lrcorner_{\mathbb{C}^{3}} Re\Omega. 
\end{equation}
The contraction $\lrcorner_{\mathbb{C}^{3}} \omega_{\mathbb{C}^{3}}$ is the complex- structure $J_{\mathbb{C}^{3}} $ on $\Omega^{1}[ad(E)]$.
Summing   \eqref{equ 1 formula pre splitting} and \eqref{equ 2 formula pre splitting} up, we arrive at the following. 

\begin{eqnarray*}& &(d_{\mathbb{C}^{3}\times \mathbb{S}^{1}}a_{\mathbb{C}^{3}\times \mathbb{S}^{1}}) \lrcorner_{\mathbb{C}^{3}\times \mathbb{S}^{1}}\phi_{\mathbb{C}^{3}\times \mathbb{S}^{1}}
\\&=&-J_{\mathbb{C}^{3}} (d_{\mathbb{C}^{3}}\underline{a}_{s})+(d_{\mathbb{C}^{3}}a_{\mathbb{C}^{3}}\lrcorner_{\mathbb{C}^{3}} \omega_{\mathbb{C}^{3}})ds+J_{\mathbb{C}^{3}} (\frac{\partial a_{\mathbb{C}^{3}}}{\partial s})+(d_{\mathbb{C}^{3}}a_{\mathbb{C}^{3}})\lrcorner_{\mathbb{C}^{3}} Re\Omega.
\end{eqnarray*} 

Using the above,  the easy identity $d_{\mathbb{C}^{3}\times \mathbb{S}^{1}}\sigma=(\frac{\partial \sigma}{\partial s})ds+d_{\mathbb{C}^{3}}\sigma$, and definition \eqref{equ  formula for model deformation operator} of $L_{A_{O},\phi_{\mathbb{C}^{3}\times \mathbb{S}^{1}}}$, we obtain the following. 
\begin{eqnarray}& &L_{A_{O},\phi_{\mathbb{C}^{3}\times \mathbb{S}^{1}}}\left[\begin{array}{c}\sigma\\ \underline{a}_{s} ds +  a_{\mathbb{C}^{3}}\end{array}\right]=\left[\begin{array}{c}-\frac{\partial \underline{a}_{s}}{\partial s}+d^{\star_{\mathbb{C}^{3}} }_{\mathbb{C}^{3}}a_{\mathbb{C}^{3}}\\ \{\frac{\partial \sigma}{\partial s}+(d_{\mathbb{C}^{3}}a_{\mathbb{C}^{3}})\lrcorner_{\mathbb{C}^{3}} \omega_{\mathbb{C}^{3}}\} ds +  \\ J_{\mathbb{C}^{3}} (\frac{\partial a_{\mathbb{C}^{3}}}{\partial s})+d_{\mathbb{C}^{3}}\sigma-J_{\mathbb{C}^{3}} (d_{\mathbb{C}^{3}}\underline{a}_{s})+(d_{\mathbb{C}^{3}}a_{\mathbb{C}^{3}})\lrcorner_{\mathbb{C}^{3}} Re\Omega\end{array}\right].\nonumber
\end{eqnarray}
The desired formula follows. 
\end{proof}

\section{The fine separation of variable: proof of Lemma \ref{lem formula of the model dirac deformation operator}\label{Appendix proof of formula of P}}
We prove Lemma \ref{lem formula of the model dirac deformation operator} by computing each row in the operator $\square$ (see Formula \ref{formula pre splitting}). We first recall the following splitting.
\begin{equation}\label{equ fine splitting for ac3}
a_{\mathbb{C}^{3}}=a_{0}+(a_{\eta}\eta)+{a}_{r}\frac{dr}{r}.
\end{equation}

Given a section $a$ of $\wedge^{p}T^{\star}\mathbb{S}^{5}$ and a section $b$ of $\wedge^{q}T^{\star}\mathbb{S}^{5}$ such that $5\geq q\geq p$, we need  the following identity.  $$a\lrcorner_{\mathbb{C}^{3}}b=\frac{1}{r^{2p}}a\lrcorner_{\mathbb{S}^{5}}b.$$ 

Employing the splitting
\begin{equation}\label{equ splitting of dc3}
d_{\mathbb{C}^{3}}=d_{0}+\eta\wedge L_{\xi}+dr\wedge L_{\frac{\partial}{\partial r}},\end{equation}

 we calculate
\begin{eqnarray}\label{equ dc3a}d_{\mathbb{C}^{3}}a_{\mathbb{C}^{3}}&=&(d_{0}+\eta\wedge L_{\xi}+dr\wedge L_{\frac{\partial}{\partial r}})[a_{0}+(a_{\eta}\eta)+{a}_{r}\frac{dr}{r}]\nonumber
\\&=& d_{0}a_{0}+(2a_{\eta})\frac{d\eta}{2}+\eta\wedge (L_{\xi}a_{0}-d_{0}a_{\eta})+\frac{dr}{r}\wedge(r\frac{\partial a_{0}}{\partial r}-d_{0}a_{r})\nonumber
\\& &+ (\frac{dr}{r}\wedge \eta) (r\frac{\partial a_{\eta}}{\partial r}-L_{\xi}a_{r})\label{equ 0 formula da contract Euc metric}.
\end{eqnarray}

Via the splitting \eqref{equ splitting of dc3}, using $\sigma=\frac{u}{r},\ \underline{a}_{s}=\frac{a_{s}}{r}$, we routinely verify the following two identities. 
\begin{equation}\label{equ dsigma das}d_{\mathbb{C}^{3}}\sigma=\frac{d_{0}u}{r}+\frac{\eta\wedge L_{\xi}u}{r}+(\frac{\partial u}{\partial r}-\frac{u}{r})\frac{dr}{r};\ d_{\mathbb{C}^{3}}\underline{a}_{s}=\frac{d_{0}a_{s}}{r}+\frac{\eta\wedge L_{\xi}a_{s}}{r}+(\frac{\partial a_{s}}{\partial r}-\frac{a_{s}}{r})\frac{dr}{r}.
\end{equation}
Employing the table
 \begin{equation}\label{equ tabular volume forms}  \begin{tabular}{|p{6cm}|}
  \hline
 $\omega_{\mathbb{C}^{3}}=rdr\wedge \eta+\frac{r^{2}d\eta}{2}$ \\   \hline
  $dVol_{\mathbb{P}^{2}}=\frac{1}{2}(\frac{d\eta}{2})^{2}$ \\   \hline
  $dVol_{\mathbb{S}^{5}}=\eta\wedge dVol_{\mathbb{P}^{2}}$ \\   \hline
   $dVol_{\mathbb{C}^{3}}=r^{5}dr\wedge\eta\wedge dVol_{\mathbb{P}^{2}}$ \\   \hline
\end{tabular}
 \renewcommand\arraystretch{1.5}
  \end{equation}
 via \eqref{equ dsigma das}, we calculate the contraction
\begin{eqnarray}\nonumber & &(d_{\mathbb{C}^{3}}\underline{a}_{s})\lrcorner_{\mathbb{C}^{3}} \omega_{\mathbb{C}^{3}}=[\frac{d_{0}a_{s}}{r}+\frac{\eta\wedge L_{\xi}a_{s}}{r}+(\frac{\partial a_{s}}{\partial r}-\frac{a_{s}}{r})\frac{dr}{r}]\lrcorner_{6}  (rdr\wedge\eta+r^{2}\frac{d\eta}{2})
\\&=& \frac{1}{r}d_{0}a_{s}\lrcorner \frac{d\eta}{2}- \frac{L_{\xi}a_{s}}{r}\frac{dr}{r}+(\frac{\partial a_{s}}{\partial r}-\frac{a_{s}}{r})\eta. \label{equ das contract omega c3}
\end{eqnarray}

Employing the commutator identities \eqref{equ lie commutator with hyperkahler str} and formula \eqref{equ dc3a} for $d_{\mathbb{C}^{3}}a_{\mathbb{C}^{3}}$, we calculate
 \begin{eqnarray}& &d_{\mathbb{C}^{3}}a_{\mathbb{C}^{3}}\lrcorner_{\mathbb{C}^{3}} Re\Omega_{\mathbb{C}^{3}}\label{equ dac3 contract ReOmegac3}
\\&=& \{d_{0}a_{0}+(2a_{\eta})\frac{d\eta}{2}+\eta\wedge (L_{\xi}a_{0}-d_{0}a_{\eta})+\frac{dr}{r}\wedge(r\frac{\partial a_{0}}{\partial r}-d_{0}a_{r})\nonumber
\\& &+ (\frac{dr}{r}\wedge \eta)(r\frac{\partial a_{\eta}}{\partial r}-L_{\xi}a_{r})\}\lrcorner_{\mathbb{C}^{3}} (r^{2}dr\wedge H+r^{3}\eta\wedge G)\nonumber
\\&=& \frac{dr}{r}\cdot \frac{d_{0}a_{0}\lrcorner H}{r}+\eta\cdot \frac{d_{0}a_{0}\lrcorner G}{r}+\frac{1}{r}[(L_{\xi}a_{0})\lrcorner G-(d_{0}a_{\eta})\lrcorner G]+(\frac{\partial}{\partial r}J_{H}a_{0})-\frac{J_{H}(d_{0}a_{r})}{r}\nonumber
\\&=& \frac{dr}{r}\cdot \frac{d_{0}a_{0}\lrcorner H}{r}+\eta\cdot \frac{d_{0}a_{0}\lrcorner G}{r}+\frac{1}{r}[L_{\xi}(J_{G}a_{0})+3J_{H}(a_{0})-J_{G}(d_{0}a_{\eta})]+(\frac{\partial}{\partial r}J_{H}a_{0})-\frac{J_{H}(d_{0}a_{r})}{r}.\nonumber
\end{eqnarray}

Assembling identity \eqref{equ dsigma das}, \eqref{equ das contract omega c3}, and \eqref{equ dac3 contract ReOmegac3}, we can characterize row $3$ of the operator $\square$. 
\begin{formula}\label{equ the 3rd row in square}In view of Formula \ref{formula pre splitting}, on the third row of the operator $\square$, we have 
 \begin{eqnarray}& &d_{\mathbb{C}^{3}}\sigma- (d_{\mathbb{C}^{3}}\underline{a}_{s})\lrcorner \omega_{\mathbb{C}^{3}}+d_{\mathbb{C}^{3}}a_{\mathbb{C}^{3}}\lrcorner_{\mathbb{C}^{3}} Re\Omega_{\mathbb{C}^{3}} \nonumber
\\&=& \frac{dr}{r}\cdot [\frac{\partial u}{\partial r}-\frac{u}{r}+\frac{L_{\xi}a_{s}}{r}+\frac{d_{0}a_{0}\lrcorner H}{r}]
+\eta\cdot [-\frac{\partial a_{s}}{\partial r}+\frac{a_{s}}{r}+\frac{L_{\xi}u}{r}+\frac{d_{0}a_{0}\lrcorner G}{r}] \nonumber
\\& &+\frac{1}{r}[d_{0}u-J_{0}(d_{0}a_{s})+L_{\xi}(J_{G}a_{0})+3J_{H}(a_{0})-J_{G}(d_{0}a_{\eta})+r\frac{\partial}{\partial r}(J_{H}a_{0})-J_{H}(d_{0}a_{r})].\nonumber
\end{eqnarray}
\end{formula}

Next, we calculate the first and second row of $\square$. 
 \begin{formula}\label{formula decomposition for dstar} The following two identities hold. 
 \begin{equation}\label{equ -2 formula decomposition for dstar}d_{\mathbb{C}^{3}}^{\star_{\mathbb{C}^{3}}}a_{\mathbb{C}^{3}}=-\frac{1}{r}\frac{\partial a_{r}}{\partial r}-\frac{4a_{r}}{r^{2}}-\frac{L_{\xi}a_{\eta}}{r^{2}}+\frac{d_{0}^{\star_{0}}a_{0}}{r^{2}}.\end{equation}
\begin{equation}\label{equ -1 formula decomposition for dstar}d_{\mathbb{C}^{3}}a_{\mathbb{C}^{3}}\lrcorner_{\mathbb{C}^{3}} \omega_{\mathbb{C}^{3}}= \frac{1}{r^{2}}d_{0}a_{0}\lrcorner \frac{d\eta}{2}-\frac{L_{\xi}a_{r}}{r^{2}}+\frac{1}{r}\frac{\partial a_{\eta}}{\partial r}+\frac{4a_{\eta}}{r^{2}}.\end{equation}
 
  In particular, under the splitting $a_{\mathbb{S}^{5}}=a_{\eta}\eta +a_{0}$, 
 \begin{equation}\label{formula decomposition for dstar S5}d_{\mathbb{S}^{5}}^{\star_{\mathbb{S}^{5}}}a_{\mathbb{S}^{5}}=-L_{\xi}a_{\eta}+d_{0}^{\star_{0}}a_{0}.\end{equation}
 \end{formula}
 \begin{proof}[Proof of Formula  \ref{formula decomposition for dstar}:]  The volume forms in  table \eqref{equ tabular volume forms} imply the following two identities.
 \begin{equation}\label{equ -3 formula decomposition for dstar}\star_{\mathbb{C}^{3}}\eta=-r^{3}dr\wedge dVol_{\mathbb{P}^{2}}.\end{equation} 
   \begin{equation}\label{equ -2.5 formula decomposition for dstar}\star_{\mathbb{C}^{3}}a_{0}=r^{3}dr\wedge \eta \wedge \star_{0}a_{0}.\end{equation}
   
  Since $d\eta$ is a section of $\wedge^{(1,1)}\otimes D^{\star}$, but $G$ is a section of $[\wedge^{(2,0)}\oplus \wedge^{(0,2)}]\otimes D^{\star}-$valued, we find the following vanishing \begin{equation}\label{equ -1.5 formula decomposition for dstar}(d\eta)\lrcorner (G\wedge \eta)=0.\end{equation}

Using the above $3$ elementary identities, we calculate $d_{\mathbb{C}^{3}}^{\star_{\mathbb{C}^{3}}}a_{\mathbb{C}^{3}}$ according to the $3-$terms in the fine splitting \eqref{equ fine splitting for ac3}. 
 \begin{eqnarray}\label{equ 0 formula decomposition for dstar}\nonumber& &d_{\mathbb{C}^{3}}^{\star_{\mathbb{C}^{3}}}(a_{r}\frac{dr}{r})=-\star_{\mathbb{C}^{3}} d_{\mathbb{C}^{3}} \star_{\mathbb{C}^{3}} (a_{r}\frac{dr}{r})=-\star_{\mathbb{C}^{3}} d_{\mathbb{C}^{3}} (r^{4} a_{r} dVol_{\mathbb{S}^{5}})\\&=&-\star_{\mathbb{C}^{3}}[\frac{1}{r^{5}}\frac{\partial(r^{4} a_{r})}{\partial r} r^{5}dr\wedge dVol_{\mathbb{S}^{5}}]\nonumber
=-\frac{1}{r^{5}}\frac{\partial(r^{4} a_{r})}{\partial r} \nonumber
\\&=&-\frac{1}{r}\frac{\partial a_{r}}{\partial r}-\frac{4a_{r}}{r^{2}}.
\end{eqnarray}
 \begin{eqnarray}\label{equ 1 formula decomposition for dstar}& &d_{\mathbb{C}^{3}}^{\star_{\mathbb{C}^{3}}}(a_{\eta}\eta)=-\star_{\mathbb{C}^{3}} d_{\mathbb{C}^{3}} \star_{\mathbb{C}^{3}} (a_{\eta}\eta)=\star_{\mathbb{C}^{3}} d_{\mathbb{C}^{3}} (a_{\eta}r^{3}dr\wedge dVol_{\mathbb{P}^{2}})\nonumber
 \\&=&L_{\xi}a_{\eta}\star_{\mathbb{C}^{3}} (r^{3}\eta\wedge dr\wedge dVol_{\mathbb{P}^{2}})\nonumber
\\&=&-\frac{L_{\xi}a_{\eta}}{r^{2}}.
\end{eqnarray}
\begin{eqnarray}\label{equ 2 formula decomposition for dstar}& &d_{\mathbb{C}^{3}}^{\star_{\mathbb{C}^{3}}}a_{0}=-\star_{\mathbb{C}^{3}} d_{\mathbb{C}^{3}} \star_{\mathbb{C}^{3}} a_{0}=-\star_{\mathbb{C}^{3}} d_{\mathbb{C}^{3}} (r^{3}dr\wedge \eta \wedge\star_{0} a_{0})=-\star_{\mathbb{C}^{3}}  (r^{3}dr\wedge \eta \wedge d_{0}\star_{0} a_{0})\nonumber
\\&= & -\frac{1}{r^{2}}\star_{\mathbb{C}^{3}}  (r^{5}dr\wedge \eta \wedge d_{0}\star_{0} a_{0})=-\frac{1}{r^{2}}  \star_{0}d_{0}\star_{0} a_{0}\nonumber
\\&=& \frac{d_{0}^{\star_{0}}a_{0}}{r^{2}}.
\end{eqnarray}
Identity \eqref{equ -2 formula decomposition for dstar} follows simply by summing up \eqref{equ 0 formula decomposition for dstar}, \eqref{equ 1 formula decomposition for dstar}, and \eqref{equ 2 formula decomposition for dstar}.

Because $d_{\mathbb{C}^{3}}a_{\mathbb{C}^{3}}\lrcorner_{\mathbb{C}^{3}} \omega_{\mathbb{C}^{3}}=d^{\star_{\mathbb{C}^{3}}}_{\mathbb{C}^{3}}J_{\mathbb{C}^{3}}a_{\mathbb{C}^{3}}$, using $$J_{\mathbb{C}^{3}}(a_{r}\frac{dr}{r}+a_{\eta}\eta+a_{0})=a_{r}\eta-a_{\eta}\frac{dr}{r}+J_{0}a_{0},$$ identity \eqref{equ -1 formula decomposition for dstar} follows from \eqref{equ -2 formula decomposition for dstar} replacing $a_{\eta}$ by $a_{r}$, $a_{r}$ by $-a_{\eta}$, and $a_{0}$ by $J_{0}a_{0}$ therein.  
\end{proof}

The formulas established so far can be assembled into the desired formula of $P$. 
\begin{proof}[\textbf{Proof of Lemma \ref{lem formula of the model dirac deformation operator}}:]Still in view of Formula \ref{formula pre splitting}, it is natural to classify the terms in the fine splitting of $\square$ into $3$ kinds of terms: those only involving $\frac{\partial}{\partial r}$ (derivative in $r$), those only involving $L_{\xi}$ (derivative along the Reeb vector field), and those only involving $d_{0}$.

We carry out the above scheme.  Using
\begin{itemize}\item the formula for the isometries $K$ and $T$ in Lemma \ref{lem hk on 7-dim model}, 
\item Formula \ref{formula decomposition for dstar}  for the first and second row of $\square$, 
\item Formula \ref{equ the 3rd row in square} for the third row of $\square$, 
\end{itemize}
we find the following fine splitting for $\square$: 
\begin{equation}
\square=\frac{\partial}{\partial r}K+\frac{L_{\xi}T}{r}+\frac{\underline{B}_{0}}{r}, 
\end{equation}
where 
\begin{equation}
\underline{B}_{0}\left[\begin{array}{c}u \\ a_{s} \\  a_{r} \\ a_{\eta} \\ a_{0}\end{array}\right]=\left[\begin{array}{ccccc}0  &0 & -4 & 0& d^{\star_{0}}_{0}\\ 0 & 0 & 0 & 4& (d_{0}\cdot)\lrcorner \frac{d\eta}{2} \\  -1& 0 & 0 & 0& (d_{0}\cdot)\lrcorner {H}  \\  0 & 1 &0 & 0& (d_{0}\cdot)\lrcorner {G}  \\  d_{0} & -J_{0}d_{0} & -J_{H}d_{0} & -J_{G}d_{0}& 3J_{H}\end{array}\right].
\end{equation}

It is then routine to verify, with the help of the second commutator identity in \eqref{equ lie commutator with hyperkahler str},  that $P\triangleq K(\underline{B}_{0}+L_{\xi}T)$ is equal to the one given by \eqref{equ formula for P}. Hence
\begin{equation}\label{equ square operator}
\square=K(\frac{\partial}{\partial r}-\frac{P}{r}). 
\end{equation}
The proof is complete. 
\end{proof}
\section{Digression to Hermitian Yang-Mills connections\label{Appendix digression}}
On a smooth Hermitian vector bundle $E$ over a Calabi-Yau $3-$fold $(X,\omega,\Omega)$, a triple $(A,\sigma,u)$ consisted of a smooth connection $A$ and two smooth sections $\sigma$ and $u$ of $adE$ is called a Hermitian Yang-Mills monopole  if it satisfies the following equations   \begin{equation}F_{A}\lrcorner Re\Omega+d_{A}\sigma-J(d_{A} u)=0,\ F_{A}\lrcorner \omega=0.
\end{equation}
Suppose the Calabi-Yau is compact, the closeness of the holomorphic $(3,0)-$form $\Omega$ and the K\"ahler form $\omega$ actually implies $$F_{A}\lrcorner Re\Omega=0,\ d_{A}\sigma=d_{A}u =0,\ \ i.e.\ \ F_{A}\ \ \textrm{is}\ \ \ (1,1),$$  and $A$ is a Hermitian Yang-Mills connection. 

Given a holomorphic Hermitian triple on $\mathbb{P}^{2}$. With gauge fixing, the linearized operator with respect to the associated data setting on $\mathbb{C}^{3}\setminus O$ is precisely the operator $\square$ which is part of $L_{A_{O},\phi_{\mathbb{C}^{3}\times \mathbb{S}^{1}}}$ (see Formula \ref{formula pre splitting}). It of course depends on $P$ (see \eqref{equ square operator}).
\section{Fourier series re-visited}
We first recall an elementary fact on uniform convergence of the usual Fourier series. 
\begin{lem}\label{lem uniform convergence of F series} There exists a positive function $[\epsilon(N),\ N\in \mathbb{Z}^{+}]$ such that $\lim_{N\rightarrow \infty}\epsilon(N)=0$ and the following holds. Let $f\in W^{1,2}(\mathbb{S}^{1})$ and let its Fourier series be $\Sigma_{k}f_{k}e^{\sqrt{-1}k\theta}$, then 
\begin{equation}\Sigma_{|k|\geq N}|f_{k}e^{\sqrt{-1}k\theta}|\leq \epsilon(N)+\frac{1}{\sqrt{N}}|f|^{2}_{W^{1,2}(\mathbb{S}^{1})}.\label{equ 0 lem uniform convergence of F series}
\end{equation}
\end{lem}

\begin{proof}[Proof of Lemma \ref{lem uniform convergence of F series}:] We estimate simply by Cauchy-Schwartz inequality that \begin{equation}\label{equ 1 lem uniform convergence of F series} |f_{k}e^{\sqrt{-1}k\theta}|\leq \frac{1}{k^{\frac{3}{2}}}+\frac{k^{\frac{3}{2}}f^{2}_{k}}{2}.\end{equation} Then,  $\xi(N)\triangleq \Sigma_{N\geq 1}\frac{1}{k^{\frac{3}{2}}}$ satisfies the desired conditions. Moreover, on the other term in \eqref{equ 1 lem uniform convergence of F series}, we estimate $$\Sigma_{N\geq 1}k^{\frac{3}{2}}f^{2}_{k}\leq \frac{1}{\sqrt{N}}\Sigma_{N\geq 1}k^{2}f^{2}_{k}\leq \frac{1}{\sqrt{N}}|f|^{2}_{W^{1,2}(\mathbb{S}^{1})}.$$ The desired estimate \eqref{equ 0 lem uniform convergence of F series} follows. 
\end{proof}
Under the assumption $f\in W^{1,2}(\mathbb{S}^{1})$,  it is well known that its Fourier-series converges uniformly to $f$. Based on the above bound on the remainder, we provide an ingredient for Lemma \ref{lem global F series}. 
\begin{lem}\label{lem term by term differentiation of F series}In the setting of Lemma \ref{lem global F series}, let $\nu\in C^{1}(\mathbb{S}^{5},\pi^{\star}_{5,4}EndE)$. Under the pullback Hermitian metric on $\pi^{\star}_{5,4}End E$, for any $\beta=0,1,$ or $2$, the Sasaki-Fourier Series $\Sigma_{k}v_{\beta}(k)e^{\sqrt{-1}k\theta_{\beta}}$ converges uniformly to $\nu$ on $U_{\beta,\mathbb{S}^{5}}$. The equivalent global series $\Sigma_{k}\nu_{k}\otimes s_{-k}$ converges uniformly to $\nu$ on $\mathbb{S}^{5}$. 
\end{lem}
\begin{proof}[Proof of Lemma \ref{lem term by term differentiation of F series}:] Under the pullback connection from $EndE\rightarrow \mathbb{P}^{2}$, $L_{\xi}=\nabla_{\xi}$ on the sections of $\pi^{\star}_{5,4}EndE$. Thus the $C^{1}-$condition implies that $L_{\xi}\nu\in C^{0}(\mathbb{S}^{5},EndE)$. Because $\xi=\frac{\partial}{\partial \theta_{\beta}}$ in $U_{\beta,\mathbb{S}^{5}}$, fixing $u_{1},\ u_{2}$ in the Sasakian coordinate, under a unitary trivialization,\\ $\nu\in W^{1,2}(\theta_{\beta})$ (as a function of $\theta_{\beta}\in \mathbb{S}^{1}$). The estimate in Lemma \ref{lem uniform convergence of F series} says that $$\Sigma_{|k|\geq N}|\nu_{\beta}(k)e^{\sqrt{-1}k\theta_{\beta}}|_{\pi^{\star}_{5,4}EndE}\leq \epsilon(N)+\frac{|\nu|^{2}_{W^{1,2}(\theta_{\beta})}}{\sqrt{N}}\leq \epsilon(N)+\frac{[|\nu|^{2}_{C^{0}(\theta_{\beta})}+|\nabla_{\xi}\nu|^{2}_{C^{0}(\theta_{\beta})}]}{\sqrt{N}}.$$
This means the ``remainder" $\Sigma_{|k|\geq N}|\nu_{\beta}(k)e^{\sqrt{-1}k\theta_{\beta}}|_{\pi^{\star}_{5,4}EndE}$ is bounded uniformly in $\beta$ and $p\in U_{\beta,\mathbb{S}^{5}}$. The desired uniform convergence is proved. 
\end{proof}

\begin{rmk}\label{rmk localization of F series}Let $(\nu)_{-k}$ denote the $-k-$th term $\nu_{k}\otimes s_{-k}$ in the Fourier-series. The value of $(\nu)_{-k}$ on an arbitrary Reeb orbit only depends on the value of $\nu$ on the same Reeb orbit. 
\end{rmk}

In the setting of Lemma \ref{lem uniform convergence of F series}, let  $c(\theta)$ be a smooth function, the operator $c(\theta)\frac{\partial}{\partial \theta}$ in general can not differentiate the Sasaki-Fourierseries terms by term i.e.  in general $$[c(\theta)\frac{\partial f}{\partial \theta}]_{k}\neq [c(\theta)\frac{\partial}{\partial \theta}] f_{k},$$ where the subscript $\cdot_{k}$ means the $k-$th Fourier-coefficient.  The next result shows that this is not the case for the two operators we are interested in. 
\begin{clm}\label{clm derivative of Fourier term is = Fourier term of derivative}Still in the setting of Lemma \ref{lem global F series}, for any $\nu\in C^{3}(\mathbb{S}^{5},\pi^{\star}_{5,4}EndE)$, in view of the notation $(,)_{-k}$ in Remark \ref{rmk localization of F series} for the Sasaki-Fouriercoefficients,  $$(\nabla^{\star}\nabla \nu)_{k}=\nabla^{\star}\nabla (\nu)_{k},\ \textrm{and}\  \ (L_{\xi}\nu)_{k}=L_{\xi}(\nu)_{k}.$$
\end{clm}
\begin{proof}[Proof of Claim \ref{clm derivative of Fourier term is = Fourier term of derivative}:] It suffices to prove the two identities under the local Fourier-Series i.e. the left hand side of \eqref{equ local =global expression Fourier series}. We only need to work near each Reeb orbit. 

The identity for $L_{\xi}$ holds because $\xi=\frac{\partial}{\partial \theta_{\beta}}$ in $U_{\beta,\mathbb{S}^{5}}$, and the usual Fourier Series in $\theta_{\beta}$ can be differentiated term by term with respect to $\theta_{\beta}$.

To prove the identity for $\nabla^{\star}\nabla$, for any $[Z]\in \mathbb{P}^{2}$,  we need a transverse geodesic frame $[v_{i}=\frac{\partial}{\partial x_{i}}-\eta(\frac{\partial}{\partial x_{i}})\xi,i=1,2,3,4]$ near the Reeb orbit $\pi^{-1}_{5,4}[Z]$. Because $\xi[\eta (\frac{\partial}{\partial x_{i}})]=0$ i.e. $\eta (\frac{\partial}{\partial x_{i}})$ is independent of $\theta_{\beta}$, and that the connection is also pullback from $\mathbb{P}^{2}$, for each $i$, we find 
$$(\nabla_{v_{i}}\nu)_{-k}= (\nabla_{[\frac{\partial}{\partial x_{i}}-\eta (\frac{\partial}{\partial x_{i}})\xi]}\nu)_{-k}=\nabla_{[\frac{\partial}{\partial x_{i}}-\eta (\frac{\partial}{\partial x_{i}})\xi]}\nu_{k}=\nabla_{v_{i}}\nu_{k}\ \textrm{in the domain of}\ v_{i}.$$

Because $\nabla^{\star}\nabla \nu=\nabla_{v_{i}} \nabla_{v_{i}}\nu$ on the Reeb orbit $\pi^{-1}_{5,4}[Z]$, in view of Remark \ref{rmk localization of F series}, 
\begin{eqnarray*}& &(\nabla^{\star}\nabla \nu)_{-k}|_{\pi^{-1}_{5,4}|_{[Z]}}= (\nabla^{\star}\nabla \nu|_{\pi^{-1}_{5,4}[Z]})_{-k}= (\nabla_{v_{i}}\nabla_{v_{i}}\nu |_{\pi^{-1}_{5,4}[Z]})_{-k}
\\&=& (\nabla_{v_{i}} \nabla_{v_{i}}\nu)_{-k} |_{\pi^{-1}_{5,4}[Z]}=[\nabla_{v_{i}}\nabla_{v_{i}}(\nu)_{-k}] |_{\pi^{-1}_{5,4}[Z]}
\\&=& [\nabla^{\star}\nabla (\nu)_{-k}] |_{\pi^{-1}_{5,4}[Z]}. 
\end{eqnarray*}

The proof is complete.\end{proof}

\section{Some algebro-geometric calculations\label{Appendix Some algebro-geometric calculations}}
Let $\omega_{O(1)}$ be $\frac{\sqrt{-1}}{2\pi}$ times the curvature form of the standard metric on $O(1)\rightarrow \mathbb{P}^{2}$. Then $\omega_{O(1)}$ represents $c_{1}[O(1)]$, and  $\frac{d\eta}{2}=\pi \omega_{O(1)}$ (cf. \cite[page 142 and 30]{GH}, watch out the difference of our scaling from the one therein). Throughout this article, we call $\pi \omega_{O(1)}$ ($\frac{d\eta}{2}$) the Fubini-Study metric, and denote it by $\omega_{FS}$. The same applies to $\mathbb{P}^{n}$ as well (still let $\eta\triangleq d^{c}\log r$, $r$ is the distance to the origin in $\mathbb{C}^{n+1}$). 
\begin{proof}[\textbf{Proof of Lemma \ref{lem h1}}:] It suffices to show that 
\begin{equation}\label{equ Euler Characteristic}\chi[\mathbb{P}^{2},\ (EndE)(k)]=\frac{r^{2}k^{2}}{2}+r^{2}+\frac{3kr^{2}}{2}-2rc_{2}(E)+(r-1)c^{2}_{1}(E).\end{equation}

Because $K_{\mathbb{P}^{2}}=O(-3)$, when $k\geq 0$, Serre duality   says that 
\begin{equation}
h^{2}[\mathbb{P}^{2},\ (EndE)(k)]=h^{0}[\mathbb{P}^{2},\ (EndE)(-k-3)].
\end{equation} 
Then the desired identity \eqref{equ 0 lem h1} follows from the formula for Euler Characteristic \eqref{equ Euler Characteristic}. 

We go on to prove the corollary \eqref{equ 1 lem h1} using \eqref{equ 0 lem h1}. Since $E$ is Hermitian Einstein, then $E^{\star}$ is also Hermitian Einstein (see \cite[V, Proposition 7.7]{Kobayashi}). Combining \cite[IV, Proposition 1.4]{Kobayashi}), we find that $(End E)(-k-3)$ is Hermitian Einstein i.e. there exists a metric such that the curvature $F$ satisfies 
\begin{equation}
\frac{\sqrt{-1}}{2\pi} \omega^{i\bar{j}}_{O(1)}F_{i\bar{j}}=\mu Id_{E}\ \  \textrm{for some real constant}\ \mu.
\end{equation} 
When $k\geq -2$, 
$$\mu=\frac{deg [ (End E)(-k-3)]}{r^{2} Vol(\mathbb{P}^{2})}<0. $$
This means the mean curvature form defined in \cite[IV, below (1.3)]{Kobayashi} is negative definite, hence the vanishing theorem
\cite[III, Theorem 1.9]{Kobayashi} implies that 
\begin{equation}\label{equ h2}
h^{0}[\mathbb{P}^{2},\ (EndE)(-k-3)]=0.\ \ \textrm{Therefore}\ \ h^{2}[\mathbb{P}^{2},\ (EndE)(k)]=0. 
\end{equation}
Then \eqref{equ Euler Characteristic} and \eqref{equ h2} imply the desired identity $h^{1}[\mathbb{P}^{2},\ (EndE)(k)]=c_{2}(EndE)$. Furthermore, the usual Chern number inequality says $c_{2}(EndE)\geq 0$. Because $\mathbb{P}^{2}$ is simply connected and $rankE\geq 2$, if $c_{2}(EndE)=0$, \cite[Theorem 8.1]{UY} says that $E$ can not be simple, which contradicts stability of $E$. Thus $c_{2}(EndE)>0$.

Now we prove the Euler characteristic formula \eqref{equ Euler Characteristic}. Since $c_{1}(EndE)=0$, by \cite[II, (1.10)]{Kobayashi}, we compute\begin{eqnarray}& & ch[\mathbb{P}^{2},(EndE)(k)]= ch[\mathbb{P}^{2},O(k)]\cdot ch[\mathbb{P}^{2},End E]
\\&=&\{1+c_{1}[O(k)]+\frac{c^{2}_{1}[O(k)]}{2}\} \{r^{2}+c_{1}(End E)+\frac{1}{2}[c^{2}_{1}(End E)]-2c_{2}(End E)]\}\nonumber.
\\&=& r^{2}+kr^{2}[\omega_{O(1)}]-c_{2}(End E)\nonumber
+\frac{r^{2}k^{2}}{2}[\omega_{O(1)}]^{2}.
\nonumber \end{eqnarray}
The well known formula for Todd class states (for example, see \cite[page 288]{Kobayashi}):  \begin{equation} Td(\mathbb{P}^{2})=1+\frac{3[\omega_{O(1)}]}{2}+[\omega_{O(1)}]^{2}. \end{equation}
We compute
\begin{eqnarray}& & Td(\mathbb{P}^{2})\cdot ch[ (End E)(k)]\label{equ Td times ch}\nonumber
\\&=&r^{2}+(\frac{3r^{2}}{2}+kr^{2})\omega_{O(1)}-c_{2}(End E)
+[\frac{r^{2}k^{2}}{2}+r^{2}+\frac{3kr^{2}}{2}][\omega_{O(1)}]^{2}.
\end{eqnarray}
In conjunction with the remark on $\omega_{O(1)}$ above the underlying proof, the following holds. 
\begin{equation}
[\omega_{O(1)}]=c_{1}[O(1)],\ \textrm{hence}\ \int_{\mathbb{P}^{2}}[\omega_{O(1)}]^{2}=\int_{\mathbb{P}^{2}}\{c_{1}[O(1)]\}^{2}=1.
\end{equation}

Using Hirzebruch Riemann-Roch theorem, we integrate \eqref{equ Td times ch} to obtain
\begin{eqnarray}
& &\chi[\mathbb{P}^{2},\ (EndE)(k)]=\int_{\mathbb{P}^{2}}Td(\mathbb{P}^{2})\cdot ch[\mathbb{P}^{2},(End E)(k)]\nonumber
\\&=&\frac{r^{2}k^{2}}{2}+r^{2}+\frac{3kr^{2}}{2}-c_{2}(End E).
\end{eqnarray}

The proof of \eqref{equ Euler Characteristic} is complete. 
\end{proof}

\begin{proof}[\textbf{Proof of Lemma \ref{lem h1 EndTP2}}:] We only prove the formula for $h^{1}[\mathbb{P}^{2},\ (EndT^{\prime}\mathbb{P}^{2})(l)]$, the formula for $h^{0}[\mathbb{P}^{2},\ \{End_{0}(T^{\prime}\mathbb{P}^{2})\}(l)]$ thereupon follows by Riemann-Roch (see Lemma \ref{lem h1}).

 On $\mathbb{P}^{2}$, we tensor the Euler-Sequence $$0\rightarrow O \rightarrow O^{\oplus 3}(1)\rightarrow T^{\prime}\mathbb{P}^{2}\rightarrow 0$$ by the sheaf $\Omega^{1}(l)$, the local freeness of $\Omega^{1}(l)$ yields the exactness of the following. 
\begin{equation}0\rightarrow \Omega^{1}(l) \rightarrow [\Omega^{1}(l+1)]^{\oplus 3}\rightarrow (EndT^{\prime}\mathbb{P}^{2})(l)\rightarrow 0.
\end{equation}
Hence we have the following exact sequence of cohomologies
\begin{equation}\label{equ seq of cohomology}...\rightarrow H^{1}[\mathbb{P}^{2}, \Omega^{1}(l+1)]^{\oplus 3}] \rightarrow H^{1}[\mathbb{P}^{2},(EndT^{\prime}\mathbb{P}^{2})(l)]\rightarrow H^{2}[\mathbb{P}^{2}, \Omega^{1}(l)]\rightarrow ...
\end{equation}

By Bott formula of sheaf cohomology on complex projective spaces (see \cite[Section 1.1]{Okonek}), when $l\geq 0$, both $H^{1}[\mathbb{P}^{2}, \Omega^{1}(l+1)]^{\oplus 3}]$ and $H^{2}[\mathbb{P}^{2}, \Omega^{1}(l)]$ vanish.  Then $H^{1}[\mathbb{P}^{2},(EndT^{\prime}\mathbb{P}^{2})(l)]$ vanishes if $l\geq 0$.

When $l=-1$, 
$$H^{1}[\mathbb{P}^{2}, (\Omega^{1})^{\oplus 3}]= \{ H^{1}[\mathbb{P}^{2}, \Omega^{1}]\}^{\oplus 3}=\mathbb{C}^{3},\ H^{2}[\mathbb{P}^{2}, \Omega^{1}(-1)]=0.$$
Thus $H^{1}[\mathbb{P}^{2},(EndT^{\prime}\mathbb{P}^{2})(-1)]=\mathbb{C}^{3}$. By Serre-duality, we find $$H^{1}[\mathbb{P}^{2},(EndT^{\prime}\mathbb{P}^{2})(-2)]=\mathbb{C}^{3},\ \textrm{and}\ H^{1}[\mathbb{P}^{2},(EndT^{\prime}\mathbb{P}^{2})(l)]=0\ \textrm{if}\ l\leq -3.$$ 
\end{proof}
\section{K\"ahler identity for vector bundles}
The usual K\"ahler identity says that on a K\"ahler manifold, the Laplace-Beltrami operator (on functions) is twice of the $\bar{\partial}-$Laplacian. 
The  Lemma below is a straight-forward generalization to bundle case.  Though we do not know whether it is stated explicitly in literature,
the proof is completely routine. Please see a related calculation in \cite[III.1]{Kobayashi}. 
\begin{lem}\label{lem Kahler identity}Let $\Xi$ be a holomorphic Hermitian vector bundle over a K\"ahler manifold $(X,\omega)$. Let $A$ denote the Chern connection. Then 
\begin{equation}\label{equ general lem Kahler identity for vector bundles} \nabla_{A}^{\star}\nabla_{A}=2\partial^{\star}_{A}\bar{\partial}_{A}+2\pi\cdot \frac{\sqrt{-1}}{2\pi}F_{A}\lrcorner \omega. 
\end{equation}

Consequently, let $(E,h,A)$ be a Hermitian Yang-Mills triple on $\mathbb{P}^{n}$. In view of the convention for the K\"ahler metric in the first paragraph of Appendix \ref{Appendix Some algebro-geometric calculations} (above the proof of Lemma \ref{lem h1}), we consider the Fubini-Study metric $\omega_{FS}$.
On the twisted endomorphism bundle $(EndE)(l)$, under 
the tensor product  of $A$ and the standard connection on $O(l)$ (the twisted connection), suppressing the subscripts for the connection as usual, we have 
\begin{equation}\label{equ Pn lem Kahler identity for vector bundles}\nabla^{\star}\nabla=2\partial^{\star}\bar{\partial}+2nl\cdot Id. 
\end{equation}
In particular, when $n=2$, 
\begin{equation}\label{equ lem Kahler identity for vector bundles}\nabla^{\star}\nabla=2\partial^{\star}\bar{\partial}+4l\cdot Id. 
\end{equation}
\end{lem}
\begin{proof}[Proof of Lemma \ref{lem Kahler identity}:] At an arbitrary point $p\in X$, let $(z_{j},\ j=1,.., n)$ be a K\"ahler geodesic coordinate for the metric $\omega$. By definition, we have for any section $\varphi$ of $\Xi$ that 
\begin{equation}\nabla^{\star}\nabla \varphi=-2\Sigma_{j}(\varphi_{j\bar{j}}+\varphi_{\bar{j}j})=-4\Sigma_{j}\varphi_{\bar{j}j}+2\Sigma_{j}F_{A,j\bar{j}}\cdot \varphi=2\partial^{\star}_{A}\bar{\partial}_{A}\varphi+2\Sigma_{j}F_{A,j\bar{j}}\cdot \varphi\ \textrm{at}\ p.
\end{equation}
Please compare it to the usual K\"ahler identity in \cite[Chap 0.7, page 106]{GH}. To complete the proof of \eqref{equ general lem Kahler identity for vector bundles}, it suffices to observe that $\Sigma_{j}F_{A,j\bar{j}}=\frac{\sqrt{-1}}{2}F_{A}\lrcorner \omega$.

To prove \eqref{equ Pn lem Kahler identity for vector bundles}, based on  \eqref{equ general lem Kahler identity for vector bundles}, we contract the following by $\omega_{FS}$. \begin{equation}F_{A}=[F_{E},\cdot]\otimes Id_{O(l)}+Id_{E}\otimes F_{O(l)}.
\end{equation}
The Hermitian Yang-Mills condition says that $[F_{E}\lrcorner \omega_{FS},\cdot]$ acts by $0-$endomorphism on $EndE$, using  $c_{1}[O(l)]=[\omega_{O(1)}]$,  the following holds as endomorphisms on $(End E)(l)$.
\begin{eqnarray*}& &\frac{\sqrt{-1}}{2\pi}F_{A}\lrcorner \omega_{FS}= \frac{\sqrt{-1}}{2\pi}Id_{E}\otimes (F_{O(l)}\lrcorner \omega_{FS})=nId(\frac{\int_{\mathbb{P}^{n}}c_{1}[O(l)]\wedge \omega_{FS}^{n-1}}{\int_{\mathbb{P}^{n}}\omega_{FS}^{n}})
\\&=&\frac{(nl)Id}{\pi}(\frac{\int_{\mathbb{P}^{n}} \omega_{O(1)}^{n}}{\int_{\mathbb{P}^{n}}\omega_{O(1)}^{n}})
\\&=& \frac{(nl)Id}{\pi}.
\end{eqnarray*}
We should  notice that the $\pi$ factor in $\omega_{FS}=\pi\omega_{O(1)}$ produces the ``$\pi$" in the denominator of the last line above. The proof of \eqref{equ Pn lem Kahler identity for vector bundles} is complete.
\end{proof}

The above K\"ahler identity relates the space of holomorphic sections to a certain eigenspace of the rough Laplacian.

\begin{lem}\label{lem a holomorphic section is an eigensection of the rough Laplacian}(A holomorphic section is an eigensection of the rough Laplacian) Let $(E,h,A_{O})$ be a Hermitian Yang-Mills triple on $\mathbb{P}^{2}$. For any nonnegative integer $l$, 
\begin{equation}\label{equ the actual equality of vector speces}\mathbb{E}_{4l}\nabla^{\star}\nabla|_{(End_{0}E)(l)}=H^{0}[\mathbb{P}^{2}, (End_{0}E)(l)].
\end{equation}
Moreover, the isomorphism ``$=$" above is an actual equality: a holomorphic section of $(End_{0}E)(l)$ is an eigensection of $\nabla^{\star}\nabla|_{(End_{0}E)(l)}$ with respect to the eigenvalue $4l$, and vice versa. 
\end{lem}
The proof is straight-forward  by formula \eqref{equ lem Kahler identity for vector bundles}.
\section{Calculations on a Killing reductive homogeneous space\label{Appendix good frame}} 
\begin{proof}[\textbf{Proof of Lemma \ref{lem geodesic frame on KRHS}}:] For any $V,X,Y\in \mathfrak{g}$, the usual Koszul formula \cite[page 25]{Petersen} says 
\begin{eqnarray}\label{equ Koszul}2\langle \nabla_{V^{\star}}X^{\star},\ Y^{\star}\rangle&= &V^{\star}\langle X^{\star},\ Y^{\star}\rangle-Y^{\star}\langle V^{\star},\ X^{\star}\rangle+X^{\star}\langle Y^{\star},\ V^{\star}\rangle
\\& &+\langle [V^{\star}, X^{\star}],\ Y^{\star}\rangle-\langle [X^{\star}, Y^{\star}], V^{\star}\rangle+\langle [Y^{\star}, V^{\star}], X^{\star}\rangle.\nonumber
\end{eqnarray}
Because $V^{\star},\ X^{\star},\ Y^{\star}$ are Killing vector fields, we find 
\begin{eqnarray*}& & V^{\star}\langle X^{\star},\ Y^{\star}\rangle =\langle [V^{\star}, X^{\star}],\ Y^{\star}\rangle+\langle  X^{\star},\ [V^{\star},Y^{\star}]\rangle,
\\& & Y^{\star}\langle V^{\star},\ X^{\star}\rangle =\langle [Y^{\star}, V^{\star}],\ X^{\star}\rangle+\langle  V^{\star},\ [Y^{\star},X^{\star}]\rangle,
\\& & X^{\star}\langle Y^{\star},\ V^{\star}\rangle =\langle [X^{\star}, Y^{\star}],\ V^{\star}\rangle+\langle  Y^{\star},\ [X^{\star},V^{\star}]\rangle.
\end{eqnarray*}
Plugging  the above into  \eqref{equ Koszul}, we find 
\begin{equation}\label{equ -1 proof of lem geodesic frame}2\langle \nabla_{V^{\star}}X^{\star},\ Y^{\star}\rangle=\langle [V^{\star}, X^{\star}],\ Y^{\star}\rangle-\{\langle [X^{\star}, [Y^{\star},\ V^{\star}]\rangle+\langle V^{\star},\ [Y^{\star}, X^{\star}]\rangle \}.
\end{equation}

Next, for any $V,X,Y\in m$, we  show that the condition of Killing homogeneous space implies 
\begin{equation}\label{equ 0 proof of lem geodesic frame}\langle X^{\star}, [Y^{\star},\ V^{\star}]\rangle+\langle V^{\star},\ [Y^{\star}, X^{\star}]\rangle =0\ \textrm{at}\ eK. \end{equation}
\cite[Proposition 2.1]{Koda} says that $[X^{\star},Y^{\star}]=-[X,Y]^{\star}$ at $eK$ for any $X,\ Y\in \mathfrak{g}$. Then at $eK$,
\begin{eqnarray}\label{equ 1 proof of lem geodesic frame} & &\langle X^{\star}, [Y^{\star},\ V^{\star}]\rangle+\langle V^{\star},\ [Y^{\star}, X^{\star}]\rangle=-\langle X^{\star}, [Y,\ V]^{\star}\rangle-\langle V^{\star},\ [Y, X]^{\star}\rangle \nonumber
\\& =&-\langle X, [[Y,\ V]]_{m}\rangle_{m}-\langle V,\ [[Y, X]]_{m}\rangle_{m}\nonumber
\\& =&-\langle X, [Y,\ V]\rangle_{\mathfrak{g}}-\langle V,\ [Y, X]\rangle_{\mathfrak{g}}\ (\textrm{because}\ X,\ Y,\ V\in m,\ \textrm{and}\ m\perp \mathfrak{k})\nonumber
\\& =&0\ \ \ \ \ \ \ (\textrm{because}\ \langle,\rangle_{\mathfrak{g}}\ \textrm{is a scalar multiple of the Killing form}).
 \end{eqnarray}
In row $2$ of \eqref{equ 1 proof of lem geodesic frame}, the inner bracket $[Y,V]$ means the Lie bracket, while the outer means the projection to $m$ according to the reductive splitting. The identity \eqref{equ 0 proof of lem geodesic frame} is proved. 
 
For any $V,X\in m$, plugging  \eqref{equ 0 proof of lem geodesic frame} back into \eqref{equ -1 proof of lem geodesic frame}, because $Y\in m$ is also arbitrary, we find
\begin{equation}\nabla_{V^{\star}}X^{\star}=\frac{1}{2} [V^{\star}, X^{\star}]\ \textrm{at}\ eK.
\end{equation}
Therefore, for any $V\in m$, $\nabla_{V^{\star}}V^{\star}=0$ at $eK$.  

Because $g$ acts as an isometry (thus it preserves the Levi-Civita connection),  equation \eqref{equ geodesic frame on KRHS} holds at $gK$.
\end{proof}
\section{The standard connection on $O(l)\rightarrow \mathbb{P}^{2}$: proof of Lemma \ref{lem connection of O-1}\label{Appendix the standard connection on O(l)}}
To prove Lemma \ref{lem connection of O-1},  we need the $K-$invariant function corresponding to the local defining section of $O(-1)$.

We recall \eqref{equ the pi for Pn} for the natural map $\pi: SU(3)\rightarrow \mathbb{P}^{2}$.
\begin{lem}\label{lem defining section of O-1 and alpha}In $U_{0,\mathbb{P}^{2}}=\{[Z_{0},Z_{1},Z_{2}]\in \mathbb{P}^{2}|Z_{0}\neq 0\}$, the defining section $(1,u_{1},u_{2})$ of $O(-1)$ corresponds to the $Span{[1,0,0]}-$valued function $\alpha=(\frac{1}{g^{1}_{1}},0,0)$ on $SU(3)$, where $g_{11}$ is the $(1,1)-$entry of $g\in SU(3)$.  This means $$\left[\begin{array}{c}1 \\ u_{1} \\ u_{2}\end{array}\right]([g])=(g, \alpha)\ \textrm{for all}\ g\in \pi^{-1}U_{0,\mathbb{P}^{2}}.$$\end{lem}
\begin{rmk} $\alpha$ is obviously  $S[U(1)\times U(2)]-$invariant i.e. $\alpha(gk)=k^{-1}\alpha(g)$ for any\\ $k\in S[U(1)\times U(2)]$. Moreover, $U_{0,\mathbb{P}^{2}}$ is invariant under the action of $S[U(1)\times U(2)]$.
\end{rmk}
\begin{proof}[Proof of Lemma \ref{lem defining section of O-1 and alpha}:]  For any $g\in SU(3)$, it suffices to compute at $g\left[\begin{array}{c}1 \\ 0 \\0\end{array}\right]=\left[\begin{array}{c}g^{1}_{1} \\ g^{2}_{1} \\ g^{3}_{1}\end{array}\right]$ that  
$$(g,\alpha)=g\cdot \left[\begin{array}{c}\frac{1}{g^{1}_{1}} \\0 \\ 0\end{array}\right]=\left[\begin{array}{ccc}g^{1}_{1} & g^{1}_{2} & g^{1}_{3}  \\ g^{2}_{1} & g^{2}_{2} & g^{2}_{3}\\ g^{3}_{1} & g^{3}_{2} & g^{3}_{3} \end{array}\right]\ \left[\begin{array}{c}\frac{1}{g^{1}_{1}} \\0 \\ 0\end{array}\right]=\left[\begin{array}{c}1 \\ \frac{g^{2}_{1}}{g^{1}_{1}} \\ \frac{g^{3}_{1}}{g^{1}_{1}}\end{array}\right]\triangleq \left[\begin{array}{c}1 \\ u_{1} \\ u_{2}\end{array}\right].$$ 
\end{proof}

\begin{proof}[\textbf{Proof of Lemma} \ref{lem connection of O-1}:] On $O(-1)=SU(3)\times_{S[U(1)\times U(n)],\rho_{1}} \mathbb{C}$, the induced connection  $d_{induced}$ is $SU(3)-$invariant, so is the standard connection $d_{Chern}$. It suffices to verify that they coincide at the base point $o\triangleq\left[\begin{array}{c}1 \\ 0 \\0\end{array}\right]\in \mathbb{P}^{2}$ under the trivialization $s_{0}=\left[\begin{array}{c}1 \\ u_{1} \\u_{2}\end{array}\right]$. 

The  standard connection on $O(-1)$ yields that $d_{Chern}s_{0}=(\partial \log \phi_{0})s_{0}$. Consequently, by the definition of the K\"ahler potential $\phi_{0}$ above \eqref{equ norm of Z0}, we find   \begin{equation}
d_{Chern}s_{0}=0\ \textrm{at}\ o.
\end{equation}
All the elements in $m_{\mathbb{P}^{2}}$ have vanishing $(1,1)-$entry (see \eqref{equ basis of su3 first 4}). Therefore, for any\\ $X\in m_{\mathbb{P}^{2}}$, $$\alpha(e^{tX})=\left[\begin{array}{c}1+O(t^{2}) \\ 0 \\ 0\end{array}\right].$$ This implies  $X(\alpha)=0$ at $e\in SU(3)$. Lemma \ref{lem defining section of O-1 and alpha} means that $s_{0}=(g,\alpha)$.  Because $m_{\mathbb{P}^{2}}$ is horizontal, we find
\begin{equation}
d_{induced,X}s_{0}=0\ \textrm{at}\ o\ \textrm{for any}\ X\in m_{\mathbb{P}^{2}}.
\end{equation}
Thus,   the induced connection coincides with the Chern connection at the base point. 
\end{proof}

\section{The horizontal distribution of the Fubini-Study connection on the holomorphic tangent bundle: proof of Lemma \ref{lem induced connection = Fubini Study connection}\label{Appendix The horizontal distribution of the Fubini-Study connection on the holomorphic tangent bundle}}
\begin{Def} For any $X\in m_{\mathbb{P}^{2}}$, let $X^{\star,\mathbb{C}}$ be the projection of the real vector field $X^{\star}$ to $T^{\prime}\mathbb{P}^{2}$ (see the material from \eqref{equ J on mP2} to \eqref{equ s} for the projection, and see \eqref{equ Xstar} for the definition of $X^{\star}$).
\end{Def}
To prove Lemma \ref{lem induced connection = Fubini Study connection}, we need another form for the vector fields $ X_{1}^{\star,\mathbb{C}},\  Y_{1}^{\star,\mathbb{C}},\  X_{3}^{\star,\mathbb{C}},\  Y_{3}^{\star,\mathbb{C}}$.
\begin{formula}\label{formula vector fields generated by mCP2}In view of the basis \eqref{equ basis of su3 first 4} of $m_{\mathbb{P}^{2}}$,\begin{eqnarray}& & X_{1}^{\star,\mathbb{C}}=\pi_{5,4,\star}(Z_{1}\frac{\partial}{\partial Z_{0}}-Z_{0}\frac{\partial}{\partial Z_{1}}).\ \textrm{In}\ U_{0,\mathbb{P}^{2}},\ X_{1}^{\star,\mathbb{C}}=-(1+u^{2}_{1})\frac{\partial}{\partial u_{1}}-u_{1}u_{2}\frac{\partial}{\partial u_{2}}.\ \nonumber
\\& & Y_{1}^{\star,\mathbb{C}}=\sqrt{-1}\pi_{5,4,\star}(Z_{1}\frac{\partial}{\partial Z_{0}}+Z_{0}\frac{\partial}{\partial Z_{1}}).\ \textrm{In}\ U_{0,\mathbb{P}^{2}},\ Y_{1}^{\star,\mathbb{C}}=\sqrt{-1}(1-u^{2}_{1})\frac{\partial}{\partial u_{1}}-\sqrt{-1}u_{1}u_{2}\frac{\partial}{\partial u_{2}}.\nonumber
\\& & X_{3}^{\star,\mathbb{C}}=\pi_{5,4,\star}(Z_{2}\frac{\partial}{\partial Z_{0}}-Z_{0}\frac{\partial}{\partial Z_{2}}).\ \textrm{In}\ U_{0,\mathbb{P}^{2}},\  X_{3}^{\star,\mathbb{C}}= -u_{1}u_{2}\frac{\partial}{\partial u_{1}}-(1+u^{2}_{2})\frac{\partial}{\partial u_{2}}.\nonumber 
\\& &Y_{3}^{\star,\mathbb{C}}=\sqrt{-1}\pi_{5,4,\star}(Z_{2}\frac{\partial}{\partial Z_{0}}+Z_{0}\frac{\partial}{\partial Z_{2}}).\ \textrm{In}\ U_{0,\mathbb{P}^{2}},\ Y_{3}^{\star,\mathbb{C}}=-\sqrt{-1}u_{1}u_{2}\frac{\partial}{\partial u_{1}}+\sqrt{-1}(1-u^{2}_{2})\frac{\partial}{\partial u_{2}}.\nonumber
\end{eqnarray}

Consequently, $s^{\star}_{1}=-\pi_{5,4,\star}(Z_{0}\frac{\partial}{\partial Z_{1}}),\ s^{\star}_{2}=-\pi_{5,4,\star}(Z_{0}\frac{\partial}{\partial Z_{2}})$. In $U_{0,\mathbb{P}^{2}}$, 
$$s^{\star}_{1}=-\frac{\partial}{\partial u_{1}},\  \ \ s^{\star}_{2}=-\frac{\partial}{\partial u_{2}}. $$
\end{formula}
\begin{proof}[Proof of Formula \ref{formula vector fields generated by mCP2}:] We verify that $e^{tX_{1}}=\left[\begin{array}{ccc} \cos t  & \sin t &   0  \\  -\sin t & \cos t & 0  \\ 0 & 0 & 1 \end{array}\right] $ . Thus, when $t$ is sufficiently small with respect to $u_{1}$,  the following holds on $\mathbb{P}^{2}$.
\begin{equation}\label{equ 0 proof formula vector fields generated by mCP2}
e^{tX_{1}}\left[\begin{array}{c}1   \\  u_{1}  \\ u_{2} \end{array}\right] =\left[\begin{array}{c}\cos t+(\sin t)u_{1}   \\ -\sin t+(\cos t) u_{1}  \\ u_{2} \end{array}\right] =\left[\begin{array}{c}1  \\ \frac{-\sin t+(\cos t) u_{1}}{\cos t+(\sin t)u_{1} }  \\ \frac{u_{2}}{\cos t+(\sin t)u_{1} } \end{array}\right].
\end{equation}
Then the identity \begin{eqnarray}\label{equ 1 proof formula vector fields generated by mCP2}
& &\nonumber X^{\star,\mathbb{C}}_{1}=\frac{d}{dt}|_{t=0}e^{tX_{1}}\left[\begin{array}{c}1   \\  u_{1}  \\ u_{2} \end{array}\right] =\left[\begin{array}{c}0   \\ -(1+ u^{2}_{1})  \\ -u_{1}u_{2} \end{array}\right] =-(1+u^{2}_{1})\frac{\partial}{\partial u_{1}}-u_{1}u_{2}\frac{\partial}{\partial u_{2}}
\\&=&-\pi_{5,4,\star}(Z_{0}\frac{\partial}{\partial Z_{1}})+\pi_{5,4,\star}(Z_{1}\frac{\partial}{\partial Z_{0}})\end{eqnarray}
holds in $U_{0,\mathbb{P}^{2}}$. While the vector ``$\left[\begin{array}{c}1   \\  u_{1}  \\ u_{2} \end{array}\right]$" above means a point in ``$U_{0,\mathbb{P}^{2}}\subset \mathbb{P}^{2}$, the vector ``$\left[\begin{array}{c}0   \\ -(1+ u^{2}_{1})  \\ -u_{1}u_{2} \end{array}\right]$" above means a $(1,0)$ tangent vector (at the point). 

   By continuity of both $X^{\star,\mathbb{C}}_{1}$ and $-\pi_{5,4,\star}(Z_{0}\frac{\partial}{\partial Z_{1}})+\pi_{5,4,\star}(Z_{1}\frac{\partial}{\partial Z_{0}})$, they are identical everywhere on $\mathbb{P}^{2}$.

 Employing the following identities of matrix exponentials,
\begin{eqnarray}
& & e^{tY_{1}}=\left[\begin{array}{ccc} \cos t  & \sqrt{-1}\sin t &   0  \\  \sqrt{-1}\sin t & \cos t & 0  \\ 0 & 0 & 1 \end{array}\right],\ e^{tX_{3}}=\left[\begin{array}{ccc} \cos t  & 0 &    \sin t   \\ 0 & 1& 0  \\  -\sin t & 0 & \cos t \end{array}\right],\nonumber
\\& & e^{tY_{3}}=\left[\begin{array}{ccc} \cos t  &0 &   \sqrt{-1}\sin t   \\  0& 1& 0  \\ \sqrt{-1}\sin t  & 0 & \cos t  \end{array}\right],
\end{eqnarray}
similar computations as \eqref{equ 0 proof formula vector fields generated by mCP2} and \eqref{equ 1 proof formula vector fields generated by mCP2} show that in $U_{0,\mathbb{P}^{2}}$, 
 \begin{eqnarray}
& & Y^{\star,\mathbb{C}}_{1}=\sqrt{-1}(1-u^{2}_{1})\frac{\partial}{\partial u_{1}}-\sqrt{-1}u_{1}u_{2}\frac{\partial}{\partial u_{2}},\ X^{\star,\mathbb{C}}_{3}=-u_{1}u_{2}\frac{\partial}{\partial u_{1}}-(1+u^{2}_{2})\frac{\partial}{\partial u_{2}},\nonumber
\\& & Y^{\star,\mathbb{C}}_{3}=-\sqrt{-1}u_{1}u_{2}\frac{\partial}{\partial u_{1}}+\sqrt{-1}(1-u^{2}_{2})\frac{\partial}{\partial u_{2}}.
\end{eqnarray}
As below \eqref{equ 1 proof formula vector fields generated by mCP2}, the $3$ formulas respectively for $Y_{1}^{\star,\mathbb{C}},
X_{3}^{\star,\mathbb{C}}, Y_{3}^{\star,\mathbb{C}}$ follow by continuity. 
\end{proof}

\begin{proof}[\textbf{Proof of Lemma} \ref{lem induced connection = Fubini Study connection}:] Similarly to the proof of Lemma \ref{lem connection of O-1}, because both connections are left invariant, it suffices to show that they are identical at the base point $o$. 

The Fubini-Study co-variant derivatives of both $\frac{\partial}{\partial u_{1}}$ and $\frac{\partial}{\partial u_{2}}$ are $0$ at $o$. Using Formula \ref{formula vector fields generated by mCP2}, we find  \begin{equation}
\nabla^{FS}s_{1}^{\star}=\nabla^{FS}s_{2}^{\star}=0\ \ \textrm{at}\  o.
\end{equation}

In view of the correspondence in Lemma \ref{lem mtildex and xstar}, at $e\in SU(3)$, for any $X,Y\in m_{\mathbb{P}^{2}}$, we compute the ordinary derivative $$[Y\widetilde{m}_{\mathbb{P}^{2}}(X)](e)=-[[Y,X]]_{m_{\mathbb{P}^{2}}}.$$ On the right hand side of the above, the inner bracket is the Lie bracket, the outer one is the projection to $m_{\mathbb{P}^{2}}$.

 We straight-forwardly verify $[m_{\mathbb{P}^{2}},m_{\mathbb{P}^{2}}]\subseteq s[u(1)\times u(2)]$. Then 
$[Y\widetilde{m}_{\mathbb{P}^{2}}(X)](e)=0$. 
The correspondence  \eqref{equ lem mtildex and xstar} and Kobayashi-Nomizu formula \eqref{equ KN formula} again yields that $$(\nabla^{induced}_{Y}X^{\star})|_{o}=0.$$ On the complexification, this means for any $s\in m^{(1,0)}_{\mathbb{P}^{2}}$, $\nabla^{induced}s^{\star}=0$ at $o$.

Then $\nabla^{induced}$ coincides with $\nabla^{FS}$ at the base point $o$. By $SU(3)-$invariance,  they coincide everywhere on $\mathbb{P}^{2}$.
\end{proof}

\end{document}